\documentclass{article}
\usepackage{amsmath,amsfonts,amssymb,amsthm}
\usepackage{color}

\usepackage{hyperref}

\providecommand{\R}{\mathbb{R}}
\providecommand{\C}{\mathbb{C}}
\providecommand{\N}{\mathbb{N}}

\providecommand{\eps}{\varepsilon}
\providecommand{\om}{\omega}

\newcommand{\supp}{\operatorname{supp}}
\def\longrightharpoonup{\relbar\joinrel\rightharpoonup}

\renewcommand{\leq}{\leqslant}
\renewcommand{\geq}{\geqslant}
\renewcommand{\Re}{{\rm Re}}

\renewcommand{\div}{\operatorname{div}}
\newcommand{\curl}{\operatorname{curl}}

\newcommand{\dist}{\operatorname{dist}}
\newcommand{\Id}{\operatorname{Id}}
\newtheorem{Theorem}{Theorem}[section]

\newtheorem{Corollary}[Theorem]{Corollary}
\newtheorem{Proposition}[Theorem]{Proposition}
\newtheorem{Lemma}[Theorem]{Lemma}
\newtheorem{Remark}[Theorem]{Remark}
\numberwithin{equation}{section}

\begin{document}

\author{Olivier Glass\footnote{CEREMADE, UMR 7534,
Universit\'e Paris-Dauphine \& CNRS, 
Place du Mar\'echal de Lattre de Tassigny,
75775 Paris Cedex 16, France.
}, 
Christophe Lacave\footnote{Univ Paris Diderot, Sorbonne Paris Cit\'e, Institut de Math\'ematiques de Jussieu-Paris Rive Gauche, UMR 7586, CNRS, Sorbonne Universit\'es, UPMC Univ Paris 06, F-75013, Paris, France.
},
Franck Sueur\footnote{Institut de Math\'ematiques de Bordeaux, UMR CNRS 5251,
Universit\'e de Bordeaux, 351 cours
de la Lib\'eration, F33405 Talence Cedex, France. }
}
\date{\today}
\title{On the motion of a small light body immersed in a two dimensional incompressible perfect fluid with vorticity}
\maketitle

\begin{abstract}
In this paper we consider the motion of a rigid body immersed in  a two dimensional unbounded incompressible perfect fluid with vorticity. 
We prove that when the body shrinks to a massless  pointwise particle  with fixed circulation, the ``fluid+rigid body'' system  converges to  the vortex-wave system introduced by  Marchioro and Pulvirenti in \cite{MP}.
This extends both the paper \cite{GLS} where the case of a solid tending to a massive pointwise particle was tackled and the paper \cite{GMS} where the massless case was considered but in a bounded cavity filled with an irrotational fluid. 
\end{abstract}
\tableofcontents
%
%
\section{Introduction}
\label{Intro}
In this paper we consider the motion  of a  rigid body immersed in  a two dimensional  incompressible perfect fluid, when the size of the body converges to $0$. 
Initially the rigid body is assumed to occupy
$$
\mathcal{S}_0^\varepsilon := \varepsilon \mathcal{S}_0,
$$
where $\mathcal{S}_0$ is a simply connected smooth compact subset of $\R^2$
and  $\varepsilon \in (0,1)$.
 
The body moves rigidly so that at times $t$ it occupies a domain $\mathcal{S}^\varepsilon(t)$ which is isometric to $\mathcal{S}_0^\varepsilon$.
We denote 
$$
\mathcal{F}^\varepsilon (t) := \R^2  \setminus \mathcal{S}^\varepsilon(t)
$$
the domain occupied by the fluid  at  time $t$ starting from the initial domain
$$
\mathcal{F}^\varepsilon_{0}  := \R^2 \setminus {\mathcal{S}}^\varepsilon_{0} .
$$ 
The equations modelling the dynamics of the system then read :
\\ 
\\   \textbf{Fluid equations:}
\begin{align}
\displaystyle \frac{\partial u^\varepsilon }{\partial t}+(u^\varepsilon  \cdot\nabla)u^\varepsilon   + \nabla \pi^\varepsilon =0 && \text{for}\ t\in (0,\infty), \ x\in \mathcal{F}^\varepsilon (t), \label{Euler1}\\
\div u^\varepsilon   = 0 && \text{for}\ t\in [0,\infty), \ x\in \mathcal{F}^\varepsilon(t) , \label{Euler2} 
\end{align}
\\  
\\  \textbf{Solid equations:}
\begin{align}
m^\varepsilon (h^\varepsilon)'' (t) &=  \int_{\partial  \mathcal{S}^\varepsilon (t)} \pi^\varepsilon n \, ds &&\text{for}\ t\in (0,\infty),   \label{Solide1} \\ 
\mathcal{J}^\varepsilon (r^\varepsilon)' (t) &=  \int_{\partial  \mathcal{S}^\varepsilon (t)} (x-h^\varepsilon(t))^{\perp} \cdot \pi^\varepsilon n \, ds &&\text{for}\ t\in (0,\infty),   \label{Solide2} 
\end{align}
\\  
\\  \textbf{Boundary conditions:} 
\begin{align}
u^\varepsilon  \cdot n &=   \Big( (h^\varepsilon)'(t) + r^\varepsilon(t)(x-h^\varepsilon(t))^{\perp} \Big)  \cdot n && \text{for}\ t\in [0,\infty),  \  x\in \partial \mathcal{S}^\varepsilon  (t),   \label{Euler3} \\
\lim_{|x|\to \infty} |u^\varepsilon(t,x)| &=0& & \text{for}\ t\in [0,\infty),  
\end{align}
\\  
\\  \textbf{Initial data:}
\begin{gather}
u^\varepsilon |_{t= 0} = u^\varepsilon_0  \quad \text{for} \  x\in  \mathcal{F}^\varepsilon_0 ,  \label{Eulerci2} \\
h^\varepsilon (0)= 0 ,   \quad \ (h^\varepsilon)' (0)=  \ell_0^\varepsilon ,  \quad \theta^\varepsilon(0)=0 ,   \quad r^\varepsilon (0)=  r_0^\varepsilon .  \label{Solideci}
\end{gather}
Here $u^\varepsilon=(u_1^\varepsilon,u_2^\varepsilon)$ and $\pi^\varepsilon$ denote the velocity and pressure fields in the fluid,
$m^\varepsilon >0$ and $\mathcal{J}^\varepsilon >0$ denote respectively the mass and the momentum of inertia of the body.
The fluid is supposed to be  homogeneous, of density $1$ to simplify the notations.

When $x=(x_1,x_2)$ the notation $x^\perp $ stands for $x^\perp =( -x_2 , x_1 )$, 
$n$ denotes  the unit normal vector pointing outside the fluid,  $(h^\varepsilon )'(t)$
is the velocity of the center of mass  $h^\varepsilon (t)$ of the body and $r^\varepsilon(t)$ is the angular velocity. 
Indeed, since $\mathcal{S}^\varepsilon(t)$ is isometric to $\mathcal{S}_{0}^\varepsilon$ there exists an angle $\theta^\varepsilon(t)$ such that, with the notation of the rotation matrix
\begin{equation} \label{NotationRot2D}
R_{\theta^\varepsilon(t)}:=\begin{pmatrix} \cos \theta^\varepsilon (t) & -\sin \theta^\varepsilon(t) \\ \sin \theta^\varepsilon (t) & \cos \theta^\varepsilon(t)\end{pmatrix},
\end{equation}
one has
\begin{equation*}
\mathcal{S}^\varepsilon(t) := \{ h^\varepsilon(t)+ R_{\theta^\varepsilon(t)} x , \ x \in \mathcal{S}_{0}^\varepsilon \} .
\end{equation*}
Furthermore, this angle satisfies
\begin{equation*}
(\theta^\varepsilon)'(t)=r^\varepsilon(t) .
\end{equation*}
In this paper, we will systematically take the convention that the initial position of the center of mass is $0$, that is,
\begin{equation} \label{hen0}
h^{\varepsilon}(0)=0.
\end{equation}
\ \par
\ \par
Equations \eqref{Euler1} and \eqref{Euler2} are the incompressible Euler equations, the condition \eqref{Euler3} means that the boundary is impermeable and Equations \eqref{Solide1}-\eqref{Solide2} are the Newton's balance law for linear and angular momenta. \par 
\ \par
The Cauchy problem for System  \eqref{Euler1}-\eqref{Solideci}  is now well understood.
In particular we have the following result proven in \cite{GLS}, which is the equivalent for the fluid-body system of the celebrated result of Yudovich for the fluid alone \cite{Yudovich}. We recall that the space of log-Lipschitz functions on some domain $X \subset \R^{2}$ is the set of functions $f \in L^{\infty}(X)$ such that
\begin{equation*}
\| f \|_{\mathcal{LL}(X)} := \| f\|_{L^{\infty}(X)} + \sup_{x\not = y} \frac{|f(x)-f(y)|}{|(x-y)(1+ \ln^{-}|x-y|)|} < +\infty.
\end{equation*}
\begin{Theorem} \label{ThmYudo}
For any
$$
u_0^\varepsilon \in C^{0}(\overline{\mathcal{F}^{\varepsilon}_0};\R^{2}), \quad (\ell_0^\varepsilon ,r_0^\varepsilon) \in \R^2 \times \R ,
$$
such that:
\begin{gather*} 
\div u_0^\varepsilon =0 \text{ in } {\mathcal F}_0^\varepsilon \\
w_0^\varepsilon := \curl u_0^\varepsilon  \in L_c^{\infty}(\overline{\mathcal F}_0^\varepsilon), \\
u_0^\varepsilon   \cdot  n = (\ell_0^\varepsilon + r_0^\varepsilon x^{\perp})   \cdot  n \text{ on } \partial \mathcal{S}_0^\varepsilon , \\
\lim_{|x| \rightarrow +\infty} u_{0}^\varepsilon (x) =0,
\end{gather*}
there exists a unique solution $(h^\varepsilon,\theta^\varepsilon,u^\varepsilon)$ of \eqref{Euler1}--\eqref{Solideci} with   
\begin{gather*}
 (h^\varepsilon,\theta^\varepsilon) \in C^2 (\R^+; \R^2 \times \R), \quad 
u^\varepsilon \in L^{\infty}_{loc}(\R^+; \mathcal{LL}({\mathcal F}^\varepsilon (t))) \text{ and }  \\
\partial_{t} u^\varepsilon, \nabla p^\varepsilon \in L^{\infty}_{loc}(\R^{+}; L^{q}({\mathcal F}^\varepsilon(t)))
\text{  for any }
q \in (1,+\infty).
\end{gather*}
Moreover such a solution satisfies some conservation laws which will be recalled in Subsections~\ref{Subsec:COF} and \ref{Subsec:vor}. 
\end{Theorem}
There is a slight abuse of notations in $L^{\infty}_{loc}(\R^+;\mathcal{LL}({\mathcal F}^\varepsilon (t)))$ and $L^{\infty}_{loc}(\R^{+};L^{q}({\mathcal F}^\varepsilon(t)))$ since the domain ${\mathcal F}^\varepsilon (t)$ of the $x$-variable actually depends on $t$. By this we refer to functions defined for almost each $t$ as a function in the space $\mathcal{LL}({\mathcal F}^\varepsilon (t))$ (resp. $L^{q}({\mathcal F}^\varepsilon(t))$), and which can be extended to a function in $L^{\infty}_{loc}(\R^{+};\mathcal{LL}(\R^{2}))$ (resp. $L^{\infty}_{loc}(\R^{+};L^q(\R^{2}))$). \par
\ \par
A central fact to obtain Theorem~\ref{ThmYudo} is that in these solutions the fluid vorticity
\begin{equation*}
w^\varepsilon:= \curl  u^\varepsilon =\partial_1 u_2^\varepsilon- \partial_2 u_1^\varepsilon ,
\end{equation*}
satisfies
\begin{equation} \label{Euler4}
\displaystyle \frac{\partial w^\varepsilon }{\partial t}+\div (u^\varepsilon  w^\varepsilon) =0
\qquad \text{for}\ t\in (0,\infty), \ x\in \mathcal{F}^\varepsilon (t).
\end{equation}
Actually, the uniqueness part of Theorem~\ref{ThmYudo} can be established starting either from the velocity equation \eqref{Euler1} or from the vorticity equation \eqref{Euler4}. \par
In this paper we are interested in the asymptotic behavior of $(h^\eps, \theta^\eps, u^\eps, w^\eps)$ when $\varepsilon\to 0^+$.
This issue  depends on the behavior of the data with respect to $\eps$ which we now describe. \par
\ \par
\noindent
\textbf{Mass and momentum of inertia}.
In the paper \cite{GLS} we studied the case where the solid occupying the domain $\mathcal S^\varepsilon(t)$ is assumed to have a mass and a moment of inertia of the form $m^\varepsilon=  m^{1}  \text{ and } \mathcal J^\varepsilon=\varepsilon^2\mathcal J^{1}$,
where $m^{1}>0$  and  $\mathcal J^{1}>0$ are fixed, so that the solid tends to a   \textit{massive pointwise particle}.
The goal of this paper is to study the case where 
\begin{equation} \label{mass-inertie2}
m^\varepsilon= \varepsilon^{\alpha} \,  m^{1}  \text{ and } \mathcal J^\varepsilon= \varepsilon^{\alpha+2} \,  {\mathcal J}^{1},
\end{equation}
where $ \alpha > 0$ and  $m^{1}>0$ and ${\mathcal J}^{1}>0$ are fixed, so that the solid tends to a  \textit{massless pointwise particle}.
The particular case where $ \alpha =2$ corresponds to the case of a fixed solid density whereas 
the case tackled in   \cite{GLS}  corresponded to the case where  $ \alpha = 0$. 
The regime  \eqref{mass-inertie2} was  considered in  \cite{GMS}  in the irrotational case (and when the fluid occupies a bounded domain). The purpose of this paper is to extend it to the case where the vorticity $w^\varepsilon$ in the fluid does not vanish. \par
\ \par
\noindent
\textbf{Vorticity and circulation around the solid}.
We will consider an initial fluid vorticity  $w_0  \in L_c^{\infty}(\R^{2}\setminus\{0\})$ independent of $\varepsilon $
and an initial circulation
$$\gamma := \int_{\partial \mathcal{S}_0^\eps} u^\eps_0\cdot \tau\, ds $$
independent of $\varepsilon $ as well. The fact that we consider $0 \notin \supp w_0$ is connected to \eqref{hen0}: for $\varepsilon>0$ small enough, one has $\supp w_0 \cap \mathcal{S}_0^\varepsilon = \emptyset$. \par
\ \par
\noindent
\textbf{Initial solid velocity}.
We will consider an initial solid velocity $ (\ell_0^\eps , r_{0}^\eps ) $ independent of $\varepsilon $:
$$ (\ell_0^\eps , r_{0}^\eps ) = (\ell_0 , r_{0} )  \in \R^2 \times \R .$$
\ \par
\noindent
\textbf{Initial fluid velocity}.
The initial fluid velocity  $u_0^\eps$ is then defined as the unique log-Lipschitz solution of the div-curl type system:
\begin{equation} \label{UDI}
\left\{ \begin{array}{l}
\div u^\eps_0=0,\ \curl u^\eps_0=w_0^{\varepsilon} \text{ in } {\mathcal F}_{0}^{\varepsilon}, \\
u^\eps_0\cdot n= (\ell_0 + r_{0} x^{\perp}) \cdot n  \text{ on } \partial {\mathcal S}_{0}^{\varepsilon},  \\
\lim_{|x|\to \infty} |u^\eps_0(x)|=0, \ \int_{\partial \mathcal{S}_0^\eps} u^\eps_0\cdot \tau\, ds =\gamma,
\end{array} \right.
\end{equation}
where $w_{0}^{\varepsilon} := w_{0|{\mathcal F}_{0}^{\varepsilon}}$, hence, for $\varepsilon$ small enough (depending on $\dist(\supp w_{0}; 0)$ and the size of $\mathcal{S}_{0}$), we have
\begin{equation} \label{Defw0eps}
w_{0}^{\varepsilon} := w_{0}.
\end{equation}
\ \par
To state the main result, we will use the following notation for the Biot-Savart operator in $\R^{2}$:
\begin{equation*}
K_{\R^{2}}[w](t,x)= K_{\R^{2}}[w(t,\cdot)](x):=\frac{1}{2 \pi} \int_{ \R^{2}} \frac{(x-y)^{\perp}}{|x-y|^{2}} w (t,y) \, dy .
\end{equation*}
Now our goal in this paper is to prove the following theorem.
\begin{Theorem} \label{MR}
Let us be given $\gamma \in \R \setminus \{ 0 \}$, $(\ell_0,r_{0}) \in \R^3$,  $ w_0 $ in $L^\infty_c (\R^2\setminus \{0\})$.
For any $\eps \in (0,1]$, we associate $u^\eps_0$ by \eqref{UDI}-\eqref{Defw0eps} and consider $(h^{\varepsilon}, \theta^\varepsilon, u^\eps)$ the unique  solution of the system \eqref{Euler1}--\eqref{Solideci}. 
Then for any $T>0$, as $\varepsilon \rightarrow 0^{+}$,
\begin{itemize}
\item  $h^\eps $ converges to $h$ weakly-$\star$ in $W^{1,\infty} (0,T;\R^{2})$, 
\item $w^{\varepsilon}$ converges to $w$ in $C^{0} ([0,T]; L^{\infty}(\R^{2})-w\star)$,
\item  $u^{\varepsilon}$ converges to $\displaystyle \tilde{u} + \frac{\gamma}{2\pi} \frac{(x-h(t))^{\perp}}{|x-h(t)|^{2}}$ in $C^{0} ([0,T]; L^{q}_{loc} (\R^{2} ))$ for $q<2$, where $ \tilde{u}(t,x) $ is defined on $[0,T] \times \R^{2}$ by
\begin{equation*}
\tilde{u}(t,x) := K_{\R^{2}}[w(t,\cdot)](x).
\end{equation*}
\end{itemize}
Moreover, $(\tilde u, w, h)$ satisfies:
\\
\\ \textbf{Fluid equation:}
\begin{equation} \label{EulerPoint} 
\frac{\partial w }{\partial t}+ \div
\bigg( \left[ \tilde{u}+ \frac{\gamma}{2\pi} \frac{(x-h(t))^{\perp}}{| x-h(t)|^{2}} \right] w \bigg) = 0  
\ \text{ in } \ [0,T] \times \R^{2}, 
\end{equation}
\\ \textbf{Particle equation:} 
\begin{equation} \label{PointEuler}
h'(t) = \tilde{u}(t,h(t))  \ \text{ in } \ [0,T]  ,
\end{equation}
\\ \textbf{Initial conditions:}
\begin{equation} \label{EP0}
w |_{t= 0}=  w_0 ,\ h(0) = 0.
\end{equation}
\end{Theorem}
\ \par
The above convergences of $w^{\varepsilon}$ and $u^{\varepsilon}$ hold when we have extended these functions by $0$ inside the solid.%
\begin{Remark}
Equation  \eqref{EulerPoint} and the $w$-part of the initial data given in \eqref{EP0} 
hold in the sense that  for any test function $\psi\in C^\infty_c([0,T)\times\R^2)$ we have 
\begin{equation} \label{EqSolFaibleIntro}
\int_0^\infty \int_{\R^2} \psi_t  w \, \, dx\, dt
+\int_0^\infty \int_{\R^2} \nabla_x \psi \cdot \Big( \tilde{u}+ \frac{\gamma}{2\pi} \frac{(x-h(t))^{\perp}}{|x-h(t)|^{2}} \Big) w  \, dx\, dt
+\int_{\R^2} \psi(0,x) w_0 (x) \, dx=0 .
\end{equation}
\end{Remark}
The equations (\ref{EulerPoint})--(\ref{PointEuler})  describe the vortex-wave system introduced by Marchioro and Pulvirenti in  \cite{MP}.
Equation (\ref{EulerPoint}) describes the evolution of the vorticity of the fluid: $w$ is transported by a velocity obtained by the usual Biot-Savart law in the plane, but from a vorticity which is the sum of the fluid vorticity and of a point vortex placed at the (time-dependent) position $h(t)$ where the solid shrinks, with a strength equal to the circulation $\gamma$ around the body. The point vortex is transported only under the influence of the fluid vorticity \eqref{PointEuler}.\par
We recall from Marchioro-Pulvirenti \cite{MP} (Lagrangian formulation) and by Lacave-Miot \cite{CricriLylyne} (Eulerian formulation, i.e.  in our case \eqref{EqSolFaibleIntro}) that in the case of one point vortex $h(t)$ and of an initial vorticity $w_{0}\in L^\infty_{c}(\R^2 \setminus \{ 0 \})$, the vortex-wave system \eqref{EulerPoint}-\eqref{PointEuler} admits a unique solution such that $w\in L^\infty(\R^+,L^1\cap L^\infty(\R^2))$ and $h\in C(\R^+,\R^2)$. Moreover, such a solution has the following property: for any $T>0$, there exists $\rho_{T}>0$ such that
\begin{equation} \label{supportcontrol}
\supp w(t) \subset B(h(t),\rho_{T})\setminus B(h(t),1/\rho_{T}) \quad \forall t\in [0,T].
\end{equation}
\begin{Remark}
Note that  the convergence of $h^{\varepsilon}$ cannot be strong in $W^{1,\infty} (0,T;\R^{2})$, in general, as this would entail that
$$
 \ell_0 = K[w_{0}](0)= -\frac{1}{2 \pi} \int_{ \R^{2}} \frac{y^{\perp}}{| y|^{2}}  w_0  (y) \, dy .
 $$
\end{Remark}
Let us mention the paper \cite{Bjorland} which provides another derivation of the  vortex-wave system (\ref{EulerPoint})--(\ref{PointEuler}) from smooth solutions of the Euler equations alone, without any rigid body but with some concentrated vorticity, following the approach of  \cite{MP-CMP}  for the derivation of the vortex points system from smooth solutions of the Euler equations.\par
In the case of a massive pointwise particle, we have obtained in \cite{GLS} the same theorem except that the particle equation was:
\begin{equation} \label{OldEq}
m^1 h''(t)=\gamma\Big( h'(t)- \tilde{u}(t,h(t))\Big)^\perp  \ \text{ in } \ [0,T] ,
\end{equation}
and that, due to the lack of uniqueness in the limit system, the convergence held only along a subsequence. One can see that in the massive situation, the point vortex is accelerated by a force similar to the Kutta-Joukowski lift force of the irrotational theory. Formally, the massless situation corresponds to the case where $m^1=0$ in \eqref{OldEq}, from which one recovers the point vortex equation \eqref{PointEuler}. Nevertheless, the rigorous analysis is more complicated than in \cite{GLS} because the second-order equation \eqref{Solide1} degenerates to a first order equation \eqref{PointEuler}. For example, we will note in Remark~\ref{RemEstNRJ} that a standard energy estimate does not give anymore that $(h^\varepsilon)'$ is bounded uniformly in $\varepsilon$.
\par
Let us also mention the paper of Silvestre and Takahashi \cite{ST} (and the references therein) of a related problem with a small ball immersed in a 3D viscous fluid.
%
%
%
%
%
%
%
%
%
%
\section{Structure of the proof}  %
\label{Sec:Structure}
In this section, we describe the general structure of the proof of Theorem~\ref{MR}. \par
The basic estimates (energy estimates, estimates coming from the vorticity) for system \eqref{Euler1}-\eqref{Solideci}  will prove insufficient to pass to the limit. We will strengthen these estimates by establishing a so-called {\it modulated energy estimate}. To prove it, we will need to reformulate the equation in a suitable {\it normal form}. This is explained in greater detail below. \par
\ \par
\noindent
{\bf First step. A normal form.} 
First we introduce the solid velocity in the rotated frame:
\begin{equation*}
{\ell}^{\eps} (t)= R_{\theta^\varepsilon(t)}^T \ (h^{\eps})' (t) ,
\end{equation*}
where $R_{\theta^\eps}$ is the rotation of angle $\theta^\varepsilon$, see \eqref{NotationRot2D}.
Then we introduce the following ``modulated velocity'':
\begin{equation} \label{DefLTilde}
\tilde{\ell}^{\varepsilon}(t):= \ell^{\varepsilon}(t) -  K_{\R^{2}}[\omega^{\varepsilon}](t,0)
- \varepsilon D K_{\R^{2}}[\omega^{\varepsilon}](t,0) \cdot \xi,
\end{equation}
where $\xi$ is the conformal center of  $\mathcal{S}_0$ defined below in \eqref{DefXi} and $\omega^{\varepsilon}$ is the vorticity in the rotated frame, that is
\begin{equation} \label{DefOmega}
\omega^{\varepsilon}(t,x):=w^{\varepsilon}(t, R_{\theta^\varepsilon(t)}x+h^{\eps}(t) ) .
\end{equation}
Next we introduce the notation for the modulated unknown:
\begin{equation} \label{DefPtilde}
\tilde{p}^{\varepsilon}:= \begin{pmatrix} \tilde{\ell}^{\varepsilon} \\ \varepsilon r^{\varepsilon} \end{pmatrix}.
\end{equation}
To introduce the normal form, we need a few more notations.
We introduce the inertia matrix:
\begin{equation} \label{DefMG}
{\mathcal M}_{g} := \begin{pmatrix}
	m^1 & 0 & 0 \\
	0 & m^1 & 0 \\
	0 & 0 & {\mathcal J}^1
\end{pmatrix},
\end{equation}
and
\begin{equation} \label{DefB}
{\bf B}:= \begin{pmatrix} \xi^{\perp} \\ -1 \end{pmatrix}.
\end{equation}
We also introduce a bilinear symmetric mapping $\Lambda_{g}: \R^{3} \times \R^{3} \rightarrow \R^{3}$ as follows:
\begin{equation} \label{DefGammag}
\forall {p} \in \R^{3}, \ \  \langle \Lambda_{g}   , {p}, p \rangle = m^{1} r \begin{pmatrix} \ell^{\perp} \\ 0 \end{pmatrix}
 \ \text{ for } p = \begin{pmatrix} \ell \\ r \end{pmatrix}.
\end{equation}
Note that
\begin{equation} \label{AnnulationGammag}
\forall {p} \in \R^{3}, \ \ \ \langle \Lambda_{g}, {p}, p \rangle \cdot p =0.
\end{equation}
We are now in position to describe our normal form.
\begin{Proposition}\label{Pro:NormalForm}
There exist a symmetric positive  matrix ${\mathcal M}_{a} \in S^{+}_{3}(\R)$, depending only on 
${\mathcal S}_{0}$, and a bilinear symmetric mapping $\Lambda_{a}: \R^{3} \times \R^{3} \rightarrow \R^{3}$, depending only on ${\mathcal S}_{0}$, satisfying 
\begin{equation} \label{AnnulationGammaa}
\forall {p} \in \R^{3}, \ \  \langle \Lambda_{a}, {p}, p \rangle \cdot p =0,
\end{equation}
such that the following holds. \par
Let us fix $\rho>1$. There exist $C>0$ and $\varepsilon_{0}\in (0,1]$ such that: if for a given $T>0$ and an $\varepsilon \in (0,\varepsilon_{0}]$ one has for all $t \in [0,T]$:
\begin{equation} \label{CondDistVorticiteSolide}
\supp \omega^{\varepsilon}(t) \subset B(0,\rho)\setminus B(0,1/\rho),
\end{equation}
then there exist
a function $G=G(\varepsilon,t): (0,1) \times [0,T] \rightarrow \R^{3}$ satisfying
\begin{equation} \label{IneqWG}
\left| \int_0^t \tilde{p}^{\varepsilon}(s) \cdot G (\varepsilon,s) \, ds \right| 
\leq \varepsilon C \left( 1 + t + \int_0^t |\tilde{p}^{\varepsilon}(s)|^{2}\, ds \right) ,
\end{equation}
and a function $F=F(\varepsilon,t):(0,1) \times [0,T] \rightarrow \R^{3}$ satisfying
\begin{equation} \label{IneqWNL}
|F(\varepsilon,t)| \leq C \left( 1+ |\tilde{p}^{\varepsilon}(t)| + \varepsilon |\tilde{p}^{\varepsilon}(t)|^{2} \right),
\end{equation}
such that one has on $[0,T]$:
\begin{equation} \label{Eq:NormalForm}
\big[ \varepsilon^\alpha {\mathcal M}_{g} + \varepsilon^2 {\mathcal M}_{a} \big] (\tilde{p}^{\varepsilon})' 
+\langle \varepsilon^{\alpha-1} \Lambda_{g} + \varepsilon \Lambda_{a} , \tilde{p}^{\varepsilon}, \tilde{p}^{\varepsilon} \rangle
= \gamma\, \tilde{p}^{\varepsilon} \times {\bf B}
+ \varepsilon \gamma G(\varepsilon,t) 
+ \varepsilon^{\min(\alpha,2)} F(\varepsilon,t).
\end{equation}
\end{Proposition}
We will make ${\mathcal M}_{a}$, $\Lambda_{a}$ and $G$ explicit in the course of the proof, see  \eqref{defMa}, \eqref{DefGammaa}
and \eqref{DefG}. \par
We will refer to the quadratic mappings $\Lambda_{g}$ and $\Lambda_{a}$ satisfying \eqref{AnnulationGammag} and \eqref{AnnulationGammaa} as {\it gyroscopic terms}, to a function $G$ satisfying \eqref{IneqWG} as a {\it weakly gyroscopic term} and to a function $F$ satisfying \eqref{IneqWNL} as a {\it weakly nonlinear term}. \par
\ \par
\noindent
{\bf Second step. Modulated energy estimates.} 
From this normal form, we will be able to deduce the following modulated energy estimate.
\begin{Proposition} \label{Pro:ModulatedEnergy}
Let us fix $\rho>1$ and $\overline{T}>0$. There exist $C>0$ and $\varepsilon_{0}\in (0,1]$ such that the following holds. If for a given $T\in (0,\overline{T}]$ and an $\varepsilon \in (0,\varepsilon_{0}]$ one has that \eqref{CondDistVorticiteSolide} is valid on $[0,T]$, then one has
\begin{equation} \label{EstNRJMod}
|\ell^{\varepsilon}(t)| + \varepsilon|r^{\varepsilon}(t)| \leq C, \ \  \forall t \in [0,T].
\end{equation}
\end{Proposition}
This proposition improves the estimate coming from a basic energy argument (see Lemma~\ref{LemEstNRJ2}).
\begin{Remark}
We can track in the proofs of Propositions~\ref{Pro:NormalForm} and \ref{Pro:ModulatedEnergy} that the only constraint on $\varepsilon_{0}$ is to verify:
$$\partial {\mathcal S}_0^{\eps_{0}}\subset \overline{B(0,1/(2\rho))}.$$
\end{Remark}
\ \par
\noindent
{\bf Third step. Local and global passage to the limit.}
In a first time, we obtain the convergence stated in Theorem~\ref{MR} on a small time interval, and only in a second time we obtain this convergence on any time interval. The proof follows the following steps. \par
Since the modulated estimate above require assumptions on the support of the vorticity, taking \eqref{supportcontrol} into account, we set the following definition, given a fixed $\overline{T}>0$:
\begin{equation}\label{T-eps}
T_{\varepsilon}:= \sup\Big\{ \tau \in [0,\overline{T}],\ \forall t\in [0,\tau], \ 
\supp w^\varepsilon(t) \subset B(h^\varepsilon(t),2\rho_{\overline{T}}) \setminus B(h^\varepsilon(t),1/(2\rho_{\overline{T}}))
 \Big\},
\end{equation}
where $\rho_{\overline{T}}$ is defined from \eqref{supportcontrol} with $T=\overline{T}$. \par
As $\supp w_{0} \subset B(0,\rho_{\overline{T}}) \setminus B(0,1/\rho_{\overline{T}})$, we have of course $T_{\varepsilon}>0$. Using Proposition~\ref{Pro:ModulatedEnergy} with $\rho=2\rho_{\overline{T}}$ and $T=T_{\varepsilon}$, we deduce the following.
\begin{Proposition} \label{Pro:TempsMinimal}
Let us fix $\overline{T}>0$. There exists $\varepsilon_{0}>0$ and $\underline{T}>0$ such that
\begin{equation*}
\inf_{\varepsilon \in (0,\varepsilon_{0})} T_{\varepsilon} \geq \underline{T}.
\end{equation*}
\end{Proposition}
In turn this allows to prove the following local version of Theorem~\ref{MR}.
\begin{Proposition} \label{Pro:CVLocal}
We consider $\underline{T}>0$ such that $\inf_{\varepsilon \in (0,\varepsilon_{0})} T_{\varepsilon} \geq \underline{T}$ for some $\varepsilon_{0}>0$. Then $h^\eps $ converges to $h$ weakly-$\star$ in $W^{1,\infty} (0,\underline{T};\R^{2})$ and  $w^{\varepsilon}$ converges to $w$ in $C^{0} ([0,\underline{T}]; L^{\infty}(\R^{2})-w\star)$, where $(w,h)$ is the solution of the vortex-wave system \eqref{EulerPoint}--\eqref{PointEuler}.
\end{Proposition}
The proof uses a compactness argument (using the estimates above), the uniqueness of the solutions in the limit and Proposition~\ref{Pro:NormalForm}. \par
Finally we obtain Theorem~\ref{MR} by a sort of continuous induction argument. \par
%
%
%
%
%
%
%
%
%
%
%
\section{Basic material}
\label{Sec:Material}
In this section, we introduce basic material which we will use in subsequent sections.
In the whole paper, we will need some arguments of complex analysis (see e.g. Appendix~\ref{complex ana}): for the rest of the paper, we identify $\C$ and $\R^{2}$ through 
\begin{equation*}
(x_{1},x_{2})= x_{1} + i x_{2}=z.
\end{equation*}
The complex conjugate of a complex number $z$ will be classically denoted by $\overline{z}$, but we may also use the notation $z^{*}$ for large expressions. \par
We also use the notation $\widehat f = f_{1}-if_{2}$ for any $f=(f_{1},f_{2})$. The reason of this notation is the following consequence of the Cauchy-Riemann equations: 
\newline\centerline{ $f$ is divergence and curl free if and only if $\widehat f$ is holomorphic.}
%
%
%
%
%

%
%
%
%
\subsection{Equations in the body frame}
\label{Subsec:COF}
First we transfer the equations for the velocity and the vorticity in the body frame (all the details can be found in \cite{GLS}). 
For that we apply the following isometric change of variable:
\begin{equation*}
\left\{
\begin{array}{l}
 v^{\eps} (t,x)=R_{\theta^\eps(t)}^T \, u^\eps(t,R_{\theta^\eps(t)}x+h^{\eps}(t)), \\
 \omega^{\varepsilon}(t,x)=w^{\varepsilon}(t, R_{\theta^\eps(t)}x+h^{\eps}(t) ) = \curl v^{\varepsilon}(t,x)\\
 \tilde{\pi}^\varepsilon (t,x)=\pi^\eps(t,R_{\theta^\eps(t)}x+h^{\eps}(t)), \\
 {\ell}^{\eps} (t)=R_{\theta^\eps(t)}^T \ (h^{\eps})' (t) ,
\end{array}\right.
\end{equation*}
where we recall that the usual two dimensional rotation $R_{\theta^\eps}$ was introduced in \eqref{NotationRot2D}.
The equations  \eqref{Euler1}-\eqref{Solideci}  become
\begin{eqnarray}
\label{Euler11}
\displaystyle \frac{\partial v^{\eps}}{\partial t}
+ \left[(v^{\eps}-\ell^{\eps}-r^{\eps} x^\perp)\cdot\nabla\right]v^{\eps} 
+ r^{\eps} (v^{\eps})^\perp+\nabla \tilde{\pi}^\varepsilon =0 && x\in \mathcal{F}^{\eps}_{0} ,\\
\label{Euler12}
\div v^{\eps} = 0 && x\in \mathcal{F}^{\eps}_{0} , \\
\label{Euler13}
v^{\eps}\cdot n = \left(\ell^{\eps} +r^{\eps} x^\perp\right)\cdot n && x\in \partial \mathcal{S}^{\eps}_0, \\
\label{Solide11}
m^\eps (\ell^{\eps})'(t)=\int_{\partial \mathcal{S}_0^{\eps}} \tilde{\pi}^\varepsilon n \ ds-m^{\eps} r^{\eps} (\ell^{\eps})^\perp & & \\
\label{Solide12}
\mathcal{J}^{\varepsilon} (r^{\eps})'(t)=\int_{\partial \mathcal{S}^{\eps}_0} x^\perp \cdot \tilde{\pi}^\varepsilon n \ ds & &  \\
\label{Euler1ci}
v^{\eps}(0,x)= v^{\eps}_0 (x) && x\in \mathcal{F}^{\eps}_{0} ,\\
\label{Solide1ci}
\ell^{\eps}(0)= \ell_0,\ r^{\eps} (0)= r_0 . 
\end{eqnarray}
Moreover \eqref{Euler11} gives
\begin{equation} \label{vorty1}
\partial_t  \omega^{\eps} + \left[(v^{\eps}-\ell^{\eps}-r^{\eps} x^\perp)\cdot\nabla\right]  \omega^{\eps} =0 \text{ for }
x \in \mathcal{F}^{\eps}_{0} .
\end{equation}
The advantage of this formulation is that the space domain is now fixed. A large part of the analysis will be performed with these equations. \par
\ \par
As mentioned in Theorem~\ref{ThmYudo}, some quantities are conserved along the time, whose the circulation and the mass of the vorticity:
\begin{gather*}
\nonumber
\gamma =  \int_{  \partial \mathcal{S}^{\eps}_0} v^{\eps}  \cdot  \tau \, ds = \int_{  \partial \mathcal{S}^{\eps}(t)} u^{\eps}  \cdot  \tau \, ds
= \int_{  \partial \mathcal{S}^{\eps}_{0}} u^{\eps}_{0}  \cdot  \tau \, ds, \\
%
\beta^{\varepsilon} = \int_{ \mathcal{F}^{\eps}_0} \omega^{\eps}(t,x) \, dx = \int_{\mathcal{F}^{\eps}(t)} w^{\eps}(t,x) \, dx = \int_{\mathcal{F}^{\eps}_{0}} w^{\varepsilon}_{0}(x) \, dx.
\end{gather*}
As $w_{0}$ is assumed to be compactly supported in $\R^2 \setminus \{ 0\}$, we note that $\beta^{\varepsilon}$ is independent of $\eps$ when $\eps$ is small enough. \par
\ \par
In the next subsection, we introduce several velocity fields in the frame attached to the body. These will allow in particular to decompose the velocity field $v^{\varepsilon}$ in a way that clarifies how it is generated from the vorticity, the velocity of the rigid body and the circulation of the flow around the solid (see formulas \eqref{vdecomp} or \eqref{vitedechydro} below).
%
%
%
%
%
%
%
\subsection{Some useful velocity fields}
\label{Sec:GreensFunction}
We regroup the particular velocity fields mentioned above in three paragraphs. We refer to \cite{GLS} for more details. \par
\subsubsection{Harmonic field}
To take the velocity circulation around the body into account, we introduce the following harmonic field: let $H^{\eps}$ the unique solution vanishing at infinity of 
\begin{equation} \label{DefHeps}
\div H^{\eps} = 0 \quad   \text{in }  \mathcal{F}^{\eps}_{0}, \qquad \curl H^{\eps} = 0 \quad   \text{in }   \mathcal{F}^{\eps}_{0}, \qquad
H^{\eps} \cdot n = 0 \quad   \text{on }  \partial \mathcal{S}^{\eps}_0, \qquad
\int_{\partial \mathcal{S}^{\eps}_0 } H^{\eps} \cdot \tau \, ds = 1 .
\end{equation}
We list here a list of properties concerning $H^\varepsilon$ which are established in \cite{GLS}:
\begin{itemize}
\item The vector field $H^{\eps}$ admits a harmonic stream function $ \Psi_{H^{\varepsilon}} (x)$:
\begin{equation} \label{DefPsiH}
H^{\varepsilon} = \nabla^{\perp} \Psi_{H^{\varepsilon}},
\end{equation}
which vanishes on the boundary $ \partial \mathcal{S}^{\eps}_0$, and is equivalent to $\frac{1}{2\pi} \ln |x|$ as $x$ goes to infinity. 
\item We have the following scaling law
\begin{equation} \label{ScalingH}
H^{\varepsilon}(x) = \frac{1}{\varepsilon} H^{1} \left( \frac{x}{\varepsilon}\right).
\end{equation}
\item The function $\widehat{H}^{1}$ is holomorphic (as a function of $z=x_{1}+ix_{2}$), and can be decomposed in Laurent Series (see Remark~\ref{RemCircu}) with:
\begin{equation} \label{HSeriesLaurent}
\widehat{H}^{1}(z) = \frac{1}{2i\pi z} + {\mathcal O}(1/z^{2}) \ \text{ as } z \rightarrow \infty.
\end{equation}
Coming back to the variable $x\in \R^2$, the previous decomposition implies
\begin{equation*}
H^1(x) = {\mathcal O}\left(\frac{1}{|x|}\right) \text{ and } \nabla H^1 = {\mathcal O}\left(\frac{1}{|x|^{2}}\right).
\end{equation*}
\end{itemize}
From the scaling law  and the asymptotic behavior, we deduce the following estimate on the support of $\omega^\varepsilon$.
\begin{Lemma}\label{LemHomega}
Let us fix $\rho>1$. There exists $C>0$ such that the following holds. If for a given $T>0$ and $\eps\in (0,1]$, \eqref{CondDistVorticiteSolide} is valid on $[0,T]$, then one has
\begin{equation*} 
\| H^\varepsilon \|_{L^\infty(\supp \omega^\varepsilon(t))} +  \| \nabla H^\varepsilon \|_{L^\infty(\supp \omega^\varepsilon(t))}  \leq C \quad \forall t\in [0,T].
\end{equation*}
\end{Lemma}
The harmonic field $H^{1}$ allows to define the following geometric constant, appearing in \eqref{DefB} and known as the conformal center of ${\mathcal S}_{0}$:
\begin{equation} \label{DefXi}
\xi_{1} + i \xi_{2}: = \int_{ \partial  \mathcal{S}_{0}}  z \widehat{H^1}  \, dz = \varepsilon^{-1} \int_{ \partial  \mathcal{S}^{\varepsilon}_{0}}  z \widehat{H^\varepsilon}  \, dz.
\end{equation}
In the same way, we introduce $\eta$ is defined in $\R^{2}\simeq \C$ by
\begin{equation} \label{DefEta}
\eta=\eta_{1}+i\eta_{2}: = \int_{ \partial  \mathcal{S}_{0} } z^2 \widehat{H^1} \, dz
= \eps^{-2}\int_{ \partial  \mathcal{S}^\eps_{0}} z^2 \widehat{H^\eps}  \, dz .
\end{equation}
\subsubsection{Kirchhoff and other related potentials}\label{ParKirchoff}
\par
We will make use of the Kirchhoff potentials which allow to lift the boundary conditions in a harmonic manner: let $\Phi^{\eps}:=(\Phi^{\eps}_{i})_{i=1,2, 3}$ be the solutions of the following Neumann problems: 
\begin{equation} \label{def Phi}
-\Delta \Phi^{\eps}_i = 0 \quad   \text{in } \mathcal{F}^{\eps}_{0} ,\qquad
\Phi^{\eps}_i \longrightarrow 0 \quad  \text{when }  x \rightarrow  \infty, \qquad
\frac{\partial \Phi^{\eps}_i}{\partial n}=K_i  \quad  \text{on }   \partial \mathcal{F}^{\eps}_{0}  ,
\end{equation}
where we set on $\partial {\mathcal F}_{0}^{\varepsilon}$:
\begin{equation} \label{DefK1-3}
(K_{1},\, K_{2}, \, K_{3}) :=(n_1,\, n_2 ,\, x^\perp \cdot n).
\end{equation}
Note that $K_{1}$, $K_{2}$ and $K_{3}$ actually depend on $\varepsilon$. 
Changing variables according to $y=x/\eps$, we see that
\begin{gather} \label{phi-scaling}
\Phi_{i}^\eps(x)= \eps \Phi_{i}^1(x/\eps) \ \text{ for } i=1,2, \qquad
\Phi_{3}^\eps(x)= \eps^{2} \Phi_{3}^1(x/\eps).
\end{gather}
\ \par
Next to design approximations of the velocity field, we will use the following related potentials. Let $\Phi^{\eps}_{4}$ and $\Phi^{\eps}_{5}$ be the unique solution of \eqref{def Phi} where $K_{4}$ and $K_{5}$ are defined on $\partial {\mathcal F}_{0}^{\varepsilon}$ by:
\begin{equation} \label{DefK4-5}
(K_{4},\, K_{5}) :=\left( \begin{pmatrix} -x_{1}\\x_{2} \end{pmatrix} \cdot n, \begin{pmatrix} x_{2}\\x_{1} \end{pmatrix}\cdot n \right).
\end{equation}
The scaling law for $\Phi_{i}$ with $i=4,5$ is the same as for $i=3$:
\begin{gather} \label{phi-scaling2}
\Phi_{i}^\eps(x)= \eps^2 \Phi_{i}^1(x/\eps) \ \text{ for } i=4,5.
\end{gather}
We have from Lemma~\ref{LemmeCircu} that for all $i=1,2,3,4,5$:

\begin{equation} \label{ComportementPhii}
\Phi_{i}^1(x) = {\mathcal O}\left(\frac{1}{|x|}\right) \text{ and }
\nabla \Phi_{i}^1(x) = {\mathcal O}\left(\frac{1}{|x|^{2}}\right) \text{ as } |x| \rightarrow +\infty,
\end{equation}
and consequently that for all $i=1,2,3,4,5$, $\nabla \Phi^{\eps}_i$ belongs to $L^2 (\mathcal{F}^{\eps}_{0})$. \par
From the scaling law  and the asymptotic behavior, we deduce the following estimate on the support of $\omega^\varepsilon$.
\begin{Lemma}\label{LemPhiomega}
Let us fix $\rho>1$. There exists $C>0$ such that the following holds. If for a given $T>0$ and $\eps\in (0,1]$, \eqref{CondDistVorticiteSolide} is valid on $[0,T]$, then one has
\begin{equation*} 
\varepsilon^{-2}\| \nabla \Phi_{1}^\varepsilon \|_{L^\infty(\supp \omega^\varepsilon(t))} + \varepsilon^{-2} \| \nabla \Phi_{2}^\varepsilon \|_{L^\infty(\supp \omega^\varepsilon(t))} + \varepsilon^{-3} \| \nabla \Phi_{3}^\varepsilon \|_{L^\infty(\supp \omega^\varepsilon(t))} \leq C \quad \forall t\in [0,T].
\end{equation*}
\end{Lemma}
\ \par
We introduce the following quantities for $i,j \in \{ 1,2,3,4,5\}$:
\begin{equation} \label{DefMasse}
m^{\varepsilon}_{i,j}:= \int_{{\mathcal F}^{\varepsilon}_{0}} \nabla \Phi^{\varepsilon}_{i} \cdot \nabla \Phi^{\varepsilon}_{j}
= {\eps}^{2 + \delta_{\{i \geq 3\}} + \delta_{\{j \geq 3\}}} \int_{{\mathcal F}_{0}} \nabla \Phi^{1}_{i} \cdot \nabla \Phi^{1}_{j} .
\end{equation}
Of particular importance is the following matrix
\begin{equation} \label{AddedMass}
\mathcal{M}^{\eps}_a 
:=\begin{bmatrix} m^{\varepsilon}_{i,j} \end{bmatrix}_{i,j \in \{1,2,3\}}
= \begin{bmatrix} \displaystyle {\eps}^{2 + \delta_{i,3} + \delta_{j,3}} m^{1}_{i,j} \end{bmatrix}_{i,j \in \{1,2,3\}}
= \eps^2I_\eps \mathcal{M}_a I_\eps ,
\end{equation}
with
\begin{equation} \label{defIeps}
I_{\varepsilon} := \begin{pmatrix}
	1 & 0 & 0 \\
	0 & 1 & 0 \\
	0 & 0 & \varepsilon
\end{pmatrix},
\end{equation}
and
\begin{equation} \label{defMa}
\mathcal{M}_a = \begin{bmatrix} m^{1}_{i,j} \end{bmatrix}_{i,j \in \{1,2,3\}}.
\end{equation}
The matrix $\mathcal{M}^{\eps}_a$ actually encodes the phenomenon of added mass, which, loosely speaking, measures how much the  surrounding fluid resists the acceleration as the body moves through it. The index $a$ refers to ``added''; at the opposite the index $g$ in ${\mathcal  M}_{g}$ and $\Lambda_{g}$ (see \eqref{DefMG} and \eqref{DefGammag}) stands for ``genuine''. \par
We will also use the $2\times2$ restriction of $\mathcal{M}_{a}^\eps$:
\begin{equation} \label{DefMbemol}
\mathcal{M}_{\flat}^\eps = \begin{pmatrix} m_{1,1}^\eps & m_{1,2}^\eps \\ m_{1,2}^\eps & m_{2,2}^\eps \end{pmatrix}.
\end{equation}
\subsubsection{Biot-Savart kernel}

In this paragraph, we recall briefly the elliptic $\div$/$\curl$ system which allows to pass from the vorticity to the velocity field.
We denote by $G^{\eps} (x,y)$ the Green's function for the Laplacian in $\mathcal{F}^{\eps}_{0}$ with Dirichlet boundary conditions, 
and we introduce the kernel 
\begin{equation*}
K^{\eps} (x,y)=\nabla^\perp_{x} G^{\eps} (x,y)
\end{equation*}
of the Biot-Savart operator  $K^{\eps} [\om]$ which therefore acts on $\om^{\eps} \in L^\infty_c (\overline{\mathcal{F}^{\eps}_{0}})$  through the formula 
\begin{equation*}
  K^{\eps}[\om^{\eps}](x)= \int_{\mathcal{F}^{\eps}_{0}}K^{\eps} (x,y) \om^{\eps}(y) \, dy .
\end{equation*}
For $\om^{\eps}\in L^\infty_c (\overline{\mathcal{F}^{\eps}_{0}})$, we recall that $K^{\eps}[\om^{\eps}]$ is in $\mathcal{LL}(\mathcal{F}^{\eps}_{0})$, divergence-free, tangent to the boundary, square integrable (namely $K^{\eps}[\om^{\eps}](x) = {\mathcal O}( \frac{1}{|x|^{2}}) \ \text{ as } x \rightarrow \infty$) and such that $\curl K^{\eps}[\om^{\eps}]=\om^{\eps}$.
Moreover its circulation around $\partial \mathcal{S}^{\eps}_0$ is given by
\begin{equation*}
\int_{\partial \mathcal{S}^{\eps}_0  }   K^{\eps} [ \omega^{\eps}] \cdot {\tau} \, ds = -  \int_{\mathcal{F}^{\eps}_{0}  }  \omega^{\eps} \, dx ,
\end{equation*}
where $\tau=-n^\perp$ is the tangent unit vector field on $\partial {\mathcal S}^{\varepsilon}_{0}$.
\ \par
As in the introduction, we denote by $K_{\R^{2}}$ the Biot-Savart operator associated to the full plane, that is the operator which maps a vorticity $\omega$ to the velocity
\begin{equation*} 
K_{\R^{2}} \lbrack \omega \rbrack (x) :=  \frac1{2\pi}\int_{ \R^{2}} \frac{(x-y)^\perp}{|x-y|^2}  \omega (y) \,dy .
\end{equation*}
For  $\om\in L^\infty_{c} (\R^2)$, $K_{\R^{2}}  [\om]$ is bounded, continuous, divergence-free and such that $\curl K_{\R^{2}} [\om]=\om$.
Moreover, there exists $C>0$ such that
\begin{equation} \label{loglip}
\| K_{\R^{2}} [\om]  \|_{\mathcal{LL} (\R^2)} \leq  C (\| \om  \|_{L^\infty (\R^2)} +  \| \om  \|_{L^1 (\R^2)}), 
\end{equation}
and
\begin{equation} \nonumber
K_{\R^{2}} [\om](x) = {\mathcal O}\left( \frac{1}{|x|}\right) \ \text{ as } x \rightarrow \infty.
\end{equation}
We will also use several times the fact that $K_{\R^{2}}$ commutes with translations and obeys to the following rule with rotations in the plane:
\begin{equation*}
K_{\R^{2}} [ \omega \circ R_{\theta}] (x) = R_{\theta}^{T} \, K_{\R^{2}} [ \omega ] (R_{\theta} x) .
\end{equation*}
Now, for $\om^{\eps}$ in $L^\infty_c (\mathcal{F}^{\eps}_{0})$, $\ell^{\eps}$ in $\R^2$, $r^{\eps}$ and $\gamma$ in $\R$, there exists a unique vector field $v^{\eps}$ verifying:
\begin{equation} \label{DivCurlSystem}
\begin{aligned}
 \div v^{\eps} = 0     \text{ in }  \mathcal{F}^{\eps}_{0} , \quad
 \curl v^{\eps}  = \omega^{\eps}   \text{ in }   \mathcal{F}^{\eps}_{0}  , \quad
 v^{\eps} \cdot n = \left(\ell^{\eps}+r^{\eps} x^\perp\right)\cdot n     \text{ on }  \partial \mathcal{S}^{\eps}_0  , \\
 \int_{ \partial \mathcal{S}^{\eps}_0} v^{\eps}  \cdot  \tau \, ds=  \gamma , \quad
 \lim_{x\to \infty} v^{\eps}(x) = 0, 
\end{aligned}\end{equation}
and it is given by the following Biot-Savart law:
\begin{equation} \label{vdecomp}
v^{\eps} = K^{\eps} [\omega^{\eps} ] + (\gamma + \beta^{\eps} )  H^{\eps} + \ell_1^{\eps} \nabla \Phi^{\eps}_1 + \ell_2^{\eps} \nabla \Phi^{\eps}_2
+ r^{\eps} \nabla \Phi^{\eps}_3 ,
\end{equation}
with
\begin{equation*} 
\beta^{\eps} :=   \int_{\mathcal{F}^{\eps}_{0}  }  \omega^{\eps} \, dx .
\end{equation*}
We also introduce the so-called hydrodynamic Biot-Savart kernel $K_{H}^{\varepsilon}$:
\begin{equation*}
K_{H}^{\eps}[\om^{\eps}](x) :=
K^\eps[\om^{\eps}](x) +  \beta^{\eps} H^{\varepsilon}(x).
\end{equation*}
and consequently
\begin{eqnarray*}
\div K^\eps_H [\omega^{\eps}]= 0  \text{ in }  \mathcal{F}^{\eps}_{0}, \quad
\curl K^\eps_H [\omega^{\eps}]  = \omega^{\eps}  \text{ in }   \mathcal{F}^{\eps}_{0}, \quad
K^\eps_H [\omega^{\eps}] \cdot n = 0 \text{ on }   \partial \mathcal{S}^{\eps}_0, \quad
\int_{  \partial \mathcal{S}^{\eps}_0} K^\eps_H [\omega^{\eps}]  \cdot  \tau\, ds  = 0 .
\end{eqnarray*}
Hence another possibility for the decomposition of $v^{\eps}$ is
\begin{equation} \label{vitedechydro}
v^{\eps} =  K^{\eps}_H [\omega^{\eps}] + \gamma H^\eps + \ell_1^{\eps} \nabla \Phi^\eps_1
+ \ell_2^{\eps} \nabla \Phi^\eps_2 + r^{\eps} \nabla \Phi^{\varepsilon}_{3}  .
\end{equation}
We also mention the fact (see \cite{GLS}) that
\begin{equation*}
K_{H}^{\eps}[\om](x)= \int_{\mathcal{F}^{\eps}_{0}} \nabla^{\perp}_{x} G_{H}^{\eps} (x,y) \om(y) \, dy,
\end{equation*}
where the hydrodynamic Green function $G_{H}^{\varepsilon}$ is given as follows:
\begin{equation} \label{GepsHydro}
G^{\eps}_{H} (x,y) := G^{\eps} (x,y) +  \Psi_{H^\eps} (x) +  \Psi_{H^\eps} (y) ,
\end{equation}
with $\Psi_{H^\eps}$ defined in \eqref{DefPsiH}.
Finally we also introduce the part of $v^{\eps}$ without circulation:
\begin{equation} \label{allo}
\tilde{v}^{\eps}: = v^{\eps} - \gamma H^{\eps}
=  K_{H}^{\eps}[\omega^\eps] + \ell^{\eps}_1 \nabla \Phi^\eps_1
+ \ell^{\eps}_2 \nabla \Phi^\eps_2 + r^{\varepsilon} \nabla \Phi_{3}^{\varepsilon}.
\end{equation}
%
%
%
%
%
%
%
%
%
%
%
%
%
\section{First a priori estimates}
%
%

Let $\ell_{0} \in \R^{2}$, $r_{0} \in \R^{2}$, $\gamma \in \R$ and $w_{0}\in L^\infty_{c}(\R^2\setminus\{0\})$ be given. We associate the initial velocity field $u_{0}^{\varepsilon}$ in ${\mathcal F}_{0}^{\varepsilon}$ by \eqref{DivCurlSystem} (restricting in particular $w_{0}$ to ${\mathcal F}_{0}^{\varepsilon}$) and consider the resulting solution $(h^{\varepsilon},\theta^{\varepsilon},u^{\varepsilon})$ given by Theorem~\ref{ThmYudo}. By the change of variable above we obtain the corresponding solution $(\ell^\eps, r^\eps, v^\eps)$ of \eqref{Euler11}-\eqref{Solide1ci}. The goal of this section is to derive basic a priori bounds on the solutions $(\ell^\eps, r^\eps, v^\eps)$ and on the vorticity $\omega^{\varepsilon}$.
\subsection{Vorticity}\label{Subsec:vor}
Due to \cite{GS}, the generalized enstrophies are conserved when time proceeds, in particular, we have for any $t>0$ and any $p \in [1,\infty]$,
\begin{equation} \label{ConsOmega}
\| \omega^{\eps} (t, \cdot) \|_{L^{p}(\mathcal{F}^{\eps}_{0})} 
= \| w^{\varepsilon}_{0} \|_{ L^{p}( \mathcal{F}^{\eps}_{0} )}
\leq \| w_{0} \|_{ L^{p}( \R^{2})}.
\end{equation}
Extending $\omega^{\eps}$ by  $0$ inside ${\mathcal S}^{\varepsilon}_0$, we deduce from \eqref{loglip} and \eqref{ConsOmega} that
\begin{equation} \label{loglip2}
\| K_{\R^{2}} [\omega^{\eps} (t, \cdot)]  \|_{\mathcal{LL} (\R^2)} \text{ is bounded independently of $t$ and of $\eps$}.
\end{equation}
\subsection{Energy}
\label{NRJ}
Let us introduce the total mass matrix 
\begin{equation} \label{InertieMatrix}
\mathcal{M}^{\eps} :=   \mathcal{M}^{\varepsilon}_g + \mathcal{M}^{\eps}_a 
\ \text{ where } \ 
\mathcal{M}^{\varepsilon}_g := \begin{bmatrix} m^\eps \Id_2 & 0 \\ 0 & \mathcal{J}^{\varepsilon} \end{bmatrix}.
\end{equation}
Observe that 
\begin{equation}
\label{raoul}
\mathcal{M}^{\varepsilon}_g
= \eps^{\alpha}I_\eps \mathcal{M}_g I_\eps ,
\end{equation}
where $I_\eps$ is defined in  \eqref{defIeps} and  $\mathcal{M}_g$ in \eqref{DefMG}.
We recall that the added mass matrix $\mathcal{M}^{\eps}_a$ was defined in \eqref{AddedMass}. 
The matrix $\mathcal{M}^{\eps}$ is symmetric and positive definite. \par
Using this matrix, one can deduce the following conserved quantity, where we recall that the functions $\Psi_{H^{\varepsilon}}$ and $G^{\varepsilon}_{H}$ were respectively defined in \eqref{DefPsiH} and \eqref{GepsHydro}.
\begin{Proposition} \label{PropKirchoffSueur}
The following quantity is conserved along the motion:
\begin{eqnarray*}
2 \mathcal{H}^{\eps} = (p^{\varepsilon})^{T} \mathcal{M}^{\varepsilon} p^{\varepsilon}
- \int_{\mathcal{F}^{\varepsilon}_0 \times  \mathcal{F}^{\varepsilon}_0 }  G^{\varepsilon}_{H} (x,y) \omega^{\varepsilon} (x) \omega^{\varepsilon} (y) \, dx \, dy
-  2  \gamma \int_{\mathcal{F}^{\varepsilon}_0  } \omega^{\varepsilon} (x)  \Psi_{H^{\varepsilon}} (x) \, dx ,
\end{eqnarray*}
where  
\begin{equation*}
p^{\varepsilon} := \begin{pmatrix} \ell^{\varepsilon} \\ r^{\varepsilon} \end{pmatrix}.
\end{equation*}
\end{Proposition} 
The proof of Proposition~\ref{PropKirchoffSueur} is given in \cite{GLS}. Therein, we can also find the following consequence: \par
\begin{Proposition}\label{ProRata}
Let $\rho_{w_0}>0$ such that $\supp w_0 \subset B(0, \rho_{w_0})$.
One has the following estimate for some constant $C=C ( m^{1}, {\mathcal J}^{1}, \| w_{0} \|_{L^{1} \cap L^{\infty}} , |\ell_{0}|, |r_{0}|,|\gamma|, \rho_{w_{0}})$, depending only on these values and the geometry for $\varepsilon=1$:
\begin{equation} \label{EstNRJ0}
(p^{\varepsilon})^{T} \mathcal{M}^{\varepsilon} p^{\varepsilon}
\leq C [1 + \ln (\rho^{\varepsilon}(t))],
\end{equation}
where
\begin{equation*}
\rho^{\varepsilon}(t) := \rho_{\omega^{\varepsilon}(t,\cdot)} =\inf \{ d >1 \ / \  \supp \omega^{\varepsilon}(t,\cdot) \subset B(0,d) \}.
\end{equation*}
\end{Proposition}
Now using \eqref{AddedMass} and \eqref{raoul}, the estimate \eqref{EstNRJ0} can be rewritten as
\begin{equation*}
\eps^{\alpha} (p^{\varepsilon})^{T} I_\eps \mathcal{M}_g I_\eps p^{\varepsilon} + \eps^{2} (p^{\varepsilon})^{T} I_\eps \mathcal{M}_a I_\eps p^{\varepsilon} \leq C [1 + \ln (\rho^{\varepsilon}(t))].
\end{equation*}
Now there are two possibilities. The first possibility is that ${\mathcal S}_{0}$ is not a ball, and in that case, $\mathcal{M}_a$ is positive definite as a Gram matrix associated to a free family of vectors (the third one degenerates when ${\mathcal S}_{0}$ is a ball!). Hence we deduce that
\begin{equation} \label{EstNRJ}
|\varepsilon^{\min(1,\alpha/2)} \ell^{\varepsilon}(t) | + |\varepsilon^{1+\min(1,\alpha/2)} r^{\varepsilon}(t)| \leq C [1 + \ln (\rho^{\varepsilon}(t))].
\end{equation}
The second possibility is that ${\mathcal S}_{0}$ is a ball. But in this situation we infer from \eqref{Solide2} that $r^{\varepsilon}(t)$ is constant over time (and in particular is trivially bounded). Moreover, using that in this case the restriction of $\mathcal{M}_a$ to the first two coordinates is $m_{11} \Id$, we see that \eqref{EstNRJ} is also valid in this case. 
\begin{Remark}\label{RemEstNRJ}
In \cite{GLS}, we could deduce by a Gronwall argument that $\rho^{\varepsilon}$ was bounded on $[0,T]$ uniformly in $\varepsilon$.
This estimate cannot be straightforwardly deduced here, because of the powers of $\varepsilon$ on the left hand side.
For $\alpha=2$, for instance, we have $|\varepsilon \ell^{\varepsilon}(t) | + |\varepsilon^2 r^{\varepsilon}(t)|$ in \eqref{EstNRJ}, instead of $| \ell^{\varepsilon}(t) | + |\varepsilon r^{\varepsilon}(t)|$ in \cite{GLS}. 
\end{Remark}
As a byproduct of \eqref{EstNRJ}, we will use throughout the next section the following estimate on $(\ell^\varepsilon,r^\varepsilon)$:
\begin{Lemma} \label{LemEstNRJ2}
Let us fix $\rho>1$. There exists $C>0$ such that the following holds. If for a given $T>0$ and $\eps\in (0,1]$, \eqref{CondDistVorticiteSolide} is valid on $[0,T]$, then one has
\begin{equation*}
|\varepsilon \ell^{\varepsilon}(t) | + |\varepsilon^2 r^{\varepsilon}(t)| \leq C  \quad \forall t\in [0,T].
\end{equation*}
\end{Lemma}
\subsection{Basic velocity estimates}
\label{APF}
We can deduce from Lemma~\ref{LemEstNRJ2} and from the results of Subsection~\ref{Sec:GreensFunction} new estimates on the fluid velocity, as long as the solid stays away from the vorticity and that the support of the vorticity remains bounded. \par
\par
First, let us recall the following lemma  (see \cite[Theorem 4.1]{ift_lop_euler}):
\begin{Lemma} \label{ift_lop}
There exists a constant $C>0$ which depends only on the shape of the solid for $\varepsilon=1$ such that for any $ \omega$ smooth enough,
\begin{equation*} 
\|  K^{\eps}_H [\omega ]  \|_{L^{\infty} ( \mathcal{F}^{\eps}_{0}) }  \leqslant C  \|    \omega \|^{1/2}_{ L^{1} ( \mathcal{F}^{\eps}_{0} )}   \|    \omega \|^{1/2}_{ L^{\infty} ( \mathcal{F}^{\eps}_{0} )} .
\end{equation*}
\end{Lemma}
Combining with the conservation laws \eqref{ConsOmega} we obtain that for any $t>0$,
\begin{equation*}
\|K^{\eps}_H [\omega ] (t,\cdot)  \|_{L^{\infty}(\mathcal{F}^{\eps}_{0}) }
\leqslant C  \| w_{0} \|^{1/{2}}_{ L^{1} ( \R^{2} )} \| w_{0} \|^{1/2}_{ L^{\infty} ( \R^{2} )} .
\end{equation*}
Together with Lemmas~\ref{LemHomega}, \ref{LemPhiomega} and \ref{LemEstNRJ2}, we deduce from the decompositions \eqref{vitedechydro} and \eqref{allo} the next velocity estimates.
\begin{Lemma}\label{LeminfVomega}
Let us fix $\rho>1$. There exists $C>0$ such that the following holds. If for a given $T>0$ and $\eps\in (0,1]$, \eqref{CondDistVorticiteSolide} is valid on $[0,T]$, then one has
\begin{equation*} 
\| {v}^{\eps} (t,\cdot) \|_{L^{\infty}(\supp \omega^\varepsilon(t)) }
+ \| \tilde{v}^{\eps} (t,\cdot) \|_{L^{\infty}(\supp \omega^\varepsilon(t)) }
\leq C \quad \forall t\in [0,T].
\end{equation*}
\end{Lemma}
\subsection{Estimates related to the modulation terms}
Due to the definition of the modulated velocity $\tilde \ell^\eps$  \eqref{DefLTilde}, we are interested in estimating the terms $K_{\R^2}[\omega^\eps](t,0)$, $D K_{\R^2}[\omega^\eps](t,0) \cdot \xi$ and their time derivatives. \par
\ \par
Regarding the modulation terms themselves, we have the following.
\begin{Lemma} \label{LemModulesBornes}
Let us fix $\rho>1$. There exists $C>0$ such that the following holds. If for a given $T>0$ and $\eps\in (0,1]$, \eqref{CondDistVorticiteSolide} is valid on $[0,T]$, then one has
\begin{equation*}
\| K_{\R^2}[\omega^\eps](t,0) \|_{L^{\infty}(0,T)} + \| D K_{\R^2}[\omega^\eps](t,0)   \|_{L^{\infty}(0,T)} \leq C.
\end{equation*}
\end{Lemma}
\begin{proof}
The statement concerning the first term is actually included in \eqref{loglip2}. Concerning the second one, we use that $K_{\R^{2}}[\omega^{\eps}]$ is harmonic on $B(0,1/\rho)$, so its $C^1$  norm at $0$ can be estimated by the $L^\infty$ norm on $B(0,1/\rho)$.
\end{proof}
\ \par
We now turn an estimate relative to the time derivatives of the modulation terms. 
We will use the elementary formula (that we already used to obtain \eqref{Solide11}):
\begin{equation} \label{DeriveeRotation}
(R_{\theta^\eps(t)}^T)' = -r^\eps(t) R_{\theta^\eps(t)}^T J_{2} \ \text{ and } \ 
(R_{\theta^\eps(t)})' = r^\eps(t) J_{2} R_{\theta^\eps(t)} \ \text{ with } \ 
J_{2} := 
\begin{pmatrix}
 0 & -1 \\
 1 & 0
\end{pmatrix},
\end{equation}
and we deduce from ${\ell}^{\eps} (t)=R_{\theta^\eps(t)}^T \ (h^{\eps})' (t)$ that
\begin{equation*}
({\ell}^{\eps})' (t)=-r^\eps R_{\theta^\eps(t)}^T \ ((h^{\eps})')^{\perp} (t) + R_{\theta^\eps(t)}^T \ (h^{\eps})'' (t)= -r^\eps ({\ell}^{\eps})^{\perp} (t) + R_{\theta^\eps(t)}^T \ (h^{\eps})'' (t).
\end{equation*}
We have the following statement concerning the time derivatives of the modulation terms.
\begin{Proposition} \label{labelkoi2}
Let us fix $\rho>1$. There exist $C>0$ and $\varepsilon_{0}\in (0,1]$ such that the following holds. If for a given $T>0$ and $\eps\in (0,\varepsilon_{0}]$, \eqref{CondDistVorticiteSolide} is valid on $[0,T]$, then one has
\begin{equation}\label{K0'}
\left(K_{\R^2}[\omega^\eps](t,0) + \eps D K_{\R^2}[\omega^\eps](t,0) \cdot \xi \right)' =  -r^\eps(t)  (K_{\R^2}[\omega^\eps](t,0))^{\perp}  + F_{d}(\eps,t)
\end{equation}
where $F_{d}$ is weakly nonlinear in the sense of \eqref{IneqWNL}.
\end{Proposition}
\begin{proof}
An important fact in the proof is that $\partial_{t} \omega^\eps$ becomes singular as $\varepsilon \rightarrow 0^{+}$ (see the factor $r^\eps x^\perp$ in \eqref{vorty1}); as a consequence, we will rather consider $\partial_{t} w^\eps$ which behaves less singularly. \par
From \eqref{DefOmega}, we easily see that 
\begin{equation}\label{KomKw}
 K_{\R^2}[\omega^\eps](t,x)=R_{\theta^\eps (t)}^T K_{\R^2}[w^\eps](t,R_{\theta^\eps(t)}x+h^{\eps}(t))
\end{equation}
and
\begin{equation} \label{nablaK}
D K_{\R^2}[\omega^\eps](t,x) \cdot (\cdot) = R_{\theta^\eps (t)}^T \, D K_{\R^2}[w^\eps](t,R_{\theta^\eps(t)}x+h^{\eps}(t)) \cdot R_{\theta^\eps(t)} (\cdot).
\end{equation}
Now, using \eqref{DeriveeRotation}, we compute the time derivative of $K_{\R^2}[\omega^\eps](t,0)$:
\begin{align*}
K_{\R^2}[\omega^\eps](t,0)' =& -r^\eps(t) R_{\theta^\eps (t)}^T (K_{\R^2}[w^\eps](t,h^{\eps}(t)))^{\perp} + R_{\theta^\eps (t)}^T K_{\R^2}[\partial_{t} w^\eps](t,h^{\eps}(t)) \\
&+ R_{\theta^\eps (t)}^T D K_{\R^2}[w^\eps](t,h^{\eps}(t)) \cdot (h^\eps)'(t)\\
=& -r^\eps(t)  (K_{\R^2}[\omega^\eps](t,0))^{\perp} + R_{\theta^\eps (t)}^T K_{\R^2}[\partial_{t} w^\eps](t,h^{\eps}(t))
+ D K_{\R^2}[\omega^\eps](t,0) \cdot \ell^\eps(t),
\end{align*}
thanks to \eqref{KomKw}-\eqref{nablaK}. \par
For the time derivative of $D K_{\R^2}[\omega^\eps](t,0) \cdot \xi$, we get from \eqref{DeriveeRotation} that
\begin{align*}
(D K_{\R^2}[\omega^\eps](t,0) \cdot \xi )' =& -r^\eps(t) (D K_{\R^2}[\omega^\eps](t,0) \cdot \xi)^\perp    
+   R_{\theta^\eps (t)}^T D K_{\R^2}[\partial_{t} w^\eps](t,h^{\eps}(t)) \cdot R_{\theta^\eps(t)}  \xi \\
& +   R_{\theta^\eps (t)}^T D^2 K_{\R^2}[ w^\eps](t,h^{\eps}(t)) : \Big( (h^{\eps})'(t) \otimes R_{\theta^\eps(t)} \xi\Big)
 + r^\eps(t) D K_{\R^2}[\omega^\eps](t,0) \cdot \xi^\perp.
\end{align*}
Putting these two computations together, we obtain \eqref{K0'} with
\begin{align*}
F_{d}(\eps, t):=&R_{\theta^\eps (t)}^T K_{\R^2}[\partial_{t} w^\eps](t,h^{\eps}(t)) + D K_{\R^2}[\omega^\eps](t,0) \cdot \ell^\eps(t) \\
&-\eps r^\eps(t) (D K_{\R^2}[\omega^\eps](t,0) \cdot \xi)^\perp    
+   \eps R_{\theta^\eps (t)}^T D K_{\R^2}[\partial_{t} w^\eps](t,h^{\eps}(t))R_{\theta^\eps(t)} \cdot \xi \\
& +   \eps R_{\theta^\eps (t)}^T D^2 K_{\R^2}[ w^\eps](t,h^{\eps}(t)) : \Big( (h^{\eps})'(t) \otimes R_{\theta^\eps(t)} \xi\Big)
 +\eps r^\eps(t) D K_{\R^2}[\omega^\eps](t,0) \cdot \xi^\perp,
\end{align*}
which we now prove to be weakly nonlinear. \par
As in Lemma~\ref{LemModulesBornes}, we see that, by harmonicity of $K_{\R^{2}}[\omega^{\eps}]$ on $B(0,1/\rho)$ and its $L^\infty$ estimate \eqref{loglip2}, there exists $C$ such that
\begin{gather*}
 \left| D K_{\R^2}[\omega^\eps](t,0) \cdot \ell^\eps(t) -\eps r^\eps(t) (D K_{\R^2}[\omega^\eps](t,0) \cdot \xi)^\perp  \right | \leq C( |\ell^\eps(t)| + \eps |r^\eps(t)|)\\
\left| \eps R_{\theta^\eps (t)}^T D^2 K_{\R^2}[ w^\eps](t,h^{\eps}(t)) : \Big( (h^{\eps})'(t) \otimes R_{\theta^\eps(t)} \xi\Big)
 +\eps r^\eps(t) D K_{\R^2}[\omega^\eps](t,0) \cdot \xi^\perp \right| \leq C(\eps |\ell^\eps(t)|  + \eps |r^\eps(t)|).
\end{gather*}
Since we assume \eqref{CondDistVorticiteSolide} to be valid on $[0,T]$, we deduce that $u^\eps w^\eps$ is compactly supported in $B(h^\eps(t),\rho)\setminus B(h^\eps(t),1/\rho)$.
Hence we can infer from \eqref{ConsOmega} and Lemma~\ref{LeminfVomega} that $u^\eps w^\eps$ is bounded in $L^1\cap L^\infty(\R^2)$ uniformly in $\eps$. Therefore, using the vorticity equation $\partial_{t} w^\eps =- \div (u^\eps w^\eps)$, we deduce that $K_{\R^2}[\partial_{t} w^\eps ]$ is uniformly bounded and harmonic around $h^\eps(t)$. This gives that
\begin{equation*}
\left| R_{\theta^\eps (t)}^T K_{\R^2}[\partial_{t} w^\eps](t,h^{\eps}(t)) \right | \leq C
 \ \text{ and } \ 
\left|  \eps R_{\theta^\eps (t)}^T D K_{\R^2}[\partial_{t} w^\eps](t,h^{\eps}(t)) \cdot R_{\theta^\eps(t)} \xi \right|
 \leq C \eps.
\end{equation*}
Hence one obtains an estimate for $F_{d}$ of the form:
\begin{equation*}
|F_{d}(\varepsilon,t)| \leq C \left( 1+ | \ell^{\varepsilon}(t)| + |\varepsilon {r}^{\varepsilon}(t)| \right).
\end{equation*}
Putting the ``modulated velocity''  $\tilde{\ell}^{\varepsilon}$ (see \eqref{DefLTilde}) in the right hand side instead of $\ell^{\varepsilon}(t)$, taking Lemma~\ref{LemModulesBornes} into account, this gives that $F_{d}$ is weakly nonlinear.
\end{proof}
%
%
%
%
%
%
%
%
%
%
%
%
\subsection{Approximation of the velocity}
For the computation of the pressure force, it will useful to approximate the velocity vector field $\tilde{v}^{\eps}$ (introduced in \eqref{allo}) along the solid boundary.
As $K_{\R^{2}}[\omega^{\eps}]=\nabla^\perp \psi[\omega^{\eps}]$ with $\Delta \psi =0$ in the neighborhood of the solid (namely on $B(0,1/\rho)$, see \eqref{CondDistVorticiteSolide}), then  $D K_{\R^{2}}[\omega^{\eps}](t,0)$ is of the form 
\begin{equation} \label{a-et-b}
D K_{\R^{2}}[\omega^{\eps}](t,0) = 
\begin{pmatrix} -a^\varepsilon & b^\varepsilon \\ b^\varepsilon & a^\varepsilon \end{pmatrix}
 \text{ for some } a^{\varepsilon}, b^{\varepsilon} \in \R,
\end{equation}
where we have by Lemma~\ref{LemModulesBornes}:
\begin{equation*}
\|a^\varepsilon \|_{L^\infty(0,T)}+\|b^\varepsilon \|_{L^\infty(0,T)}
 \leq C.
\end{equation*}
Note that in particular we have
\begin{equation*}
D K_{\R^{2}}[\omega^{\eps}](t,0) \cdot x= a^\eps \begin{pmatrix} -x_{1}\\x_{2} \end{pmatrix} + b^\eps \begin{pmatrix} x_{2}\\x_{1} \end{pmatrix}.
\end{equation*}
Therefore, reminding \eqref{allo}, we introduce 
\begin{equation}\label{Defvd2}
v^{\eps}_{\#}  (x):= K_{\R^{2}}[\omega^{\eps}](t,0) +  D K_{\R^{2}}[\omega^{\eps}](t,0) \cdot x 
+ \sum_{i=1}^2 (\ell^{\eps} - K_{\R^{2}}[\omega^{\eps}](t,0) )_{i} \nabla \Phi^{\eps}_{i}(x) - a^\eps \nabla \Phi^{\eps}_{4}(x)-b^\eps \nabla \Phi^{\eps}_{5}(x).
\end{equation}
The following proposition allows to obtain a good approximation of the fluid velocity on $\partial \mathcal{S}_0^\varepsilon$.
\begin{Proposition} \label{Pr}
The vector field  $v^{\eps}_{\#}-\ell^\eps$ is tangent to the boundary. Moreover, for fixed $\rho>1$, there exist $C>0$ and $\varepsilon_{0}\in (0,1]$ such that the following holds. If for a given $T>0$ and $\eps\in (0,\varepsilon_{0}]$, the inclusion \eqref{CondDistVorticiteSolide} is valid on $[0,T]$, then one has 
\begin{equation*} 
\| v^{\eps}_{\#} + r^{\eps} \nabla \Phi^{\eps}_{3} - \tilde{v}^{\eps} \|_{L^{\infty}(0,T;L^{2}(\partial {\mathcal S}_0^{\eps}))} \leq C\eps^{5/2}.
\end{equation*}
\end{Proposition}
\begin{proof}
The proof mimics the proof of Proposition 7 in  \cite{GLS} with the  more accurate approximation $v^{\eps}_{\#} $ of $ \tilde{v}^{\eps} - r^{\eps} \nabla \Phi^{\eps}_{3}$. 
Let us shortly sketch the proof for sake of completeness. We first introduce 
\begin{equation} \label{DefCheckV2}
\check{v}^{\eps}  := K_{\R^{2}}[\omega^{\eps}] + \sum_{i=1}^2 (\ell^{\eps} - K_{\R^{2}}[\omega^{\eps}](t,0) )_{i} \nabla \Phi^{\eps}_{i} + r^{\eps} \nabla \Phi^{\eps}_{3} - a^\eps \nabla \Phi^{\eps}_{4}-b^\eps\nabla \Phi^{\eps}_{5} ,
\end{equation}
where $a^\eps$ and $b^\eps$ come from \eqref{a-et-b}. We observe that 
\begin{equation} \label{CheckvTildev}
\left\{ \begin{array}{l} 
\curl( \check{v}^{\eps} - \tilde{v}^{\eps}) =0,  \quad   \text{for}  \ x\in  \mathcal{F}^{\eps}_{0} , \\ 
\div( \check{v}^{\eps} - \tilde{v}^{\eps}) =0,  \quad  \text{for}  \ x\in  \mathcal{F}^{\eps}_{0}  , \\
\int_{\partial {\mathcal S}_0^{\eps}} ( \check{v}^{\eps} - \tilde{v}^{\eps})\cdot \tau \, ds =0, \\
( \check{v}^{\eps} - \tilde{v}^{\eps}) \cdot n =  g^{\eps} , \quad   \text{for}  \ x\in \partial \mathcal{S}^{\eps}_0  , \\
\check{v}^{\eps} - \tilde{v}^{\eps} \rightarrow 0 \quad  \text{as}  \ x \rightarrow  \infty ,
\end{array} \right.
\end{equation}
with
\begin{equation*}
g^\eps := (\check{v}^{\eps} - \tilde{v}^{\eps}) \cdot n =  \Big(K_{\R^{2}}[\omega^{\eps}]  - K_{\R^{2}}[\omega^{\eps}](t,0)
-  D K_{\R^{2}}[\omega^{\eps}](t,0) \cdot x\Big)\cdot n . 
\end{equation*}
As in Lemma~\ref{LemModulesBornes}, we note that $K_{\R^{2}}[\omega^{\eps}]$ is harmonic on $B(0,1/\rho)$, so its $C^2$ norm on some smaller ball can be estimated by the $L^\infty$ norm on $B(0,1/\rho)$. Since \eqref{loglip2} provides in particular a uniform $L^{\infty}$ bound, this yields that for all $\varepsilon \in (0,\varepsilon_{0}]$:
\begin{equation} \label{harmonic2}
\| K_{\R^{2}}[\omega^{\eps}](t,\cdot) - K_{\R^{2}}[\omega^{\eps}](t,0) - D K_{\R^{2}}[\omega^{\eps}](t,0) \cdot (\cdot)\|_{L^{\infty}(0,T;L^{\infty}(\partial {\mathcal S}_0^{\eps}))} \leq C \eps^2,
\end{equation}
where $\varepsilon_{0} >0$ is chosen such that $\partial {\mathcal S}_0^{\eps_{0}}\subset \overline{B(0,1/(2\rho))}$.
\par
Now we will use the following classical lemma (see for instance \cite{Kikuchi83}).
\begin{Lemma} \label{LemmeElliptiqueDeBase}
There exists $C>0$ such that for any $g$ in $L^2 ( \partial \mathcal{S}_0 )$ satisfying 
\begin{equation*} 
\int_{\partial {\mathcal S}_0} g(s) ds = 0,
\end{equation*}
there is a unique solution $\Psi$ in $H^\frac{3}{2} ( \mathcal{F}_{0} )$ of 
\begin{equation*} 
\left\{ \begin{array}{l} 
\Delta \Psi = 0 ,  \quad   \text{for}  \ x\in  \mathcal{F}_{0} , \\ 
\partial_n  \Psi = g ,  \quad  \text{for}  \ x\in \partial \mathcal{S}_0   , \\
 \Psi \rightarrow 0 \quad  \text{as}  \ x \rightarrow  \infty .
\end{array} \right.
\end{equation*}
and 
\begin{equation*} 
 \| \Psi \|_{H^{3/2}({\mathcal F}_0)} \leq C \|  g  \|_{L^{2}( \partial  {\mathcal S}_0)} .
\end{equation*}
\end{Lemma}
The first and third equations of \eqref{CheckvTildev} imply that $\check{v}^{\eps} - \tilde{v}^{\eps}$ is a gradient, say $\check{v}^{\eps} - \tilde{v}^{\eps} = \nabla \chi$. 
Hence with a dilatation argument we can apply Lemma~\ref{LemmeElliptiqueDeBase} to $\chi$.
Using also the estimate \eqref{harmonic2} on $g^\eps$, we deduce the following lemma:
\begin{Lemma} \label{LemEstCheckvTildev2}
For fixed $\rho>1$, there exist $C>0$ and $\varepsilon_{0}\in (0,1]$ such that the following holds. If for a given $T>0$ and $\eps\in (0,\varepsilon_{0}]$, the inclusion \eqref{CondDistVorticiteSolide} is valid on $[0,T]$, then one has 
\begin{equation*} 
\| \check{v}^{\eps} - \tilde{v}^{\eps} \|_{L^{\infty}(0,T;L^{2}(\partial {\mathcal S}_0^{\eps}))} +
\eps^{-1/2}\| \check{v}^{\eps} - \tilde{v}^{\eps} \|_{L^{\infty}(0,T ; L^2({\mathcal F}_0^{\eps}))}
=  \mathcal{O}(\eps^{5/2}),
\end{equation*}
where $\tilde{v}^{\eps}$ and $\check{v}^{\eps}$ are defined in \eqref{allo} and \eqref{DefCheckV2}.
\end{Lemma}
To finish the proof of Proposition~\ref{Pr}, it remains to estimate
\begin{equation*}
v^{\eps}_{\#} + r^{\eps} \nabla \Phi^{\eps}_{3} - \check{v}^{\varepsilon}
= K_{\R^{2}}[\omega^{\eps}](t,0) +  D K_{\R^{2}}[\omega^{\eps}](t,0) \cdot x - K_{\R^{2}}[\omega^{\eps}](t,x)
\end{equation*}
on $\partial {\mathcal S}_{0}^{\varepsilon}$. Using again the harmonicity of $K_{\R^{2}}[\omega^{\eps}]$ in $B(0,1/\rho)$, we infer easily that
\begin{eqnarray*}
\| v^{\eps}_{\#} + r^{\eps} \nabla \Phi^{\eps}_{3} - \check{v}^{\varepsilon} \|_{L^{\infty}(0,T;L^{\infty}(\partial {\mathcal S}^{\varepsilon}_{0}))}
& \leq & C \varepsilon^{2} \| \nabla^{2} K_{\R^{2}}[\omega^{\eps}] \|_{L^{\infty}({\mathcal S}^{\varepsilon}_{0})} \\
& \leq & C \varepsilon^{2} \| K_{\R^{2}}[\omega^{\eps}] \|_{L^{\infty}(B(0,1/\rho))} ,
\end{eqnarray*}
for $\varepsilon\leq \varepsilon_{0}$. With \eqref{loglip2}, integrating over $\partial {\mathcal S}_{0}^{\varepsilon}$, we reach the conclusion.
\end{proof} 
%
%
%
%
%
%
%
%
%
%
\section{Normal form. Proof of Proposition~\ref{Pro:NormalForm}}
\label{Sec:NormalForm}
The goal of this section is to establish Proposition~\ref{Pro:NormalForm}. \par
The following notations will be used in this section: $|\mathcal{S}^\eps_{0} |$ is the Lebesgue measure of $\mathcal{S}^\eps_{0}$,
$x_{G}^\eps$ is the position of the geometrical center of $\mathcal{S}^\eps_{0}$ (which can be different of the center of mass $0$ if the solid is not homogenous):
\begin{equation} \label{DefXg}
x_{G}^\eps := \frac{1}{|{\mathcal S}_{0}^{\varepsilon}|} \int_{{\mathcal S}_{0}^{\varepsilon}} x \, dx=\varepsilon x_{G} .
\end{equation}
The following formula for the vector product will be useful later on in some computations:
\begin{equation} \label{vprod}
\forall  p_a := (\ell_a, \omega_a), \ \forall p_b := ({\ell}_b, \omega_b) \ \text{ in  } \  \R^{2}  \times  \R , \quad 
p_a \times  {p}_b = (  \, \omega_a  \;  {\ell}^{\perp}_b -    {\omega}_b \;   \ell^{\perp}_a ,  \ell^{\perp}_a \cdot {\ell}_b) .
\end{equation}
We will frequently use the complex variable and the correspondence between $\R^{2}$ and $\C$ as described at the beginning of Section~\ref{Sec:Material}. The proofs of many technical lemmas of complex analysis used in this section are given in Appendix~\ref{complex ana}. \par
\subsection{Decomposition of the pressure}
We first reformulate the main solid equations \eqref{Solide11}-\eqref{Solide12}. Recall that $\tilde{v}^{\varepsilon}$ and ${\mathcal M}^{\varepsilon}$ were introduced in \eqref{allo} and \eqref{InertieMatrix} respectively.
\begin{Lemma} \label{ReformulationEqSolide}
Equations \eqref{Solide11}-\eqref{Solide12} can be rewritten in the form
\begin{equation} \label{EqSolideBis}
{\mathcal M}^{\varepsilon} \begin{pmatrix} \ell^{\varepsilon} \\ r^{\varepsilon }\end{pmatrix}'
= - (B_i^\eps)_{i=1,2,3} - (C_i^\eps)_{i=1,2,3} - \begin{pmatrix} m^\varepsilon r^{\varepsilon} (\ell^{\varepsilon})^{\perp} \\ 0\end{pmatrix},
\end{equation}
where for $i=1,2,3$, 
\begin{equation}
\label{DefBi}
B_i^\eps := \int_{\mathcal{F}^\eps_{0}} \omega^{\eps}[ v^{\eps}-\ell^{\eps}-r^{\eps} x^\perp ]^\perp \cdot \nabla \Phi^\eps_i (x) \, dx, 
\end{equation}
and 
\begin{equation*}
C_i^\eps := C_{i,a}^{\varepsilon} + C_{i,b}^{\varepsilon} + C_{i,c}^{\varepsilon}, 
\end{equation*}
with
\begin{eqnarray}
\label{Cia}
C_{i,a}^\eps &:=& \frac{1}{2} \int_{\partial  \mathcal{S}^\eps_{0}  } |\tilde{v}^{\eps}|^2  K_i \, ds
- \int_{\partial  \mathcal{S}^\eps_{0}  } (\ell^{\eps} + r^{\eps} x^\perp)\cdot \tilde{v}^{\eps}  K_i \, ds, \\ 
\label{Cib}
C_{i,b}^\eps &:=& \gamma  \int_{\partial  \mathcal{S}^\eps_{0}  } (\tilde{v}^{\eps} - (\ell^{\eps} + r^{\eps} x^\perp))\cdot  H^{\eps}  K_i \, ds, \\
\label{Cic}
C_{i,c}^\eps &:=& \frac{\gamma^2}{2}   \int_{\partial  \mathcal{S}^\eps_{0} } |H^{\eps}|^2  K_i \, ds.
\end{eqnarray}
\end{Lemma}
\begin{proof}
The proof, which we reproduce for the sake of self-containedness, is mainly the same as in \cite{GLS} (though the decomposition is a bit different here). Using the following equality for two vector fields $a$ and $b$ in a domain of the plane:
\begin{equation} \label{vect}
\nabla(a\cdot b)=a\cdot \nabla b + b \cdot \nabla a - (a^\perp \curl b + b^\perp \curl a), 
\end{equation}
the equation \eqref{Euler11} can be written as
\begin{equation*}
\frac{\partial v^{\eps}}{\partial t}+
[ v^{\eps}-\ell^{\eps}-r^{\eps} x^\perp ]^\perp \omega^{\eps}
+ \nabla \frac{1}{2} (v^{\eps})^2 
- \nabla ( (\ell^{\eps} + r^{\eps} x^\perp)\cdot v^{\eps} )
+ \nabla \tilde{\pi}^{\eps} =0 .
\end{equation*}
Plugging the decomposition \eqref{allo} into the previous equation, we find
\begin{gather*}
\frac{\partial v^{\eps}}{\partial t}+ [ v^{\eps}-\ell^{\eps}-r^{\eps} x^\perp ]^\perp \omega^{\eps}
+ \nabla ({\mathcal Q}^{\eps}  + \tilde{\pi}^{\eps} )= 0 , \\
{\mathcal Q}^{\eps} :=  \frac{1}{2} |\tilde{v}^{\eps}|^2  + \gamma  (\tilde{v}^{\eps} 
- (\ell^{\eps} + r^{\eps} x^\perp))\cdot  H^{\eps} + \frac{1}{2} \gamma^2 |H^{\eps} |^2 
- (\ell^{\eps} + r^{\eps} x^\perp)\cdot \tilde{v}^{\eps}  .
\end{gather*}
We use this equation do deduce the force/torque acting on the body:
\begin{equation*}
\left( \int_{ \partial \mathcal{S}^\eps_0} \tilde{\pi}^\eps n \, ds, \
\int_{ \partial \mathcal{S}^\eps_0} \tilde{\pi}^\eps x^{\perp} \cdot n \,ds \right) = 
\left( \int_{ \mathcal{F}^\eps_0} \nabla \tilde{\pi}^\eps \cdot \nabla \Phi_{i}^{\varepsilon} \, dx \right)_{i=1,2,3} .
\end{equation*}
One can check that the above integration by parts is licit by using the compact support of $\omega$ and the decay properties of $\tilde{v}^{\varepsilon}$, $H^{\varepsilon}$ and $\nabla \Phi^{\varepsilon}_{i}$ (see \cite{GLS} for more details). Using Green's formula and the boundary condition, the contribution of $\frac{\partial v^{\eps}}{\partial t}$ is
\begin{equation*}
\Big(\int_{  \mathcal{F}^\eps_{0}  }  \partial_{t}  v^\eps  \cdot  \nabla   \Phi^\eps_i (x) \, dx\Big)_{i=1,2,3} = {\mathcal M}_{a}^{\varepsilon} \begin{pmatrix} \ell^{\varepsilon} \\ r^{\varepsilon }\end{pmatrix}', 
\end{equation*}
and one obtains the result.
\end{proof}
\ \par
In the next subsections, we expand the various terms $B^{\varepsilon}_{i}$ and $C^{\varepsilon}_{i}$ up to order $2$ (respectively $3$) in $\varepsilon$ for $i=1,2$ (resp. $3$). Then in Subsection~\ref{Subsec:Regroupement}, we regroup these terms and get to the normal form. We will take advantages of {\it cancellations} when performing this merging. \par
%
%
%
%
%
%
%
%
\subsection{Expansion of $B_i^\eps$}
We begin by developing $B_i^\eps$ under the assumptions of Proposition~\ref{Pro:NormalForm}. \par
\begin{Proposition} \label{LimBi}
Let $\rho>1$ be fixed. There exists $C>0$ such that if for a given $T>0$ and an $\varepsilon \in (0,1]$, \eqref{CondDistVorticiteSolide} is satisfied for all $t \in [0,T]$, then one has:
\begin{equation*}
\begin{split}
\left\| 
\begin{pmatrix}B_{1}^{\varepsilon} \\B_{2}^{\varepsilon} \end{pmatrix}  \right.
& - r^\eps(\mathcal{M}_{\flat}^\eps + |\mathcal{S}^\eps_{0}| {\rm I}_{2}) (K_{\R^{2}}[\omega^{\eps}](t,0) )^\perp 
\\
&
+ \ell_{1}^\eps 
\begin{pmatrix} 
-(m_{1,1}^{\eps}+ |\mathcal{S}^\eps_{0} |) a^\eps +m_{1,2}^{\eps} b^\eps\\
 -m_{2,1}^{\eps} a^\eps +(m_{2,2}^{\eps}+ |\mathcal{S}^\eps_{0} |)  b^\eps
 \end{pmatrix} 
+ \ell_{2}^\eps 
\begin{pmatrix} 
 m_{1,2}^{\eps} a^\eps +(m_{1,1}^{\eps}+ |\mathcal{S}^\eps_{0} |)  b^\eps \\
(m_{2,2}^{\eps}+ |\mathcal{S}^\eps_{0} |) a^\eps +m_{2,1}^{\eps} b^\eps
 \end{pmatrix}
  \\
&\left.
+ 2r^\eps a^\eps 
\begin{pmatrix} 
m_{1,5}^{\eps}+ |\mathcal{S}^\eps_{0} | x_{G,2}^\eps  \\
m_{2,5}^{\eps}+ |\mathcal{S}^\eps_{0} | x_{G,1}^\eps
 \end{pmatrix} 
+ 2r^\eps b^\eps 
\begin{pmatrix} 
-m_{1,4}^{\eps}+ |\mathcal{S}^\eps_{0} | x_{G,1}^\eps  \\
-m_{2,4}^{\eps}- |\mathcal{S}^\eps_{0} | x_{G,2}^\eps
 \end{pmatrix}
\right\|_{L^\infty(0,T)}
\leq C \varepsilon^{2} 
\end{split}
\end{equation*}
and
\begin{equation*} 
\begin{split}
\Bigg\| B_{3}^{\varepsilon}  - r^\eps \Big(\begin{pmatrix} m_{3,2}^{\eps} \\-m_{3,1}^{\eps} \end{pmatrix}  
 +& |\mathcal{S}^\eps_{0} |x_{G}^\eps\Big)  \cdot K_{\R^{2}}[\omega^{\eps}](t,0)\\
&+ \ell_{1}^\eps 
\Big( (-m_{3,1}^{\eps}+ |\mathcal{S}^\eps_{0} |x_{G,2}^\eps  ) a^\eps  +(m_{3,2}^{\eps}+ |\mathcal{S}^\eps_{0} |x_{G,1}^\eps )  b^\eps\Big)\\
&+ \ell_{2}^\eps 
\Big( (m_{3,2}^{\eps}+ |\mathcal{S}^\eps_{0} |x_{G,1}^\eps  ) a^\eps  +(m_{3,1}^{\eps}- |\mathcal{S}^\eps_{0} |x_{G,2}^\eps )  b^\eps\Big)\\
&+ 2r^\eps a^\eps (m_{3,5}^{\eps}+m_{6}^{\eps})
- 2r^\eps b^\eps (m_{3,4}^{\eps}+m_{7}^{\eps})
 \Bigg\|_{L^\infty(0,T)}  
\leq C \varepsilon^{3},
 \end{split}
\end{equation*}
where
\begin{itemize}
\item $m_{i,j}^{\eps}$ is defined in \eqref{DefMasse},
\item $\mathcal{M}_{\flat}^\eps$ is defined in \eqref{DefMbemol} as the $2\times2$ restriction of $\mathcal{M}_{a}^\eps$,
\item $a^{\varepsilon}$ and $b^{\varepsilon}$ are defined in \eqref{a-et-b} as coefficients in $D K_{\R^{2}}[\omega^{\eps}](t,0)$,
\item $m_6^{\eps}=\int_{ \mathcal{S}^\eps_{0}} (x_1^2-x_{2}^2 ) \, dx= \eps^4 m_{6}^1$ and $m_{7}^{\eps}=2\int_{ \mathcal{S}^\eps_{0}} x_{1}x_2  \, dx= \eps^4 m_{7}^1$.
\end{itemize}
\end{Proposition}
\begin{proof}
We decompose $B_i^\varepsilon$ as follows:
\begin{equation*}\begin{split}
B_{i}^\eps &= \int_{  \mathcal{F}^\eps_{0}  } \omega^{\eps} (v^{\eps})^\perp \cdot  \nabla   \Phi^\eps_i (x) \, dx -\int_{  \mathcal{F}^\eps_{0}  } \omega^{\eps}( \ell^{\eps} )^\perp \cdot  \nabla   \Phi^\eps_i (x) \, dx + \int_{  \mathcal{F}^\eps_{0}  } \omega^{\eps} r^\eps x \cdot  \nabla   \Phi^\eps_i (x) \, dx\\
&=: B_{a,i}^\eps + B_{b,i}^\eps + B_{c,i}^\eps.
\end{split}\end{equation*}
\ \par
\noindent
$\bullet$
According to \eqref{ConsOmega}, Lemmas~\ref{LemPhiomega} and \ref{LeminfVomega}, we estimate the first term for any $t\in [0,T]$:
\begin{eqnarray*}
|B_{a,i}^\eps(t)|&\leq&  \| \omega^{\eps}\|_{L^{1}(\mathcal{F}^{\eps}_{0})}  \| v^{\eps} \|_{L^{\infty}(\supp \omega^\varepsilon(t)) }
\|  \nabla   \Phi^\eps_i  \|_{L^{\infty}(\supp \omega^\varepsilon(t)) } \\
& \leq & C\eps^{2+ \delta_{3,i}}\| w_{0} \|_{ L^{1}(\R^{2})}.
\end{eqnarray*}
\ \par
\noindent
$\bullet$
Concerning the computation of $B_{b,i}^\eps$, we use that the property \eqref{CondDistVorticiteSolide} is valid during $[0,T]$, that is, the support of $\omega^\eps$ is included in $B(0,\rho)\setminus B(0,1/\rho)$ for any $t\in [0,T]$. Taking the scaling law \eqref{phi-scaling} into account, we are naturally led to study the asymptotic behavior of $\Phi_{i}^1(z)$ as $z \rightarrow \infty$.
Thanks to Lemmas~\ref{LemmeCircu}, \ref{lemzphi} and \ref{lemz2phi}, we can write the first terms in the Laurent series:
\begin{equation*}
\widehat{\nabla \Phi_{1}^1}(z)  = \frac{-m_{1,2}^{1}+i(m_{1,1}^{1}+|\mathcal{S}_{0} |)}{2i\pi z^2}
+ \frac{-2(m_{1,5}	 +  |\mathcal{S}_{0} |x_{G,2} )+2i(-m_{1,4}^{1} +  |\mathcal{S}_{0} |x_{G,1} )}{2i\pi z^3} +\mathcal{O}\Big(\frac1{z^4}\Big),
\end{equation*}
\begin{equation*}
\widehat{\nabla \Phi_{2}^1}(z)  = \frac{-(m_{2,2}^{1}+|\mathcal{S}_{0} |)+i m_{2,1}^{1}}{2i\pi z^2} 
+\frac{-2(m_{2,5}^{1} +  |\mathcal{S}_{0} |x_{G,1} )-2i(m_{2,4}^{1} +  |\mathcal{S}_{0} |x_{G,2} )}{2i\pi z^3}+\mathcal{O}\Big(\frac1{z^4}\Big),
\end{equation*}
\begin{equation*}
\widehat{\nabla \Phi_{3}^1}(z) =  \frac{-(m_{3,2}^{1}+|\mathcal{S}_{0} |x_{G,1}) + i (m_{3,1}^{1} - |\mathcal{S}_{0}| x_{G,2})}{2i\pi z^2}
+ \frac{-2(m_{3,5}^{1} +m_{6}^{1}  )-2i (m_{3,4}^{1} + m_{7}^{1} )}{2i\pi z^3} + \mathcal{O}\Big(\frac1{z^4}\Big).
\end{equation*}
Using the notation for polar coordinates
\begin{equation} \label{CoordPolaires}
y = \begin{pmatrix} y_{1} \\ y_{2} \end{pmatrix} = |y|  \begin{pmatrix} \cos \theta \\ \sin \theta \end{pmatrix}  \ \text{ and } \ 
z=y_{1}+iy_{2}= |y|(\cos \theta +i \sin \theta),
\end{equation} 
we note that 
\begin{equation*}
\dfrac1{z^k}=\dfrac1{|y|^k}\widehat{\begin{pmatrix}\cos k\theta \\ \sin k\theta\end{pmatrix}} \ \text{ and } \ \dfrac{i}{z^k}=\dfrac1{|y|^k}\widehat{\begin{pmatrix}\sin k\theta \\ -\cos k\theta\end{pmatrix}}. 
\end{equation*}
Hence, we use this expansion with $y=x/\eps$ together with the scaling laws of $\Phi_{i}^\eps$ (see \eqref{phi-scaling}) and of $m_{i,j}^{\eps}$ to obtain:
\begin{eqnarray*}
\nabla \Phi_{1}^\eps(x)  &=&
\frac{m_{1,2}^{\eps}}{2\pi |x|^2}{\begin{pmatrix}\sin 2\theta \\ -\cos 2\theta\end{pmatrix}}
+\frac{m_{1,1}^{\eps}+|\mathcal{S}_{0}^\eps |}{2\pi |x|^2}{\begin{pmatrix}\cos 2\theta \\ \sin 2\theta\end{pmatrix}}\\
&&+\frac{2(m_{1,5}^{\eps} +  |\mathcal{S}_{0}^\eps |x_{G,2}^\eps )}{2\pi |x|^3}{\begin{pmatrix}\sin 3\theta \\ -\cos 3\theta\end{pmatrix}}
+\frac{2(-m_{1,4}^{\eps} +  |\mathcal{S}_{0}^\eps |x_{G,1}^\eps )}{2\pi |x|^3}{\begin{pmatrix}\cos 3\theta \\ \sin 3\theta\end{pmatrix}}
 +\mathcal{O}\Big(\frac{\eps^4}{|x|^4}\Big),
 \end{eqnarray*}
\begin{eqnarray*}
\nabla \Phi_{2}^\eps(x)   &=&
\frac{m_{2,2}^{\eps}+|\mathcal{S}_{0}^\eps |}{2\pi |x|^2}{\begin{pmatrix}\sin 2\theta \\ -\cos 2\theta\end{pmatrix}}
+ \frac{m_{2,1}^{\eps}}{2\pi |x|^2}{\begin{pmatrix}\cos 2\theta \\ \sin 2\theta\end{pmatrix}}\\
&&+ \frac{2(m_{2,5}^{\eps} +  |\mathcal{S}_{0}^\eps |x_{G,1}^\eps )}{2\pi |x|^3}{\begin{pmatrix}\sin 3\theta \\ -\cos 3\theta\end{pmatrix}}
-\frac{2(m_{2,4}^{\eps} +  |\mathcal{S}_{0}^\eps |x_{G,2}^\eps )}{2\pi |x|^3}{\begin{pmatrix}\cos 3\theta \\ \sin 3\theta\end{pmatrix}}
 +\mathcal{O}\Big(\frac{\eps^4}{|x|^4}\Big),
 \end{eqnarray*}
and
\begin{eqnarray*}
\nabla \Phi_{3}^\eps(x) &=&
\frac{m_{3,2}^{\eps}+|\mathcal{S}_{0}^\eps |x_{G,1}^\eps}{2\pi |x|^2}{\begin{pmatrix}\sin 2\theta \\ -\cos 2\theta\end{pmatrix}}
+\frac{m_{3,1}^{\eps}-|\mathcal{S}_{0}^\eps |x_{G,2}^\eps}{2\pi |x|^2}{\begin{pmatrix}\cos 2\theta \\ \sin 2\theta\end{pmatrix}}\\
&&+\frac{2(m_{3,5}^{\eps} +m_{6}^{\eps}  )}{2\pi |x|^3}{\begin{pmatrix}\sin 3\theta \\ -\cos 3\theta\end{pmatrix}}
-\frac{2(m_{3,4}^{\eps} +m_{7}^{\eps} ) }{2\pi |x|^3}{\begin{pmatrix}\cos 3\theta \\ \sin 3\theta\end{pmatrix}}
 +\mathcal{O}\Big(\frac{\eps^5}{|x|^4}\Big).
\end{eqnarray*}
For the sake of simplicity of notations, we will denote during this proof
\begin{equation} \label{frakK}
(\mathfrak{K}_{i})_{i=1,2}=K_{\R^2}[\omega^\varepsilon](t,0) =\frac1{2\pi} \int_{\R^2} \frac{-y^\perp}{|y|^2}\omega^\varepsilon(t,y)\, dy.
\end{equation}
Hence we identify the components of $K_{\R^2}[\omega^\varepsilon](t,0)$ as follows:
\begin{equation*}
\mathfrak{K}_{1}=\frac1{2\pi} \int_{\R^2} \frac{\sin \theta}{|y|}\omega^\varepsilon(t,y)\, dy, \quad 
\mathfrak{K}_{2}=\frac1{2\pi} \int_{\R^2} \frac{-\cos \theta}{|y|}\omega^\varepsilon(t,y)\, dy,
\end{equation*}
where we used again the notation \eqref{CoordPolaires} for $y$. With $a^\eps$ and $b^\eps$ defined in \eqref{a-et-b}, we have
\begin{equation*}
a^\varepsilon = \partial_{2} K_{\R^2}[\omega^\varepsilon]_{2}(t,0)=-\frac1{2\pi} \int_{\R^2} \frac{2y_{1}y_{2}}{|y|^4}\omega^\varepsilon(t,y)\, dy=-\frac1{2\pi} \int_{\R^2} \frac{\sin 2\theta}{|y|^2}\omega^\varepsilon(t,y)\, dy
\end{equation*}
and
\begin{equation*}
b^\varepsilon = \partial_{1} K_{\R^2}[\omega^\varepsilon]_{2}(t,0)=\frac1{2\pi} \int_{\R^2} \frac{|y|^2-2y_{1}^2}{|y|^4}\omega^\varepsilon(t,y)\, dy=-\frac1{2\pi} \int_{\R^2} \frac{\cos 2\theta}{|y|^2}\omega^\varepsilon(t,y)\, dy.
\end{equation*}
\ \par
Now we use the expansion of $\nabla \Phi_{1}^\eps(x)$ at order one to compute $B_{b,1}^\eps$:
\begin{align*} 
B_{b,1}^\eps(t) &= - \int_{  \mathcal{F}^\eps_{0}  } \omega^{\eps}( \ell^{\eps} )^\perp \cdot  \nabla   \Phi^\eps_i (x) \, dx
= \int_{  \mathcal{F}^\eps_{0}  } \omega^{\eps} \begin{pmatrix} \ell_{2}^\eps\\ - \ell_{1}^\eps \end{pmatrix}
 \cdot    \Big[ \frac{m_{1,2}^{\eps}}{2\pi |x|^2}{\begin{pmatrix}\sin 2\theta \\ -\cos 2\theta\end{pmatrix}}
+\frac{m_{1,1}^{\eps}+|\mathcal{S}_{0}^\eps |}{2\pi |x|^2}{\begin{pmatrix}\cos 2\theta \\ \sin 2\theta\end{pmatrix}} +\mathcal{O}\Big(\frac{\eps^3}{|x|^3}\Big) \Big] \, dx\\
&= \ell_{2}^\eps \Big(-m_{1,2}^{\eps} a^\eps - (m_{1,1}^{\eps}+|\mathcal{S}_{0}^\eps |) b^\eps\Big) 
+ \ell_{1}^\eps \Big(-m_{1,2}^{\eps} b^\eps + (m_{1,1}^{\eps}+|\mathcal{S}_{0}^\eps |) a^\eps\Big)+\tilde R_{b,1}(t)
\end{align*}
where 
\begin{equation*}
|\tilde R_{b,1}^\eps(t)| \leq \eps^3 \|\omega^\eps \|_{L^1} |\ell^\eps | C\rho^3\leq C \eps^2,
\end{equation*}
for all $t\in [0,T]$.
\par
In the same way, we obtain
\begin{equation*} \begin{split}
B_{b,2}^\eps(t) =&
\ell_{2}^\eps \Big(-(m_{2,2}^{\eps}+|\mathcal{S}_{0}^\eps |) a^\eps - m_{2,1}^{\eps} b^\eps\Big) 
+ \ell_{1}^\eps \Big(-(m_{2,2}^{\eps}+|\mathcal{S}_{0}^\eps |) b^\eps + m_{2,1}^{\eps} a^\eps\Big)+\tilde R_{b,2}(t)
\end{split} \end{equation*}
and
\begin{multline*} 
B_{b,3}^\eps(t) =
\ell_{2}^\eps \Big(-(m_{3,2}^{\eps}+|\mathcal{S}_{0}^\eps |x_{G,1}^\eps) a^\eps - (m_{3,1}^{\eps}-|\mathcal{S}_{0}^\eps |x_{G,2}^\eps) b^\eps\Big) \\ 
+ \ell_{1}^\eps \Big(-(m_{3,2}^{\eps}+|\mathcal{S}_{0}^\eps |x_{G,1}^\eps) b^\eps + (m_{3,1}^{\eps}-|\mathcal{S}_{0}^\eps |x_{G,2}^\eps) a^\eps\Big)
+\tilde R_{b,3}(t)
\end{multline*}
where 
\begin{equation*}
|\tilde R_{b,2}^\eps(t)| \leq C \eps^2 \ \text{ and } \
|\tilde R_{b,3}^\eps(t)| \leq C \eps^3 \ \text{ for all }  \ t \in [0,T].
\end{equation*}
\ \par
\noindent
$\bullet$
Concerning $B_{c,i}^\eps$, we use the second-order expansion of $\nabla \Phi_{i}^\eps$ to get:
\begin{equation*} \begin{split}
B_{c,1}^\eps(t) =& \int_{  \mathcal{F}^\eps_{0}  } \omega^{\eps} r^\eps x \cdot  \nabla   \Phi^\eps_1 (x) \, dx\\
=& \int_{  \mathcal{F}^\eps_{0}  } \omega^{\eps} r^\eps |x| \begin{pmatrix}\cos\theta \\ \sin\theta \end{pmatrix} \cdot
\Big[
\frac{m_{1,2}^{\eps}}{2\pi |x|^2}{\begin{pmatrix}\sin 2\theta \\ -\cos 2\theta\end{pmatrix}}
+\frac{m_{1,1}^{\eps}+|\mathcal{S}_{0}^\eps |}{2\pi |x|^2}{\begin{pmatrix}\cos 2\theta \\ \sin 2\theta\end{pmatrix}}
\Big]\, dx\\
&+ \int_{  \mathcal{F}^\eps_{0}  } \omega^{\eps} r^\eps |x| \begin{pmatrix}\cos\theta \\ \sin\theta \end{pmatrix} \cdot
\Big[
\frac{2(m_{1,5}^{\eps} +  |\mathcal{S}_{0}^\eps |x_{G,2}^\eps )}{2\pi |x|^3}{\begin{pmatrix}\sin 3\theta \\ -\cos 3\theta\end{pmatrix}}
+ \frac{2(-m_{1,4}^{\eps} +  |\mathcal{S}_{0}^\eps |x_{G,1}^\eps )}{2\pi |x|^3}{\begin{pmatrix}\cos 3\theta \\ \sin 3\theta\end{pmatrix}}
+ \mathcal{O}\Big(\frac{\eps^4}{|x|^4}\Big)
\Big]\, dx ,
\end{split} \end{equation*}
which is simplified as follows:
\begin{equation*}\begin{split}
B_{c,1}^\eps(t)
=& \int_{  \mathcal{F}^\eps_{0}  } \omega^{\eps} r^\eps 
\Big[
\frac{m_{1,2}^{\eps}}{2\pi |x|}\sin \theta
+\frac{m_{1,1}^{\eps}+|\mathcal{S}_{0}^\eps |}{2\pi |x|}\cos\theta
\Big]\, dx\\
&+ \int_{  \mathcal{F}^\eps_{0}  } \omega^{\eps} r^\eps 
\Big[
\frac{2(m_{1,5}^{\eps} +  |\mathcal{S}_{0}^\eps |x_{G,2}^\eps )}{2\pi |x|^2}\sin 2\theta
+\frac{2(-m_{1,4}^{\eps} +  |\mathcal{S}_{0}^\eps |x_{G,1}^\eps )}{2\pi |x|^2}\cos 2\theta
+\mathcal{O}\Big(\frac{\eps^4}{|x|^4}\Big)
\Big]\, dx\\
=& r^\eps 
\Big[
m_{1,2}^{\eps} \mathfrak{K}_{1}
-(m_{1,1}^{\eps}+|\mathcal{S}_{0}^\eps |) \mathfrak{K}_{2}\Big]\\
&+ 2 r^\eps 
\Big[
-(m_{1,5}^{\eps} +  |\mathcal{S}_{0}^\eps |x_{G,2}^\eps )a^\eps
-(-m_{1,4}^{\eps} +  |\mathcal{S}_{0}^\eps |x_{G,1}^\eps )b^\eps
\Big]
+\tilde R_{c,1}(t)
\end{split}\end{equation*}
where
\begin{equation*}
|\tilde R_{c,1}^\eps(t)| \leq \eps^4 \|\omega^\eps \|_{L^1} |r^\eps | C\rho^4 \leq C \eps^2,
\end{equation*}
for all $t\in [0,T]$. \par
In the same way, we obtain
\begin{equation*}
B_{c,2}^\eps(t) 
= r^\eps \Big[ (m_{2,2}^{\eps}+|\mathcal{S}_{0}^\eps |) \mathfrak{K}_{1} -m_{2,1}^{\eps} \mathfrak{K}_{2}\Big] 
+ 2 r^\eps  \Big[ -(m_{2,5}^{\eps} +  |\mathcal{S}_{0}^\eps |x_{G,1}^\eps )a^\eps +(m_{2,4}^{\eps} +  |\mathcal{S}_{0}^\eps |x_{G,2}^\eps )b^\eps \Big]
+\tilde R_{c,2}(t) 
\end{equation*}
and
\begin{equation*}
B_{c,3}^\eps(t) = r^\eps  \Big[ (m_{3,2}^{\eps}+|\mathcal{S}_{0}^\eps |x_{G,1}^\eps)\mathfrak{K}_{1} -(m_{3,1}^{\eps}-|\mathcal{S}_{0}^\eps |x_{G,2}^\eps) \mathfrak{K}_{2}\Big]\\
+ 2 r^\eps  \Big[ -(m_{3,5}^{\eps} +m_{6}^{\eps} )a^\eps +(m_{3,4}^{\eps} +m_{7}^{\eps} )b^\eps \Big]
+\tilde R_{c,3}(t)
\end{equation*}
where 
\begin{equation*}
|\tilde R_{c,2}^\eps(t)| \leq C \eps^2
\ \text{ and } \ 
|\tilde R_{c,3}^\eps(t)| \leq C \eps^3 \ \text{ for all } \ t\in [0,T]. 
\end{equation*}
Now we end the proof summing and noticing that
\begin{equation*}
\begin{pmatrix} 
m_{1,2}^{\eps} \mathfrak{K}_{1}
- (m_{1,1}^{\eps}+|\mathcal{S}_{0}^\eps |) \mathfrak{K}_{2} \\
(m_{2,2}^{\eps}+|\mathcal{S}_{0}^\eps |) \mathfrak{K}_{1}
- m_{2,1}^{\eps} \mathfrak{K}_{2} 
\end{pmatrix} 
=  (\mathcal{M}_{\flat}^\eps + |\mathcal{S}^\eps_{0} | {\rm I}_{2})(K_{\R^{2}}[\omega^{\eps}](t,0)  )^\perp 
\end{equation*}
and
\begin{equation*}
(m_{3,2}^{\eps}+|\mathcal{S}_{0}^\eps |x_{G,1}^\eps) \mathfrak{K}_{1}
- (m_{3,1}^{\eps}-|\mathcal{S}_{0}^\eps |x_{G,2}^\eps) \mathfrak{K}_{2}
= \Big(\begin{pmatrix} m_{3,2}^{\eps} \\ -m_{3,1}^{\eps} \end{pmatrix} 
+ |\mathcal{S}^\eps_{0} |x_{G}^\eps\Big) \cdot K_{\R^{2}}[\omega^{\eps}](t,0) .
\end{equation*}
\end{proof}
%
%
%
%
%
%
%
%
\subsection{Expansion of  $C_{i}^\eps$}
Concerning the term $C_{i}^\eps$, we use again complex analysis. \par
\ \par
\noindent
{\bf 1. The $C_{i,a}^\eps$ term.} We first tackle the terms $C_{i,a}^\eps$ (still under the assumptions of Proposition~\ref{Pro:NormalForm}). \par
\begin{Proposition} \label{PropCiad}
Let $\rho>1$ be fixed. There exist $C>0$ and $\varepsilon_{0}\in (0,1]$ such that if for a given $T>0$ and an $\varepsilon \in (0,\varepsilon_{0}]$, \eqref{CondDistVorticiteSolide} is satisfied for all $t \in [0,T]$, then one has:
\begin{equation*} 
\begin{split}
\Bigg\|  \begin{pmatrix}C_{1,a}^{\varepsilon} \\C_{2,a}^{\varepsilon} \end{pmatrix} 
& - (r^\eps)^2  \begin{pmatrix}-m_{3,2}^{\eps}\\m_{3,1}^{\eps}\end{pmatrix} 
 +r^\eps\Big(\mathcal{M}_{\flat}^\eps (K_{\R^{2}}[\omega^{\eps}](t,0)-\ell^{\eps}) + |\mathcal{S}^\eps_{0} | K_{\R^{2}}[\omega^{\eps}](t,0)    \Big)^\perp
 \\
&
+ \ell_{1}^\eps 
\begin{pmatrix} 
(m_{1,1}^{\eps}+ |\mathcal{S}^\eps_{0} |) a^\eps -m_{1,2}^{\eps} b^\eps\\
 -m_{1,2}^{\eps} a^\eps -(m_{1,1}^{\eps}+ |\mathcal{S}^\eps_{0} |)  b^\eps
 \end{pmatrix} 
+ \ell_{2}^\eps 
\begin{pmatrix} 
 m_{2,1}^{\eps} a^\eps -(m_{2,2}^{\eps}+ |\mathcal{S}^\eps_{0} |)  b^\eps \\
-(m_{2,2}^{\eps}+ |\mathcal{S}^\eps_{0} |) a^\eps -m_{2,1}^{\eps} b^\eps
 \end{pmatrix}
  \\
&- a^\eps r^\eps
\begin{pmatrix}
 -m_{3,1}^{\eps}+m_{4,2}^{\eps} + 2 |\mathcal{S}^\eps_{0} |x_{G,2}^\eps\\
 m_{3,2}^{\eps}-m_{4,1}^{\eps} +2 |\mathcal{S}^\eps_{0} |x_{G,1}^\eps
\end{pmatrix}
-b^\eps r^\eps
\begin{pmatrix}
 m_{3,2}^{\eps}+m_{5,2}^{\eps} + 2 |\mathcal{S}^\eps_{0} |x_{G,1}^\eps\\
 m_{3,1}^{\eps}-m_{5,1}^{\eps} - 2 |\mathcal{S}^\eps_{0} |x_{G,2}^\eps
\end{pmatrix}
 \Bigg\|_{L^\infty(0,T)}  
 \leq C \varepsilon^{2}
 \end{split}
\end{equation*}
and
\begin{equation*} 
\begin{split}
\Big\|C_{3,a}^{\varepsilon} &
-(K_{\R^{2}}[\omega^{\eps}](t,0)-\ell^\eps)^\perp \mathcal{M}_{\flat}^\eps (K_{\R^{2}}[\omega^{\eps}](t,0)-\ell^\eps)\\
&-r^\eps({K_{\R^{2}}[\omega^{\eps}](t,0)-\ell^{\eps} }) \cdot \begin{pmatrix} -m_{3,2}^{\eps} \\m_{3,1}^{\eps} \end{pmatrix} 
+r^\eps  |\mathcal{S}^\eps_{0}| K_{\R^{2}}[\omega^{\eps}](t,0) \cdot  x_{G}^\eps\\
&- \ell_{1}^\eps \Big( a^\eps ( -m_{4,2}^{\eps}+|\mathcal{S}^\eps_{0} |x_{G,2}^\eps+2 m_{1,5}^{\eps}) + b^\eps (   -m_{5,2}^{\eps}+|\mathcal{S}^\eps_{0} |x_{G,1}^\eps-2m_{1,4}^{\eps}) \Big)\\
&- \ell_{2}^\eps \Big( a^\eps (m_{4,1}^{\eps}+|\mathcal{S}^\eps_{0} |x_{G,1}^\eps+ 2m_{2,5}^{\eps}) + b^\eps (m_{5,1}^{\eps}-|\mathcal{S}^\eps_{0} |x_{G,2}^\eps-2m_{2,4}^{\eps})\Big)\\
& - 2 r^\eps a^\eps (m_{3,5}^{\eps}+m_{6}^{\eps} ) + 2 r^\eps b^\eps (m_{3,4}^{\eps}+m_{7}^{\eps} )
\Big\|_{L^\infty(0,T)}
 \leq C \varepsilon^{3}.
 \end{split}
\end{equation*}
\end{Proposition}
\begin{proof}
In this proof, we use the second-order approximation \eqref{Defvd2} of $\tilde{v}^{\varepsilon}$ and write:
\begin{equation} \label{Ciaddecompo2}
\tilde{v}^{\eps} = v^{\eps}_{\#} + r^{\eps} \nabla \Phi^{\eps}_{3} + R^\eps 
\ \text{ with } \ 
R^\eps:=\tilde{v}^{\eps}- v^{\eps}_{\#} - r^{\eps} \nabla \Phi^{\eps}_{3} .
\end{equation}
We note that by Proposition~\ref{Pr} one has
\begin{equation} \label{EstRepsilon}
\|R^\eps \|_{L^{\infty}(0,T;L^{2}(\partial {\mathcal S}_0^{\eps}))} = \mathcal{O}(\varepsilon^{5/2}) .
\end{equation}
\par
We start with the following observation:
\begin{eqnarray*}
C_{i,a}^{\varepsilon} &=& \frac{1}{2} \int_{\partial  \mathcal{S}^\eps_{0}  } |\tilde{v}^{\eps}|^2  K_i \, ds- \int_{\partial  \mathcal{S}^\eps_{0}  } (\ell^{\eps} + r^{\eps} x^\perp)\cdot \tilde{v}^{\eps}  K_i \, ds \\
&=&\frac{1}{2} \int_{\partial  \mathcal{S}^\eps_{0}  } |\tilde{v}^{\eps}-(\ell^{\eps} + r^{\eps} x^\perp)|^2  K_i \, ds - \frac{1}{2} \int_{\partial  \mathcal{S}^\eps_{0}  } |\ell^{\eps} + r^{\eps} x^\perp|^2  K_i \, ds .
\end{eqnarray*}
Replacing $\tilde{v}^{\eps}$ in this expression with the decomposition \eqref{Ciaddecompo2}, we compute
\begin{eqnarray} \nonumber
C_{i,a}^{\varepsilon} 
&=&\frac{1}{2} \int_{\partial  \mathcal{S}^\eps_{0}  } |R^\eps +(v^{\eps}_{\#}-\ell^{\eps}) +r^{\eps} (\nabla \Phi^{\eps}_{3} -x^\perp)|^2  K_i \, ds - \frac{1}{2} \int_{\partial  \mathcal{S}^\eps_{0}  } |\ell^{\eps} + r^{\eps} x^\perp|^2  K_i \, ds\\
\nonumber
&=&\frac{1}{2} \int_{\partial  \mathcal{S}^\eps_{0}  } |R^{\eps}|^2  K_i \, ds + 
\int_{\partial  \mathcal{S}^\eps_{0}  }  R^{\eps} \cdot\Big( (v^{\eps}_{\#}-\ell^{\eps}) +r^{\eps} (\nabla \Phi^{\eps}_{3} -x^\perp) \Big)  K_i \, ds +  \frac{1}{2} \int_{\partial  \mathcal{S}^\eps_{0}  } |v^{\eps}_{\#} -\ell^{\eps} |^2  K_i \, ds \\
\nonumber
&&+ \frac{1}{2} \int_{\partial  \mathcal{S}^\eps_{0}  } | r^{\eps}(\nabla \Phi^{\eps}_{3} -x^\perp) |^2  K_i \, ds
+ \int_{\partial  \mathcal{S}^\eps_{0}  }  r^{\eps}(v^{\eps}_{\#}-\ell^{\eps}) \cdot (\nabla \Phi^{\eps}_{3} -x^\perp)  K_i \, ds
- \frac{1}{2} \int_{\partial  \mathcal{S}^\eps_{0}  } |\ell^{\eps} + r^{\eps} x^\perp|^2  K_i \, ds \\
\label{DefD}
&=:& D_{i,a} + D_{i,b}+D_{i,c}+D_{i,d}+D_{i,e}+D_{i,f}.
\end{eqnarray}
We now analyze the various terms. We recall that $a^{\varepsilon}$ and $b^{\varepsilon}$ are defined in \eqref{a-et-b} and notice that
\begin{equation*} 
\widehat{(D K_{\R^{2}}[\omega^{\eps}](t,0) \cdot x)}
= \widehat{ \left(a^\eps \begin{pmatrix} -x_{1}\\x_{2} \end{pmatrix} + b^\eps \begin{pmatrix} x_{2}\\x_{1} \end{pmatrix}\right) }
= -(a^\eps+i b^\eps) z.
\end{equation*}
Now with the notation \eqref{frakK} for $K_{\R^{2}}[\omega^{\eps}](t,0)$, we can write $\widehat{v^{\eps}_{\#} -\ell^{\eps}}$ using \eqref{Defvd2} as follows:
\begin{equation} \label{v-l}
\widehat{v^{\eps}_{\#} -\ell^{\eps}}(z) = (\mathfrak{K}_{1}-\ell_{1}^\eps) - i(\mathfrak{K}_{2}-\ell_{2}^\eps) -(a^\eps+i b^\eps) z
+ (\ell_{1}^\eps - \mathfrak{K}_{1}) \widehat{\nabla \Phi^{\eps}_{1}} +(\ell_{2}^\eps - \mathfrak{K}_{2}) \widehat{\nabla \Phi^{\eps}_{2}}
- a^\eps \widehat{\nabla \Phi^{\eps}_{4}} - b^\eps \widehat{\nabla \Phi^{\eps}_{5}} .
\end{equation}
\ \par
\noindent
$\bullet$
The first term in \eqref{DefD} satisfies obviously $\|D_{i,a}\|_{L^\infty(0,T)}=\mathcal{O}(\eps^{5+\delta_{3,i}} )=o(\eps^{2+\delta_{3,i}})$.
The second one is of order $\mathcal{O}(\eps^{2+\delta_{3,i}})$, because using \eqref{phi-scaling}, Lemma~\ref{LemEstNRJ2} and the definition \eqref{Defvd2} of $v^{\eps}_{\#}$ we see that
\begin{equation*}
|D_{i,b}(t)|
\leq \int_{\partial \mathcal{S}^\eps_{0} } |R^{\eps}| \Big(|v^{\eps}_{\#}|+|\ell^{\eps}| + |r^{\eps}| (|\nabla \Phi^{\eps}_{3}| +|x|) \Big)|K_{i}| \, ds
\leq C \eps^{5/2} (1 + |\ell^{\eps}| +\eps |r^{\eps}|) \eps^{\delta_{3,i}} \sqrt{\eps}
\leq C \eps^{2+\delta_{3,i}}.
\end{equation*}
\ \par
\noindent
$\bullet$
We now turn to the third term.
As $v^{\eps}_{\#}-\ell^{\eps}$ is tangent to the boundary, we can apply the Blasius lemma (see Lemma~\ref{blasius}) and then \eqref{v-l}, Cauchy's residue theorem, Lemma~\ref{LemmeCircu} and Lemma~\ref{lemzphi} to obtain:
\begin{equation*}\begin{split}
\Big(D_{i,c}\Big)_{i=1,2} =&\frac12 \int_{\partial  \mathcal{S}^\eps_{0}  } |v^{\eps}_{\#} -\ell^{\eps} |^2  n \, ds = \frac{i}2 \left( \int_{ \partial  \mathcal{S}^\eps_{0}} (\widehat{v^{\eps}_{\#} -\ell^{\eps}})^2 \, dz \right)^*\\
=& - i \left( (a^\eps+i b^\eps) \int_{ \partial  \mathcal{S}^\eps_{0}} z\Big((\ell_{1}^\eps - \mathfrak{K}_{1}) \widehat{\nabla \Phi^{\eps}_{1}}+(\ell_{2}^\eps - \mathfrak{K}_{2}) \widehat{\nabla \Phi^{\eps}_{2}} -a^\eps \widehat{\nabla \Phi^{\eps}_{4}} -b^\eps \widehat{\nabla \Phi^{\eps}_{5}} \Big) \, dz \right)^*\\
=& - i \Bigg( (a^\eps+i b^\eps) \Big((\ell_{1}^\eps - \mathfrak{K}_{1}) (   -m_{1,2}^{\eps}+i(m_{1,1}^{\eps}+|\mathcal{S}^\eps_{0} |)  )+(\ell_{2}^\eps - \mathfrak{K}_{2}) ( -(m_{2,2}^{\eps}+|\mathcal{S}^\eps_{0} |)+im_{2,1}^{\eps} ) \\
&\quad  -a^\eps (-(m_{4,2}^{\eps}+|\mathcal{S}^\eps_{0} |x_{G,2}^\eps)+i(m_{4,1}^{\eps}-|\mathcal{S}^\eps_{0} |x_{G,1}^\eps)) -b^\eps (-(m_{5,2}^{\eps}+|\mathcal{S}^\eps_{0} |x_{G,1}^\eps)+i(m_{5,1}^{\eps}+|\mathcal{S}^\eps_{0} |x_{G,2}^\eps)) \Big)  \Bigg)^*
\end{split}
\end{equation*}
Going back to a vector notation, we get that
\begin{align*}
\Big(D_{i,c}\Big)_{i=1,2}=&
a^\eps 
\begin{pmatrix}
-(\ell_{1}^\eps - \mathfrak{K}_{1})(m_{1,1}^{\eps}+|\mathcal{S}^\eps_{0} |) -(\ell_{2}^\eps - \mathfrak{K}_{2}) m_{2,1}^{\eps} + a^\eps(m_{4,1}^{\eps}-|\mathcal{S}^\eps_{0} |x_{G,1}^\eps) +b^\eps (m_{5,1}^{\eps}+|\mathcal{S}^\eps_{0} |x_{G,2}^\eps)\\
(\ell_{1}^\eps - \mathfrak{K}_{1}) m_{1,2}^{\eps} +(\ell_{2}^\eps - \mathfrak{K}_{2}) (m_{2,2}^{\eps}+|\mathcal{S}^\eps_{0} |) -a^\eps (m_{4,2}^{\eps}+|\mathcal{S}^\eps_{0} |x_{G,2}^\eps) -b^\eps (m_{5,2}^{\eps}+|\mathcal{S}^\eps_{0} |x_{G,1}^\eps)
\end{pmatrix}\\
&+
b^\eps 
\begin{pmatrix}
(\ell_{1}^\eps - \mathfrak{K}_{1}) m_{1,2}^{\eps} +(\ell_{2}^\eps - \mathfrak{K}_{2}) (m_{2,2}^{\eps}+|\mathcal{S}^\eps_{0} |) -a^\eps (m_{4,2}^{\eps}+|\mathcal{S}^\eps_{0} |x_{G,2}^\eps) -b^\eps (m_{5,2}^{\eps}+|\mathcal{S}^\eps_{0} |x_{G,1}^\eps)\\
(\ell_{1}^\eps - \mathfrak{K}_{1})(m_{1,1}^{\eps}+|\mathcal{S}^\eps_{0} |) +(\ell_{2}^\eps - \mathfrak{K}_{2}) m_{2,1}^{\eps}- a^\eps(m_{4,1}^{\eps}-|\mathcal{S}^\eps_{0} |x_{G,1}^\eps) -b^\eps (m_{5,1}^{\eps}+|\mathcal{S}^\eps_{0} |x_{G,2}^\eps)
\end{pmatrix}\\
=&
a^\eps 
\begin{pmatrix}
-\ell_{1}^\eps (m_{1,1}^{\eps}+|\mathcal{S}^\eps_{0} |) -\ell_{2}^\eps  m_{2,1}^{\eps} \\
\ell_{1}^\eps m_{1,2}^{\eps} +\ell_{2}^\eps (m_{2,2}^{\eps}+|\mathcal{S}^\eps_{0} |) 
\end{pmatrix}
+
b^\eps 
\begin{pmatrix}
\ell_{1}^\eps  m_{1,2}^{\eps} +\ell_{2}^\eps (m_{2,2}^{\eps}+|\mathcal{S}^\eps_{0} |) \\
\ell_{1}^\eps (m_{1,1}^{\eps}+|\mathcal{S}^\eps_{0} |) +\ell_{2}^\eps m_{2,1}^{\eps}
\end{pmatrix} +\mathcal{O}(\eps^2).
\end{align*}
Now using Lemma~\ref{lemz2phi}, we proceed in the same way for $i=3$:
\begin{equation*}\begin{split}
D_{3,c}=&\frac12\int_{\partial  \mathcal{S}^\eps_{0}  } |v^{\eps}_{\#,I} -\ell^{\eps} |^2  K_3 \, ds=\frac12\Re \left( \int_{ \partial  \mathcal{S}^\eps_{0}} z (\widehat{v^{\eps}_{\#,I} -\ell^{\eps}})^2 \, dz \right) \\
=& \Re \Bigg( \Big[(\mathfrak{K}_{1}-\ell_{1}^\eps)-i(\mathfrak{K}_{2}-\ell_{2}^\eps)\Big]
\int_{ \partial  \mathcal{S}^\eps_{0}} z\Big((\ell_{1}^\eps - \mathfrak{K}_{1}) \widehat{\nabla \Phi^{\eps}_{1}}+(\ell_{2}^\eps - \mathfrak{K}_{2}) \widehat{\nabla \Phi^{\eps}_{2}} -a^\eps \widehat{\nabla \Phi^{\eps}_{4}} -b^\eps \widehat{\nabla \Phi^{\eps}_{5}} \Big) \, dz \\
&\quad - (a^\eps +i b^\eps)
\int_{ \partial  \mathcal{S}^\eps_{0}} z^2\Big((\ell_{1}^\eps - \mathfrak{K}_{1}) \widehat{\nabla \Phi^{\eps}_{1}}+(\ell_{2}^\eps - \mathfrak{K}_{2}) \widehat{\nabla \Phi^{\eps}_{2}} -a^\eps \widehat{\nabla \Phi^{\eps}_{4}} -b^\eps \widehat{\nabla \Phi^{\eps}_{5}} \Big) \, dz \Bigg)
\end{split}\end{equation*}
so that
\begin{equation*}\begin{split}
D_{3,c}
=&(\mathfrak{K}_{1}-\ell_{1}^\eps)\Big[
-(\ell_{1}^\eps - \mathfrak{K}_{1})m_{1,2}^{\eps}-(\ell_{2}^\eps - \mathfrak{K}_{2})(m_{2,2}^{\eps}+|\mathcal{S}^\eps_{0} |)+a^\eps (m_{4,2}^{\eps}+|\mathcal{S}^\eps_{0} |x_{G,2}^\eps) +b^\eps (m_{5,2}^{\eps}+|\mathcal{S}^\eps_{0} |x_{G,1}^\eps) \Big]\\
&+(\mathfrak{K}_{2}-\ell_{2}^\eps)\Big[
(\ell_{1}^\eps - \mathfrak{K}_{1})(m_{1,1}^{\eps}+|\mathcal{S}^\eps_{0}|) +(\ell_{2}^\eps - \mathfrak{K}_{2})m_{2,1}^{\eps}-a^\eps (m_{4,1}^{\eps}-|\mathcal{S}^\eps_{0} |x_{G,1}^\eps) -b^\eps (m_{5,1}^{\eps}+|\mathcal{S}^\eps_{0} |x_{G,2}^\eps) \Big]\\
&+2a^\eps \Big[
(\ell_{1}^\eps - \mathfrak{K}_{1})(m_{1,5}^{\eps} +  |\mathcal{S}^\eps_{0} |x_{G,2}^\eps ) +(\ell_{2}^\eps - \mathfrak{K}_{2})(m_{2,5}^{\eps} +  |\mathcal{S}^\eps_{0} |x_{G,1}^\eps )-a^\eps m_{4,5}^{\eps} -b^\eps ( m_{5,5}^{\eps}+ m_{8}^{\eps}) \Big]\\
&+2b^\eps \Big[
(\ell_{1}^\eps - \mathfrak{K}_{1})(-m_{1,4}^{\eps} +  |\mathcal{S}^\eps_{0} |x_{G,1}^\eps ) -(\ell_{2}^\eps - \mathfrak{K}_{2})(m_{2,4}^{\eps} +  |\mathcal{S}^\eps_{0} |x_{G,2}^\eps )+a^\eps ( m_{4,4}^{\eps}+m_{8}^{\eps}) +b^\eps m_{5,4}^{\eps} \Big].
\end{split}\end{equation*}
This can finally be simplified as follows:
\begin{equation*}\begin{split}
D_{3,c}
=& (K_{\R^{2}}[\omega^{\eps}](t,0)-\ell^\eps)^\perp \mathcal{M}_{\flat}^\eps (K_{\R^{2}}[\omega^{\eps}](t,0)-\ell^\eps)\\
&+ \ell_{1}^\eps \Big( a^\eps ( -m_{4,2}^{\eps}+|\mathcal{S}^\eps_{0} |x_{G,2}^\eps+2 m_{1,5}^{\eps}) + b^\eps (   -m_{5,2}^{\eps}+|\mathcal{S}^\eps_{0} |x_{G,1}^\eps-2m_{1,4}^{\eps}) \Big)\\
&+ \ell_{2}^\eps \Big( a^\eps (m_{4,1}^{\eps}+|\mathcal{S}^\eps_{0} |x_{G,1}^\eps+ 2m_{2,5}^{\eps}) + b^\eps (m_{5,1}^{\eps}-|\mathcal{S}^\eps_{0} |x_{G,2}^\eps-2m_{2,4}^{\eps})\Big)+\mathcal{O}(\eps^3).
\end{split}\end{equation*}\par
\ \par
\noindent
$\bullet$
We now turn to the fourth term in \eqref{DefD}. As $\nabla \Phi^{\eps}_{3} -x^\perp$  is tangent to the boundary, we can write for $i=1,2$ by Lemma~\ref{blasius} and by Cauchy's residue theorem:
\begin{align*}
\Big(D_{i,d}\Big)_{i=1,2} 
= \frac{(r^{\eps})^2}{2} \int_{\partial  \mathcal{S}^\eps_{0}  } | \nabla \Phi^{\eps}_{3} -x^\perp |^2  n \, ds
&=  \frac{i (r^{\eps})^2}2 \left( \int_{ \partial  \mathcal{S}^\eps_{0}} (\widehat{\nabla \Phi^{\eps}_{3} -x^\perp})^2 \, dz \right)^*
\\ &= \frac{i (r^{\eps})^2}2 \left( \int_{ \partial  \mathcal{S}^\eps_{0}} 2i\bar{z} \widehat{\nabla \Phi^{\eps}_{3} } \, dz 
- \int_{ \partial  \mathcal{S}^\eps_{0}} \bar{z}^2  \, dz\right)^*,
\end{align*}
where we have noted that $-\widehat{x^\perp}=i\bar{z}$. Thanks to Remark~\ref{remzphi} and \eqref{compu1}, we deduce that:
\begin{eqnarray*}
\Big(D_{i,d}\Big)_{i=1,2}
&=&i (r^{\eps})^2\Big(   i(-m_{3,2}^{\eps}+|\mathcal{S}^\eps_{0} |x_{G,1}^\eps)+(m_{3,1}^{\eps}+|\mathcal{S}^\eps_{0} |x_{G,2}^\eps)-2( |\mathcal{S}^\eps_{0} |x_{G,2}^\eps+i|\mathcal{S}^\eps_{0} |x_{G,1}^\eps )\Big)^*\\
&=&(r^{\eps})^2\Big( (-m_{3,2}^{\eps}-|\mathcal{S}^\eps_{0} |x_{G,1}^\eps)+i(m_{3,1}^{\eps}-|\mathcal{S}^\eps_{0} |x_{G,2}^\eps)\Big)\\
&=& (r^\eps)^2 \Big( \begin{pmatrix}-m_{3,2}^{\eps}\\m_{3,1}^{\eps}\end{pmatrix}- |\mathcal{S}^\eps_{0} |x_{G}^\eps\Big).
\end{eqnarray*}
For $i=3$, we have that:
\begin{equation*}\begin{split}
D_{3,d}&=\int_{\partial  \mathcal{S}^\eps_{0}  } |\nabla \Phi^{\eps}_{3} -x^\perp |^2  (  x^\perp \cdot n) \, ds= \int_{\partial  \mathcal{S}^\eps_{0}  } |\nabla \Phi^{\eps}_{3} |^2  x^\perp \cdot n \, ds - 2\int_{\partial  \mathcal{S}^\eps_{0}  } (\nabla \Phi^{\eps}_{3} \cdot x^\perp  )(  x^\perp \cdot n) \, ds +\int_{\partial  \mathcal{S}^\eps_{0}  } |x |^2  x^\perp \cdot n \, ds\\
&= \int_{  \mathcal{F}^\eps_{0}  } \div(|\nabla \Phi^{\eps}_{3} |^2  x^\perp) \, dx - 2\int_{  \mathcal{F}^\eps_{0}  }\nabla (\nabla \Phi^{\eps}_{3} \cdot x^\perp )  \cdot\nabla  \Phi^{\eps}_{3}\, dx -\int_{  \mathcal{S}^\eps_{0}  } \div(|x |^2  x^\perp) \, dx,
\end{split}\end{equation*}
where there is no boundary term at infinity because $\nabla \Phi_{3}(x)=\mathcal{O}(1/|x|^2)$ as $|x| \rightarrow +\infty$. Next we use the general relation \eqref{vect}
to obtain that
\begin{equation*}
\nabla(\nabla \Phi^{\eps}_{3} \cdot x^\perp )  \cdot\nabla  \Phi^{\eps}_{3} = \Big[(\nabla \Phi^{\eps}_{3} \cdot \nabla)x^\perp+(x^\perp\cdot \nabla)\nabla \Phi^{\eps}_{3}\Big] \cdot\nabla  \Phi^{\eps}_{3} = -(\nabla  \Phi^{\eps}_{3})^\perp\cdot \nabla  \Phi^{\eps}_{3}+\frac12 (x^\perp \cdot \nabla) |\nabla  \Phi^{\eps}_{3} |^2= \frac12 \div (x^\perp |\nabla  \Phi^{\eps}_{3} |^2).
\end{equation*}
Hence
\begin{eqnarray*}
D_{3,d}=-\int_{  \mathcal{S}^\eps_{0}  } \div(|x |^2  x^\perp) \, dx=-\int_{  \mathcal{S}^\eps_{0}  } (-2x_{1}x_{2}+2x_{1}x_{2}) \, dx=0.
\end{eqnarray*}
\ \par
\noindent
$\bullet$
Concerning the fifth term, we use again the Blasius lemma together with \eqref{v-l} and Cauchy's residue theorem:
\begin{equation*}\begin{split}
\Big( D_{i,e}\Big)_{i=1,2} 
&=  \int_{\partial  \mathcal{S}^\eps_{0}  }  r^{\eps}(v^{\eps}_{\#}-\ell^{\eps}) \cdot (\nabla \Phi^{\eps}_{3} -x^\perp)  n \, ds
= i r^{\eps}\left( \int_{ \partial  \mathcal{S}^\eps_{0}} (\widehat{v^{\eps}_{\#} -\ell^{\eps}})(  \widehat{\nabla \Phi^{\eps}_{3}} 
+ i \bar{z}) \, dz \right)^* \\
&=i r^{\eps}\Bigg( -\int_{ \partial  \mathcal{S}^\eps_{0}}(a^\eps+ib^\eps)z \widehat{\nabla \Phi^{\eps}_{3}} \, dz 
+ i\Big((\mathfrak{K}_{1}-\ell_{1}^\eps)-i(\mathfrak{K}_{2}-\ell_{2}^\eps)\Big)\int_{ \partial  \mathcal{S}^\eps_{0}}  \bar{z} \, dz 
-i(a^\eps+i b^\eps) \int_{ \partial  \mathcal{S}^\eps_{0}}  |z|^2 \, dz \\
&\quad+i(\ell_{1}^\eps - \mathfrak{K}_{1}) \int_{ \partial  \mathcal{S}^\eps_{0}} \widehat{\nabla \Phi^{\eps}_{1} } \bar{z} \, dz
+i(\ell_{2}^\eps - \mathfrak{K}_{2}) \int_{ \partial  \mathcal{S}^\eps_{0}} \widehat{\nabla \Phi^{\eps}_{2} } \bar{z} \, dz
-ia^\eps \int_{ \partial  \mathcal{S}^\eps_{0}} \widehat{\nabla \Phi^{\eps}_{4} } \bar{z} \, dz 
-ib^\eps \int_{ \partial  \mathcal{S}^\eps_{0}} \widehat{\nabla \Phi^{\eps}_{5} } \bar{z} \, dz
 \Bigg)^*.
\end{split}
\end{equation*}
Therefore, it suffices to write the values obtained in Lemma~\ref{lemzphi}, Remark~\ref{remzphi} and \eqref{compu2}-\eqref{compu3} to get:
\begin{equation*}\begin{split}
\Big( D_{i,e}\Big)_{i=1,2} 
=& r^\eps \Big[
-a^\eps (m_{3,1}^{\eps}-|\mathcal{S}^\eps_{0} |x_{G,2}^\eps) + b^\eps (m_{3,2}^{\eps}+|\mathcal{S}^\eps_{0} |x_{G,1}^\eps) +(\mathfrak{K}_{2}-\ell_{2}^\eps) 2 |\mathcal{S}^\eps_{0} | +a^\eps 2 |\mathcal{S}^\eps_{0} |x_{G,2}^\eps+b^\eps 2 |\mathcal{S}^\eps_{0} |x_{G,1}^\eps
\\
&\quad -(\ell_{1}^\eps - \mathfrak{K}_{1})m_{1,2}^{\eps} + (\ell_{2}^\eps - \mathfrak{K}_{2}) (-m_{2,2}^{\eps}+|\mathcal{S}^\eps_{0} |)-a^\eps (-m_{4,2}^{\eps}+|\mathcal{S}^\eps_{0} |x_{G,2}^\eps) -b^\eps(-m_{5,2}^{\eps}+|\mathcal{S}^\eps_{0} |x_{G,1}^\eps) \Big]
\\
&+ ir^\eps \Big[
a^\eps(m_{3,2}^{\eps}+|\mathcal{S}^\eps_{0} |x_{G,1}^\eps)+ b^\eps (m_{3,1}^{\eps}-|\mathcal{S}^\eps_{0} |x_{G,2}^\eps)-(\mathfrak{K}_{1}-\ell_{1}^\eps)2 |\mathcal{S}^\eps_{0} | + a^\eps 2 |\mathcal{S}^\eps_{0} |x_{G,1}^\eps - b^\eps 2 |\mathcal{S}^\eps_{0} |x_{G,2}^\eps
\\
&\quad -(\ell_{1}^\eps - \mathfrak{K}_{1}) (-m_{1,1}^{\eps}+|\mathcal{S}^\eps_{0} |)  +(\ell_{2}^\eps - \mathfrak{K}_{2})m_{2,1}^{\eps} - a^\eps(m_{4,1}^{\eps}+|\mathcal{S}^\eps_{0} |x_{G,1}^\eps) + b^\eps(-m_{5,1}^{\eps}+|\mathcal{S}^\eps_{0} |x_{G,2}^\eps) \Big],
\end{split}\end{equation*}
which can be simplified as
\begin{equation*}\begin{split}
\Big( D_{i,e}\Big)_{i=1,2} 
= r^\eps\Bigg[&-(\mathcal{M}_{\flat}^\eps + |\mathcal{S}^\eps_{0} | {\rm I}_{2})(K_{\R^{2}}[\omega^{\eps}](t,0) -\ell^{\eps} )  \Big)^\perp\\
&+ a^\eps
\begin{pmatrix}
 -m_{3,1}^{\eps}+m_{4,2}^{\eps} + 2 |\mathcal{S}^\eps_{0} |x_{G,2}^\eps\\
 m_{3,2}^{\eps}-m_{4,1}^{\eps} +2 |\mathcal{S}^\eps_{0} |x_{G,1}^\eps
\end{pmatrix}
+b^\eps
\begin{pmatrix}
 m_{3,2}^{\eps}+m_{5,2}^{\eps} + 2 |\mathcal{S}^\eps_{0} |x_{G,1}^\eps\\
 m_{3,1}^{\eps}-m_{5,1}^{\eps} - 2 |\mathcal{S}^\eps_{0} |x_{G,2}^\eps
\end{pmatrix} \Bigg]
.
\end{split}\end{equation*}
\par
For $i=3$,  Lemma~\ref{blasius}, \eqref{v-l} and Cauchy's residue theorem imply that
\begin{equation*}\begin{split}
D_{3,e} &= \int_{\partial  \mathcal{S}^\eps_{0}  }  r^{\eps}(v^{\eps}_{\#}-\ell^{\eps}) \cdot (\nabla \Phi^{\eps}_{3} -x^\perp)  K_3 \, ds
= r^\eps\Re \left( \int_{ \partial  \mathcal{S}^\eps_{0}} z (\widehat{v^{\eps}_{\#} -\ell^{\eps}})(  \widehat{\nabla \Phi^{\eps}_{3}} + i \bar{z})  \, dz \right) \\
&= r^\eps\Re\Bigg[ \Big((\mathfrak{K}_{1}-\ell_{1}^\eps)-i(\mathfrak{K}_{2}-\ell_{2}^\eps)\Big) \int_{ \partial  \mathcal{S}^\eps_{0}} (z \widehat{\nabla \Phi^{\eps}_{3}} + i|z|^2)\, dz 
-(a^\eps + i b^\eps) \int_{ \partial  \mathcal{S}^\eps_{0}} (z^2 \widehat{\nabla \Phi^{\eps}_{3}} + iz |z|^2)\, dz \\
& \ + (\ell_{1}^\eps - \mathfrak{K}_{1}) i \int_{ \partial  \mathcal{S}^\eps_{0}} \widehat{\nabla \Phi^{\eps}_{1}} |z|^2\, dz
+(\ell_{2}^\eps - \mathfrak{K}_{2})  i \int_{ \partial  \mathcal{S}^\eps_{0}} \widehat{\nabla \Phi^{\eps}_{2}} |z|^2\, dz 
-a^\eps  i \int_{ \partial  \mathcal{S}^\eps_{0}} \widehat{\nabla \Phi^{\eps}_{4}} |z|^2\, dz 
-b^\eps  i \int_{ \partial  \mathcal{S}^\eps_{0}} \widehat{\nabla \Phi^{\eps}_{5}} |z|^2\, dz
  \Bigg].
\end{split}
\end{equation*}
Hence, we obtain from \eqref{compu3}-\eqref{compu4} and Lemmas~\ref{lemzphi}, \ref{lem|z|2phi} and \ref{lemz2phi}:
\begin{equation*}\begin{split}
D_{3,e} &= r^\eps\Bigg[ -(\mathfrak{K}_{1}-\ell_{1}^\eps) (m_{3,2}^{\eps}+|\mathcal{S}^\eps_{0} |x_{G,1}^\eps + 2 |\mathcal{S}^\eps_{0} |x_{G,1}^\eps)
+(\mathfrak{K}_{2}-\ell_{2}^\eps) (m_{3,1}^{\eps}-|\mathcal{S}^\eps_{0} |x_{G,2}^\eps -2 |\mathcal{S}^\eps_{0} |x_{G,2}^\eps)\\
& \quad \quad \quad + a^\eps 2(m_{3,5}^{\eps} +m_{6}^{\eps}  +m_{6}^{\eps} )
-b^\eps 2(m_{3,4}^{\eps} +m_{7}^{\eps}  +m_{7}^{\eps})
-(\ell_{1}^\eps - \mathfrak{K}_{1})2|\mathcal{S}^\eps_{0} |x_{G,1}^\eps\\
& \quad \quad \quad -(\ell_{2}^\eps - \mathfrak{K}_{2})2|\mathcal{S}^\eps_{0} |x_{G,2}^\eps 
 - a^\eps 2m_6^{\eps} + b^\eps 2m_7^{\eps} \Bigg]\\
& = r^\eps({K_{\R^{2}}[\omega^{\eps}](t,0)-\ell^{\eps} }) \cdot \Big(\begin{pmatrix} -m_{3,2}^{\eps} \\m_{3,1}^{\eps} \end{pmatrix} -|\mathcal{S}^\eps_{0} |x_{G}^\eps\Big) + 2 r^\eps a^\eps (m_{3,5}^{\eps}+m_{6}^{\eps} ) - 2 r^\eps b^\eps (m_{3,4}^{\eps}+m_{7}^{\eps} ).
\end{split}\end{equation*}
\ \par
\noindent
$\bullet$
Finally, for the last term we write
\begin{equation*}
\int_{\partial  \mathcal{S}^\eps_{0}  } |\ell^{\eps} + r^{\eps} x^\perp|^2  K_i \, ds 
= |\ell^{\eps}|^2  \int_{\partial  \mathcal{S}^\eps_{0}  } K_i \, ds 
- 2 \ell^{\eps}_1 r^\eps \int_{\partial  \mathcal{S}^\eps_{0}  } x_2  K_i \, ds 
+2 \ell^{\eps}_2 r^\eps \int_{\partial  \mathcal{S}^\eps_{0}  } x_1  K_i \, ds 
+ (r^{\eps})^2  \int_{\partial  \mathcal{S}^\eps_{0}  } |x|^2 K_i \, ds,
\end{equation*}
where the above integrals are computed in Lemma~\ref{CalculIntegralesStandard}. Then we check easily that
\begin{equation*}
\Big( D_{i,f}\Big)_{i=1,2} 
=\Big(- \frac{1}{2} \int_{\partial  \mathcal{S}^\eps_{0}  } |\ell^{\eps} + r^{\eps} x^\perp|^2  K_i \, ds\Big)_{i=1,2}
=-r^\eps (\ell^\eps)^\perp |\mathcal{S}^\eps_{0} | + (r^\eps)^2x_{G}^\eps |\mathcal{S}^\eps_{0} |
\end{equation*}
and
\begin{equation*}
D_{3,f}=- \frac{1}{2} \int_{\partial  \mathcal{S}^\eps_{0}  } |\ell^{\eps} + r^{\eps} x^\perp|^2  K_3 \, ds
=-r^\eps (\ell^\eps\cdot x_{G}^\eps) |\mathcal{S}^\eps_{0} |.
\end{equation*}
This ends the proof of Proposition~\ref{PropCiad}.
\end{proof}
\ \par
\noindent
{\bf 2. The $C_{i,c}^\eps$ term.} We turn to $C_{i,c}^{\varepsilon}$ introduced in \eqref{Cic}. Here, we deduce from Lemma~\ref{blasius}, \eqref{HSeriesLaurent} and Cauchy's Residue Theorem that $C^{\varepsilon}_{1,c}=C^{\varepsilon}_{2,c}=C^{\varepsilon}_{3,c}=0$. Note in particular that $\int_{\partial  \mathcal{S}^\eps_{0} } z (\widehat{H^{\eps}})^2  \, dz=-i/(2\pi) $ is purely imaginary.
\par
\ \par
\noindent
{\bf 3. The $C_{i,b}^\eps$ term.} We finally turn to the term $C^{\varepsilon}_{i,b}$ in \eqref{Cib}. Let us prove the following.
\begin{Proposition} \label{PropCib}
Let $\rho>1$ be fixed. There exist $C>0$ and $\varepsilon_{0}\in (0,1]$ such that if for a given $T>0$ and an $\varepsilon \in (0,\varepsilon_{0}]$, \eqref{CondDistVorticiteSolide} is satisfied for all $t \in [0,T]$, then one has:
\begin{equation*} 
\Big\|  \begin{pmatrix}C_{1,b}^{\varepsilon} \\C_{2,b}^{\varepsilon} \end{pmatrix}
-  \gamma ( K_{\R^2}[\omega^\eps] (t,0) -  \ell^{\eps} )^\perp 
-\gamma \varepsilon r^{\varepsilon} \xi
-\gamma\varepsilon  \Big( D K_{\R^2}[\omega^\eps](t,0) \cdot \xi \Big)^\perp
 \Big\|_{L^\infty(0,T)} \leq C \eps^2,
\end{equation*}
and 
\begin{equation*} 
\Big\| C_{3,b}^{\varepsilon} - \gamma \varepsilon \, \xi \cdot ( K_{\R^2}[\omega^\eps] (t,0) -  \ell^{\eps} )
-\gamma \eps^2 (-a^\eps \eta_{1} +b^\eps \eta_{2})
  \Big\|_{L^\infty(0,T)}  \leq C \eps^3,
\end{equation*}
where $\xi$ was defined in \eqref{DefXi} and $\eta$ was defined in \eqref{DefEta}.
\end{Proposition}
\begin{proof}
As for $C^\eps_{i,a}$, we consider the approximation \eqref{Defvd2} of $\tilde{v}^{\varepsilon}$ and use again the decomposition \eqref{Ciaddecompo2}
where $R^{\varepsilon}$ satisfies \eqref{EstRepsilon}.
Putting this decomposition in the definition of $C^\eps_{i,b}$, we obtain:
\begin{equation} \label{DecompCib}
C_{i,b}^\eps =  \gamma  \int_{\partial  \mathcal{S}^\eps_{0}  } R^\eps \cdot  H^{\eps}  K_i \, ds  
+ \gamma  \int_{\partial  \mathcal{S}^\eps_{0}  } (v^{\eps}_{\#}  - \ell^{\eps} )\cdot  H^{\eps}  K_i \, ds
+ \gamma  \int_{\partial  \mathcal{S}^\eps_{0}  }r^{\eps} (\nabla \Phi^{\eps}_{3} - x^\perp)\cdot  H^{\eps}  K_i \, ds.
\end{equation}
%
%
%
Thanks to the scaling law \eqref{ScalingH} for $H^\eps$, the first term is of order $\mathcal{O}(\eps^{2+\delta_{3,i}} )$:
\begin{equation*}
|\gamma|\int_{\partial  \mathcal{S}^\eps_{0}  }  |R^{\eps}| |H^{\eps}| |K_{i}| \, ds
\leq C\eps^{5/2} \eps^{-1} \eps^{\delta_{3,i}}\sqrt\eps\leq C \eps^{2+\delta_{3,i}}.
\end{equation*}
%
%
Concerning the second term in \eqref{DecompCib}, as $v^{\eps}_{\#}-\ell^{\eps}$ and $H^{\eps}$ are tangent to the boundary, we can apply the Blasius lemma (see Lemma~\ref{blasius}). Then we compute by \eqref{v-l}, Cauchy's residue theorem and \eqref{HSeriesLaurent}:
\begin{equation*}\begin{split}
\left(\gamma \int_{\partial  \mathcal{S}^\eps_{0}  }(v^{\eps}_{\#}  - \ell^{\eps} )\cdot  H^{\eps}   K_i \, ds\right)_{i=1,2}  &= \gamma\int_{\partial  \mathcal{S}^\eps_{0}  }(v^{\eps}_{\#}  - \ell^{\eps} )\cdot  H^{\eps}   n \, ds = i\gamma \left( \int_{ \partial  \mathcal{S}^\eps_{0}} \widehat{(v^{\eps}_{\#} -\ell^{\eps})}\widehat{H^\eps} \, dz \right)^*\\
&= i\gamma \Bigg( \Big((\mathfrak{K}_{1}-\ell_{1}^\eps)-i(\mathfrak{K}_{2}-\ell_{2}^\eps)\Big) \int_{ \partial  \mathcal{S}^\eps_{0}}\widehat{H^\eps} \, dz  - (a^\eps+i b^\eps) \int_{ \partial  \mathcal{S}^\eps_{0}} z \widehat{H^\eps} \, dz \Bigg)^*\\
&= i\gamma \Bigg( \Big((\mathfrak{K}_{1}-\ell_{1}^\eps)-i(\mathfrak{K}_{2}-\ell_{2}^\eps)\Big)   - (a^\eps+i b^\eps) \eps(\xi_{1}+i\xi_{2}) \Bigg)^*\\
&= \gamma (K_{\R^{2}}[\omega^{\eps}](t,0)-\ell^\eps)^\perp +\gamma\varepsilon  \Big(  \begin{pmatrix} -a^\eps &b^\eps \\ b^\eps & a^\eps \end{pmatrix}  \xi \Big)^\perp,
\end{split}
\end{equation*}
where we have used the notation \eqref{frakK} and the relation \eqref{DefXi}. \par
For $i=3$, we compute by Lemma~\ref{blasius} and Cauchy's residue theorem that
\begin{equation*}\begin{split}
\gamma\int_{\partial  \mathcal{S}^\eps_{0}  }(v^{\eps}_{\#}  - \ell^{\eps} )\cdot  H^{\eps}   K_3 \, ds
& = \gamma \, \Re \left( \int_{ \partial  \mathcal{S}^\eps_{0}} z (\widehat{v^{\eps}_{\#} -\ell^{\eps}})\widehat{H^\eps}  \, dz \right) \\
& = \gamma \, \Re \left(  \Big((\mathfrak{K}_{1}-\ell_{1}^\eps)-i(\mathfrak{K}_{2}-\ell_{2}^\eps)\Big) \int_{ \partial  \mathcal{S}^\eps_{0}}z\widehat{H^\eps} \, dz  
 - (a^\eps+i b^\eps) \int_{ \partial  \mathcal{S}^\eps_{0}} z^2 \widehat{H^\eps} \, dz\right) \\
& = \gamma \, \Re \left(  \Big((\mathfrak{K}_{1}-\ell_{1}^\eps)-i(\mathfrak{K}_{2}-\ell_{2}^\eps)\Big) \eps(\xi_{1}+i\xi_{2})
 - (a^\eps+i b^\eps) \eps^2(\eta_{1}+i\eta_{2}) \right) \\
& = \gamma \, \eps( K_{\R^2}[\omega^\eps] (t,0) -  \ell^{\eps} )\cdot \xi + \gamma \eps^2 (-a^\eps \eta_{1} +b^\eps \eta_{2}).
\end{split}\end{equation*}
\ \par
For the last term, we use that $\nabla \Phi^{\eps}_{3} -x^\perp$ and $H^\eps$ are tangent to the boundary, and obtain with Lemma~\ref{blasius} and Cauchy's residue theorem:
\begin{eqnarray*}
\left( \gamma r^{\eps}  \int_{\partial  \mathcal{S}^\eps_{0}  } (\nabla \Phi^{\eps}_{3} - x^\perp)\cdot  H^{\eps}   K_i \, ds\right)_{i=1,2}
&=&  i \gamma r^{\eps}  \left( \int_{ \partial  \mathcal{S}^\eps_{0}} (\widehat{\nabla \Phi^{\eps}_{3} -x^\perp})\widehat{H^\eps} \, dz \right)^*
=  i \gamma r^{\eps}  \left(  i\int_{ \partial  \mathcal{S}^\eps_{0}} \bar{z}\widehat{H^\eps}  \, dz\right)^*\\
&=& \gamma r^{\eps}  \left(  \int_{ \partial  \mathcal{S}^\eps_{0}} \bar{z}\widehat{H^\eps}  \, dz\right)^*= \gamma r^{\eps}\eps  \left(  \int_{ \partial  \mathcal{S}_{0}} \bar{z}\widehat{H^1}  \, dz\right)^* = \gamma \eps r^\eps \xi,
\end{eqnarray*}
where we have used that $-\widehat{x^\perp}=i\bar{z}$ and Lemma \ref{LemzH}. \par
For $i=3$, we have that:
\begin{equation*}
\gamma r^{\eps}  \int_{\partial  \mathcal{S}^\eps_{0}  } (\nabla \Phi^{\eps}_{3} - x^\perp)\cdot  H^{\eps}  K_3 \, ds 
= \gamma  r^{\eps} \Re \left( \int_{ \partial  \mathcal{S}^\eps_{0}} z (\widehat{\nabla \Phi^{\eps}_{3} -x^\perp})\widehat{H^\eps}  \, dz \right) 
=  \gamma  r^{\eps} \Re \left( i \int_{ \partial  \mathcal{S}^\eps_{0}} z \bar{z}\widehat{H^\eps}  \, dz \right)=0,
\end{equation*}
because of Lemma \ref{LemzH}. 
This ends the proof of Proposition~\ref{PropCib}.
\end{proof} 
\subsection{Conclusion}
\label{Subsec:Regroupement}
In this subsection, we gather all the previous results established in this section into a single proposition.
We begin by recalling and introducing several notations. \par
Let $\underline{\Lambda}: \R^{3} \times \R^{3} \rightarrow \R^{3}$ be the symmetric bilinear mapping that satisfies
\begin{equation} \label{DefGammaSouligne}
\langle \underline{\Lambda}, p , p \rangle = \begin{pmatrix} r ({\mathcal M}_{\flat}^{1} \ell)^{\perp} \\ \ell^{\perp} \cdot {\mathcal M}_{\flat}^{1} \ell \end{pmatrix} \ \ \text{ for all } p = \begin{pmatrix} \ell \\ r \end{pmatrix} \in \R^{3},
\end{equation}
where we recall that $\mathcal{M}^{\varepsilon}_{\flat}$, the $2\times2$ restriction of $\mathcal{M}_{a}^\eps$, was defined in \eqref{DefMbemol}.
We also let
\begin{equation*} 
\hat{p}^{\varepsilon} := \begin{pmatrix} \ell^{\varepsilon} \\ \varepsilon r^{\varepsilon} \end{pmatrix}, \ \ \ 
\check{p}^{\varepsilon} := \begin{pmatrix} \ell^{\varepsilon} - K_{\R^{2}}[\omega^{\varepsilon}](t,0) \\ \varepsilon r^{\varepsilon} \end{pmatrix}.
\end{equation*}
We introduce the following vectors relying on the quantities $m^{1}_{i,j}$ defined in \eqref{DefMasse}:
\begin{equation} \label{VectMass}
{\mu}^{1} := \begin{pmatrix} m_{1,3}^{1} \\ m_{2,3}^{1}  \\ 0 \end{pmatrix}, \ 
\hat{\mu}^{1} := \begin{pmatrix} 2m_{2,5}^{1}- m_{3,2}^{1} + m_{1,4}^{1} \\ -2m_{1,5}^{1} -m_{3,1}^{1}+ m_{4,2}^{1} \\ 0 \end{pmatrix}
\ \text{ and } \
\check{\mu}^{1} := \begin{pmatrix} -2m_{2,4}^{1}-m_{3,1}^{1} + m_{5,1}^{1} \\ 2m_{1,4}^{1}+ m_{3,2}^{1} + m_{5,2}^{1} \\ 0 \end{pmatrix}.
\end{equation}
Recalling that $a^{\varepsilon}$ and $b^{\varepsilon}$ were defined in \eqref{a-et-b}, $\xi=(\xi_{1}, \xi_{2})$ in \eqref{DefXi} and $\eta=(\eta_{1}, \eta_{2})$ in \eqref{DefEta}, we let
\begin{equation} \label{DefFSouligne}
\underline{F}(\varepsilon,t) = \hat{p}^{\varepsilon} \times (a^{\varepsilon} \hat{\mu}^{1} + b^{\varepsilon} \check{\mu}^{1})
- \begin{pmatrix} \Big( (m_{1,1}^{1}-m_{2,2}^{1})b^\eps +2 m_{1,2}^{1} a^\eps\Big) (\ell^\eps)^\perp \\ 0 \end{pmatrix}
\end{equation}
and
\begin{equation} \label{DefG}
G(\varepsilon,t) = \begin{pmatrix}
	0 \\ 0 \\ \xi \cdot (D K_{\R^2}[\omega^\varepsilon](t,0) \cdot \xi) + a^{\varepsilon} \eta_{1} - b^{\varepsilon} \eta_{2}
\end{pmatrix}.
\end{equation}
Finally, we recall that $I_{\varepsilon}$ is defined in \eqref{defIeps} and that ${\bf B}$ is defined in \eqref{DefB}.  \par
The previous subsections result into the following proposition.
\begin{Proposition} \label{PropositionPression2}
Under the same assumptions as Proposition~\ref{Pro:NormalForm}, the pressure force/torque can be written as follows:
\begin{multline*}
- I_{\varepsilon}^{-1} ( B_{i}^\eps + C_{i}^\eps)_{i=1,2,3} = \gamma \tilde{p}^{\varepsilon} \times {\bf B} \\
+ \varepsilon \Big[ -(\varepsilon r^{\varepsilon}) \begin{pmatrix} {\mathcal M}_{\flat}^{1} (K_{\R^2}[\omega^\varepsilon](t,0))^{\perp} \\ 0 \end{pmatrix} -(\varepsilon r^{\varepsilon}) \hat{p}^{\varepsilon} \times \mu^{1} + \gamma G(\varepsilon,t) 
- \langle \underline{\Lambda}, \check{p}^{\varepsilon}, \check{p}^{\varepsilon} \rangle \Big]
+ \varepsilon^{2} \Big[ \underline{F}(\varepsilon,t) + R^{\varepsilon} \Big],
\end{multline*}
with
\begin{equation*} 
\| R^\eps\|_{L^{\infty}(0,T)}  \leq C.
\end{equation*}
\end{Proposition}
\ \par
We are now in position to prove the main statement of a normal form, that is Proposition~\ref{Pro:NormalForm}, where $\Lambda_{a}$ is defined as the bilinear symmetric mapping such that
\begin{equation} \label{DefGammaa}
\langle \Lambda_a , p, p \rangle : = \langle \underline{\Lambda} , p, p \rangle 
+ r p \times \mu^1 \ \ \text{ for all } p = \begin{pmatrix} \ell \\ r \end{pmatrix} \in \R^{3},
\end{equation}
where we recall that $\underline{\Lambda}$ and $\mu_{1}$ were defined in \eqref{DefGammaSouligne} and \eqref{VectMass}.
We check easily that $\Lambda_{a}$ satisfies \eqref{AnnulationGammaa}. We recall also that $\Lambda_{g}$ was defined in \eqref{DefGammag}.\par
\begin{proof}[Proof of Proposition~\ref{Pro:NormalForm}]
We start from \eqref{EqSolideBis} that we reformulate as 
\begin{equation*}
\big[ \varepsilon^\alpha {\mathcal M}_{g} + \varepsilon^2 {\mathcal M}_{a} \big]
(\hat{p}^{\varepsilon} )'
= -  I_{\varepsilon}^{-1}  (B_i^\eps + C_i^\eps)_{i=1,2,3} - \varepsilon^\alpha  \begin{pmatrix} m^1 r^{\varepsilon} (\ell^{\varepsilon})^{\perp} \\ 0\end{pmatrix} .
\end{equation*}
Then we use that according to  \eqref{DefLTilde} and Proposition~\ref{labelkoi2} we have
\begin{equation*} 
 ( \tilde{\ell}^{\varepsilon})'(t) = (\ell^{\varepsilon} )' (t) + r^\varepsilon(t)  (K_{\R^2}[\omega^\varepsilon](t,0) )^\perp - F_d (\varepsilon,t)  , 
\end{equation*}
with $F_d$ weakly nonlinear in the sense of \eqref{IneqWNL}. Thus with Proposition~\ref{PropositionPression2} we deduce
\begin{align*}
\big[ \varepsilon^\alpha {\mathcal M}_{g} &+ \varepsilon^2 {\mathcal M}_{a} \big] ( \tilde{p}^{\varepsilon}) ' \\
&=  \gamma \tilde{p}^{\varepsilon} \times {\bf B} 
+ \varepsilon \Big[
-(\varepsilon r^\varepsilon) \hat{p}^{\varepsilon} \times \mu^{1} + \gamma G(\varepsilon,t) 
- \langle \underline{\Lambda}, \check{p}^{\varepsilon}, \check{p}^{\varepsilon} \rangle \Big]
- \varepsilon^\alpha  \begin{pmatrix} m^1 r^{\varepsilon} (\ell^{\varepsilon})^{\perp} \\ 0\end{pmatrix} \\
& \quad + \varepsilon^\alpha m^1 \begin{pmatrix}  r^{\varepsilon} (K_{\R^2}[\omega^\varepsilon](t,0) )^\perp  \\ 0\end{pmatrix}
+ \varepsilon^2 r^\varepsilon \begin{pmatrix} 0 \\ -m_{3,1}^{1} (K_{\R^2}[\omega^\varepsilon](t,0) )_{2} +m_{3,2}^{1} (K_{\R^2}[\omega^\varepsilon](t,0) )_{1} \end{pmatrix} \\
& \quad + \varepsilon^{\min(\alpha,2)}   \overline{F}(\varepsilon,t) +  \varepsilon^{2} R_{\varepsilon} \\
&=  \gamma \tilde{p}^{\varepsilon} \times {\bf B} 
+ \varepsilon \Big[
- (\varepsilon r^\varepsilon) \check{p}^{\varepsilon} \times \mu^{1} + \gamma G(\varepsilon,t) 
- \langle \underline{\Lambda}, \check{p}^{\varepsilon}, \check{p}^{\varepsilon} \rangle \Big]
- \varepsilon^{\alpha-1} \langle \Lambda_{g}, \check{p}^{\varepsilon},\check{p}^{\varepsilon} \rangle
+ \varepsilon^{\min(\alpha,2)}   \overline{F}(\varepsilon,t),
\end{align*}
where
\begin{equation*}
\overline{F}(\varepsilon,t)= \underline{F}(\varepsilon,t) 
- \left( \varepsilon^{\alpha- \min(\alpha,2)} {\mathcal M}_{g} + \varepsilon^{2- \min(\alpha,2)} {\mathcal M}_{a} \right)  F_d (\varepsilon,t)
+ \varepsilon^{2 - \min(\alpha,2)} R^{\varepsilon},
\end{equation*}
and where $\Lambda_{g}$ is defined in \eqref{DefGammag}.
Note that \eqref{DefGammaSouligne} and \eqref{DefGammaa} allow to simplify a bit the right hand side as follows:
\begin{equation*}
\big[ \varepsilon^\alpha {\mathcal M}_{g} + \varepsilon^2 {\mathcal M}_{a} \big] ( \tilde{p}^{\varepsilon}) ' 
=  \gamma \tilde{p}^{\varepsilon} \times {\bf B} 
+ \varepsilon \Big[ \gamma G(\varepsilon,t) - \langle \Lambda_{a}, \check{p}^{\varepsilon}, \check{p}^{\varepsilon} \rangle \Big]
- \varepsilon^{\alpha-1} \langle \Lambda_{g}, \check{p}^{\varepsilon},\check{p}^{\varepsilon} \rangle
+ \varepsilon^{\min(\alpha,2)}   \overline{F}(\varepsilon,t).
\end{equation*}
Now we modify the right hand side in order to make the modulated unknown $\tilde{p}^{\varepsilon}$ appear, instead of $\check{p}^{\varepsilon}$. This yields
\begin{equation} \label{Enfin}
\big[ \varepsilon^\alpha {\mathcal M}_{g} + \varepsilon^2 {\mathcal M}_{a} \big] ( \tilde{p}^{\varepsilon}) ' 
+ \eps^{\alpha - 1} \langle  \Lambda_g ,  \tilde{p}^{\varepsilon} ,  \tilde{p}^{\varepsilon} \rangle
+ \eps \langle  \Lambda_a ,  \tilde{p}^{\varepsilon} ,  \tilde{p}^{\varepsilon} \rangle
= \gamma \tilde{p}^{\varepsilon} \times {\bf B} 
+ \varepsilon  \gamma G(\varepsilon,t) 
+ \varepsilon^{\min(\alpha,2)} {F}(\varepsilon,t) ,
\end{equation}
where 
\begin{equation*} 
{F}(\varepsilon,t) = \overline{F}(\varepsilon,t) + \hat{F}(\varepsilon,t),
\end{equation*}
with
\begin{multline*}
\hat{F}= \varepsilon^{\alpha - \min(\alpha,2)}
\Big(- 2 \langle  \Lambda_g ,  D K_{\R^{2}}[\omega^{\varepsilon}](t,0) \cdot \xi ,  \tilde{p}^{\varepsilon} \rangle
+ \varepsilon \langle  \Lambda_g , D K_{\R^{2}}[\omega^{\varepsilon}](t,0) \cdot \xi , D K_{\R^{2}}[\omega^{\varepsilon}](t,0) \cdot \xi \rangle \Big) \\
+ \varepsilon^{2 - \min(\alpha,2)}  \Big( - 2 \langle  \Lambda_a , D K_{\R^{2}}[\omega^{\varepsilon}](t,0) \cdot \xi ,  \tilde{p}^{\varepsilon} \rangle 
+ \varepsilon \langle \Lambda_a , D K_{\R^{2}}[\omega^{\varepsilon}](t,0) \cdot \xi , D K_{\R^{2}}[\omega^{\varepsilon}](t,0) \cdot \xi\rangle \Big).
\end{multline*}
It remains to check that $F$ is weakly nonlinear in the sense of \eqref{IneqWNL} and that $G$ is weakly gyroscopic in the sense of \eqref{IneqWG}. \par
\begin{Lemma}
The term $F$ is weakly nonlinear and the term $G$ is weakly gyroscopic.
\end{Lemma}
\begin{proof}

We begin by proving that the term $G$ that we made explicit in \eqref{DefG} is weakly gyroscopic, that is, satisfies \eqref{IneqWG}. We can rewrite $G$ as follows:
\begin{equation} \label{DefG2}
G(\varepsilon,t) = \begin{pmatrix}
 	0 \\ 0 \\ \xi \cdot (D K_{\R^2}[\omega^\varepsilon](t,0) \cdot \xi)
 	-e_{1}\cdot  (D K_{\R^2}[\omega^\varepsilon](t,0) \cdot \eta)
 	\end{pmatrix},
\end{equation}
where we denote by $(e_{1},e_{2})$ the canonical basis of $\R^{2}$. That $e_{1}$ appears in the formula \eqref{DefG2} does not mean that it is a privileged direction, recall \eqref{a-et-b}.
Using \eqref{nablaK} and \eqref{a-et-b} we have
\begin{multline*} 
\int_{0}^{t} \tilde{p}^{\varepsilon}(s) \cdot G(\varepsilon,s) \, ds 
= \varepsilon \int_{0}^{t} r^{\varepsilon}(s) \Big( R_{\theta^{\varepsilon}(s)}\xi \Big)\cdot \Big( D K_{\R^2} [w^{\varepsilon}](s,h^{\varepsilon}(s)) \cdot R_{\theta^{\varepsilon}(s)}\xi \Big) \, ds \\
- \varepsilon \int_{0}^{t} r^{\varepsilon}(s) \Big(R_{\theta^{\varepsilon}(s)} e_{1}\Big) \cdot \Big(D K_{\R^2}  [w^{\varepsilon}](s,h^{\varepsilon}(s)) \cdot R_{\theta^{\varepsilon}(s)}\eta\Big) \, ds .
\end{multline*}
Let us focus on the first term on the right hand side, that we denote $I_{1}$, the second one being treated likewise. 
Since the function $K_{\R^2}  [w^{\varepsilon}](s,\cdot)$ is the orthogonal gradient of a harmonic function in the neighborhood of $h^{\varepsilon}(s)$, the matrix $D K_{\R^2}  [w^{\varepsilon}](s,h^{\varepsilon}(s))$, as for \eqref{a-et-b}, is a $2 \times 2$ traceless symmetric matrix. It is hence the composition of a linear homothety and an axial orthogonal symmetry. It follows that
\begin{equation*}
D K_{\R^2}  [w^{\varepsilon}](s,h^{\varepsilon}(s)) R_{\theta^{\varepsilon}(s)} 
= R_{-\theta^{\varepsilon}(s)} D K_{\R^2}  [w^{\varepsilon}](s,h^{\varepsilon}(s)) 
= R_{\theta^{\varepsilon}(s)}^{T} D K_{\R^2}  [w^{\varepsilon}](s,h^{\varepsilon}(s)) .
\end{equation*}
With $r^{\varepsilon}=(\theta^{\varepsilon})'$, we deduce that
\begin{eqnarray*}
I_{1}
&=& \varepsilon \int_{0}^{t} r^{\varepsilon}(s) \Big( R_{2\theta^{\varepsilon}(s)}\xi \Big) \cdot \Big( D K_{\R^2} [w^{\varepsilon}](s,h^{\varepsilon}(s)) \cdot \xi \Big) \, ds \\
&=& \frac{\varepsilon}{2} \left[ \Big( R_{2\theta^{\varepsilon}(s) - \frac{\pi}{2}}\xi \Big) \cdot \Big( D K_{\R^2}  [w^{\varepsilon}](s,h^{\varepsilon}(s)) \cdot \xi \Big)\right]_{0}^{t}
- \frac{\varepsilon}{2} \int_{0}^{t}  \Big( R_{2\theta^{\varepsilon}(s) - \frac{\pi}{2}} \xi \Big) \cdot \Big( \partial_{s}[ D K_{\R^2}  [w^{\varepsilon}](s,h^{\varepsilon}(s)) \cdot \xi]\Big) \, ds \\
&=& {\mathcal O}(\varepsilon) 
- \frac{\varepsilon}{2} \int_{0}^{t}  \Big( R_{2\theta^{\varepsilon}(s) - \frac{\pi}{2}} \xi \Big) \cdot \Big( D K_{\R^2} [\partial_{s} w^{\varepsilon}](s,h^{\varepsilon}(s)) \cdot \xi \Big) \, ds 
\\
&&- \frac{\varepsilon}{2} \int_{0}^{t}  \Big( R_{2\theta^{\varepsilon}(s) - \frac{\pi}{2}} \xi \Big) \cdot\Big( D^{2} K_{\R^2}  [w^{\varepsilon}](s,h^{\varepsilon}(s)) : (h^{\varepsilon})' \otimes \xi \Big)\, ds. 
\end{eqnarray*}
Using that $\partial_{s} w^\varepsilon = -\div (u^\varepsilon w^\varepsilon)$, the a priori estimates given by Lemma~\ref{LeminfVomega} on $u^\varepsilon$ and by \eqref{ConsOmega} on $w^\varepsilon$, and the harmonicity of $K_{\R^2}[u^\varepsilon w^\varepsilon]$ in $B(h^{\varepsilon}(s), 1/\rho)$, we see that the first integral is bounded by $Ct$, the second one by $C \int_{0}^{t} (1+ |\tilde{p}^{\varepsilon}|)$, which gives the result with Young's inequality. \par
\ \par
The fact that $F$ is weakly nonlinear can be seen when following the lines leading to \eqref{Enfin}: it is enough to see that $\overline{F}$ and $\hat{F}$ are weakly nonlinear.
Concerning $\hat{F}$, the result follows easily from the fact that $D K_{\R^2}[\omega^\eps](t,0)$ is bounded uniformly in $\varepsilon$ (see Lemma~\ref{LemModulesBornes}). Concerning $\overline{F}$ this comes from Proposition~\ref{PropositionPression2} and the fact that $\underline{F}$ and $F_{d}$ are themselves weakly nonlinear. For $\underline{F}$ this is visible from the expression \eqref{DefFSouligne}, the coefficients $a^{\varepsilon}$ and $b^{\varepsilon}$ being related to $D K_{\R^2}[\omega^\eps](t,0)$ by \eqref{a-et-b} and Lemma~\ref{LemModulesBornes} again; for $F_{d}$ this follows from Proposition~\ref{labelkoi2}.
\end{proof}

This ends the proof of Proposition~\ref{Pro:NormalForm}.
\end{proof}
\section{Modulated energy and passage to the limit}
\label{Sec:Passage}
\subsection{Modulated energy. Proof of Proposition~\ref{Pro:ModulatedEnergy}.}
\label{Subsec:ModulatedEnergy}
Let us fix $\rho>1$ and $\overline{T}>0$.  
Let $C>0$ and $\varepsilon_{0}>0$ be the constants obtained by Proposition~\ref{Pro:NormalForm}. 
Let $T \in (0,\overline{T}]$ and $\varepsilon \in (0,\varepsilon_{0}]$ such that \eqref{CondDistVorticiteSolide} is valid on $[0,T]$.
Then according to Proposition~\ref{Pro:NormalForm}, \eqref{Eq:NormalForm} is valid on $[0,T]$. 
Now to get the modulated energy estimate, we multiply \eqref{Eq:NormalForm} by $\tilde{p}^{\varepsilon}$; several simplifications occur in the process.
On the one hand, we use the symmetry of the matrices ${\mathcal M}_{g}$ (defined in \eqref{DefMG}) and ${\mathcal M}_{a}$ (defined in \eqref{defMa}) to observe the following:
\begin{equation*}
\tilde{p}^{\varepsilon} \cdot \big[ \varepsilon^\alpha {\mathcal M}_{g} + \varepsilon^2 {\mathcal M}_{a} \big] (\tilde{p}^{\varepsilon})'  
 = \big[ \varepsilon^\alpha {\mathcal E}_{g} (\tilde{p}^{\varepsilon}) + \varepsilon^2 {\mathcal E}_{a}  (\tilde{p}^{\varepsilon})  \big] ' ,
\end{equation*}
where we define, for any $p \in \R^3$,
\begin{equation*}
{\mathcal E}_{g} (p) := \frac{1}{2} p \cdot {\mathcal M}_{g} p \ \text{ and } \
{\mathcal E}_{a} (p) := \frac{1}{2} p \cdot {\mathcal M}_{a} p .
\end{equation*}
On the other hand, we have
\begin{equation*}
\tilde{p}^{\varepsilon} \cdot  \langle \varepsilon^{\alpha-1} \Lambda_{g} 
+ \varepsilon \Lambda_{a} , \tilde{p}^{\varepsilon}, \tilde{p}^{\varepsilon} \rangle =  0,
\end{equation*}
since the quadratic mappings $\Lambda_{g}$ and $\Lambda_{a}$ satisfy \eqref{AnnulationGammag} and \eqref{AnnulationGammaa} respectively; moreover, of course,
\begin{equation*}
\tilde{p}^{\varepsilon} \cdot  \big(\tilde{p}^{\varepsilon} \times {\bf B}  \big)= 0.
\end{equation*}
Thus we obtain on $[0,T]$, 
\begin{equation*}
\big[ \varepsilon^\alpha {\mathcal E}_{g} (\tilde{p}^{\varepsilon}) + \varepsilon^2 {\mathcal E}_{a}  (\tilde{p}^{\varepsilon})  \big]' (t) 
=  \varepsilon \gamma  \tilde{p}^{\varepsilon} \cdot G(\varepsilon,t) 
+ \varepsilon^{\min(\alpha,2)}\,  \tilde{p}^{\varepsilon} \cdot F(\varepsilon,t).
\end{equation*}
Next, integrating in time, we obtain for $t \in [0,T]$, 
\begin{equation*}
\big[ \varepsilon^\alpha {\mathcal E}_{g} (\tilde{p}^{\varepsilon}) + \varepsilon^2 {\mathcal E}_{a}  (\tilde{p}^{\varepsilon})  \big] (t) 
\leq \big[ \varepsilon^\alpha {\mathcal E}_{g} + \varepsilon^2 {\mathcal E}_{a} \big] (\tilde{p}^{\varepsilon}) (0) 
+ \varepsilon | \gamma |  \left| \int_0^t \tilde{p}^{\varepsilon}(s) \cdot G (\varepsilon,s)  \, ds \right| 
+  \varepsilon^{\min(\alpha,2)}\, \int_0^t \ \left| \tilde{p}^{\varepsilon}(s)  \right|  \cdot  \left|F(\varepsilon,s)  \right|  \, ds .
\end{equation*}
Moreover, from the asymptotic behaviour of the initial data we infer that 
\begin{equation*}
\big[ \varepsilon^\alpha {\mathcal E}_{g} + \varepsilon^2 {\mathcal E}_{a} \big] (\tilde{p}^{\varepsilon}) (0)  \leq C  \varepsilon^{\min(\alpha,2)} .
\end{equation*}
Now let us assume in a first time that ${\mathcal S}_{0}$ is not a ball. Then there exists $c>0$ depending only on $m^{1}$, ${\mathcal J}^{1}$ and ${\mathcal S}_{0}$ such that for any $p \in \R^3$,
\begin{equation*}
c^{-1} \, | p |^{2} \leq {\mathcal E}_{g} (p)  \leq  c  \, | p |^{2}  \, \text{ and }
c^{-1} \, | p |^{2} \leq {\mathcal E}_{a} (p)  \leq  c  \, | p |^{2} .
\end{equation*}
This is due to the fact, as we mentioned in Subsection~\ref{NRJ} that in this case ${\mathcal M}_{a}$ is a Gram matrix associated to a free family of vectors.
Therefore for  $t \in [0,T]$, 
\begin{equation*}
\left| \tilde{p}^{\varepsilon}(t)  \right|^{2} \leq C \Big(1+t+ \int_{0}^{t} ( \left| \tilde{p}^{\varepsilon}(s)  \right|^2 + \eps   \left| \tilde{p}^{\varepsilon}(s)   \right|^3 )  \, ds\Big) .
\end{equation*}
Then it suffices to use Lemma~\ref{LemEstNRJ2} and Gronwall's lemma on
\begin{equation*}
f(t):=\left| \tilde{p}^{\varepsilon}(t)  \right|^{2}  \leq C(1+\overline{T})+ 2C \int_{0}^t f(s)\, ds
\end{equation*}
to establish that on $[0,T]$,
\begin{equation*}
|\tilde{p}^{\varepsilon}(t)| \leq \sqrt{C(1+\overline{T})}e^{C \overline{T}}.
\end{equation*}
Finally it remains to use the a priori bounds on the modulation terms given in Lemma~\ref{LemModulesBornes} to deduce that \eqref{EstNRJMod} is valid. \par

In the degenerate case where ${\mathcal S}_{0}$ is a ball, then we recall that $r^{\varepsilon}(t)$ is constant over time (as seen from \eqref{Solide2}) and in particular bounded; moreover ${\mathcal M}_{a}$ is of the form 
\begin{equation*}
{\mathcal M}_{a}=\begin{pmatrix}
	m_{1,1} & 0 & 0 \\
	0 & m_{1,1} & 0 \\
	0 & 0 & 0
\end{pmatrix}.
\end{equation*}
Hence we can reason analogously as before using
\begin{equation*}
c^{-1} \, | \tilde{\ell}^{\varepsilon} |^{2} \leq {\mathcal E}_{a} (\tilde{p}^{\varepsilon})  \leq  c  \, | \tilde{\ell}^{\varepsilon} |^{2} ,
\end{equation*}
to deduce the boundedness of $\tilde{\ell}^{\varepsilon}$. Again we use the a priori bounds on the modulation terms of Lemma~\ref{LemModulesBornes} to conclude. This ends the proof of Proposition~\ref{Pro:ModulatedEnergy}. \qed
%
%
%
%
%
%
%
\subsection{Local passage to the limit. Proof of Propositions~\ref{Pro:TempsMinimal} and \ref{Pro:CVLocal}.}
\label{Subsec:LocalPTL}
In this section, we pass to the limit in the equation. To that purpose, we will come back at several stages to the original frame. In particular we now define  the inertia matrices in the original frame:
\begin{equation} \label{Mtheta}
\mathcal{M}^\eps_{a,\theta^\eps(t)} := Q_{\theta^\eps(t)}\mathcal{M}^\eps_{a} Q_{\theta^\eps(t)}^T
\qquad \text{and}  \qquad
\mathcal{M}^\eps_{\theta^\eps(t)} := Q_{\theta^\eps(t)}\mathcal{M}^\eps Q_{\theta^\eps(t)}^T =  \mathcal{M}^\eps_{g}
+ \mathcal{M}^\eps_{a,\theta^\eps(t)}
\end{equation}
where we recall that the added mass matrix $\mathcal{M}^\eps_{a}$ and the total mass matrix $\mathcal{M}^\eps$ in the solid frame are defined in \eqref{AddedMass} and \eqref{InertieMatrix} respectively and have constant coefficients with respect to time. In the above definition,  $Q_{\theta^\varepsilon}$ is defined as the rotation matrix of angle $\theta^\eps$ acting on the first two coordinates:
\begin{equation*} 
Q_{\theta^\eps}:=\begin{pmatrix}
\cos \theta^\eps & -\sin \theta^\eps & 0\\
\sin\theta^\eps & \cos\theta^\eps & 0\\
0 &0 &1
\end{pmatrix} .
\end{equation*} \par
We begin by establishing Proposition~\ref{Pro:TempsMinimal}. 
\begin{proof}[Proof of Proposition~\ref{Pro:TempsMinimal}]
We fix $\overline{T}>0$, which allows to define $T_{\varepsilon}$ according to \eqref{T-eps}.
Hence, we can apply Proposition~\ref{Pro:ModulatedEnergy} with $\rho = 2 \rho_{\overline{T}}$ and $\overline{T}$: there exist $C>0$ and $\varepsilon_{0}\in (0,1]$ such that for each $\varepsilon\in (0,\varepsilon_{0}]$, \eqref{EstNRJMod} applies on the interval $[0,T_{\varepsilon}]$, and in particular $|(h^{\varepsilon})'| \leq C$. On the other side, $w^\varepsilon$ is transported by the fluid velocity which is bounded uniformly on the vorticity support (see Lemma~\ref{LeminfVomega}) ; let us say, it is bounded by $C$ during $[0,T_{\varepsilon}]$. Then Proposition~\ref{Pro:TempsMinimal} follows immediately with for instance
\begin{equation*}
\underline{T}=\min \Big(\frac1{4C\rho_{\overline{T}}},\frac{\rho_{\overline{T}}}{2C}).
\end{equation*}
\end{proof}
Now we can turn to the proof of Proposition~\ref{Pro:CVLocal}.
\begin{proof}[Proof of Proposition~\ref{Pro:CVLocal}]
We proceed in several successive steps. \par
\ \par
\noindent
{\bf 1.} {\it Compactness for the solid velocity.} 
It is clear that on $[0,\underline{T}]$, the family $(h^{\varepsilon})$ is weakly-$\star$ relatively compact in $W^{1,\infty}([0,\underline{T}])$; hence we can pick a converging subsequence along indices $(\varepsilon_{n})_{n \in \N}$, say
\begin{equation} \label{ConvergenceH}
h^{\varepsilon_{n}} \longrightharpoonup h \text{ in } W^{1,\infty}([0,\underline{T}])-w\star \text{ as } n \rightarrow +\infty.
\end{equation}
\ \par
\noindent
{\bf 2.} {\it Compactness for the fluid velocity.} Let us now obtain some compactness for the fluid vorticity in the original frame, and for the velocity that it generates via the Biot-Savart law. We will obtain a convergence along a subsequence of $(\varepsilon_{n})$; to simplify the notations we will still call it $(\varepsilon_{n})$. \par
We extend ${w}^{\varepsilon} (t,\cdot)$ by $0$ inside ${\mathcal S}^{\varepsilon} (t)$.
Using the a priori estimate \eqref{ConsOmega}, we deduce that, up to a subsequence of $(\varepsilon_{n})$, one has, for some ${w} \in  L^{\infty}((0,\underline{T}) \times \R^2)$:
\begin{equation} \label{CVTildeW}
{w}^{\varepsilon_{n}} \longrightharpoonup {w} \ \text{ in } L^{\infty}((0,\underline{T}) \times \R^2) - w\star \ \text{ as } n \rightarrow +\infty.
\end{equation}
It follows moreover from \eqref{loglip2} that $K_{\R^{2}}[ w^{\eps}]$ is bounded in $L^{\infty}(0,\underline{T};{\mathcal{LL}}(\R^{2}))$. Since
\begin{equation} \label{vorty1Dprime}
\partial_t  w^{\eps} +  \div (u^{\eps} w^{\eps}) =0 \text{ in } {\mathcal D}'((0,\underline{T}) \times \R^{2}),
\end{equation}
we deduce from \eqref{ConsOmega} and Lemma~\ref{LeminfVomega} that $\partial_t  w^{\eps}$ is bounded in $L^{\infty}(0,\underline{T};W^{-1,p}(\R^{2}))$ for $p < +\infty$.
Hence we deduce by \cite[Appendix C]{lions} that the convergence \eqref{CVTildeW} can be improved into
\begin{equation*} 
{w}^{\varepsilon_{n}} \longrightarrow {w} \text{ in } C^{0}([0,\underline{T}];  L^{\infty}(\R^2)-w\star) \ \text{ as } n \rightarrow +\infty  .
\end{equation*}
Actually, \cite[Appendix C]{lions} considers only the compactness of a sequence in $C^{0}([0,\underline{T}];X-w)$ for $X$ a reflexive separable Banach space. However, the generalization to $C^{0}([0,\underline{T}]; L^{\infty}(\R^2)-w\star)$ is straightforward using the separability of $L^{1}(\R^{2})$.
Now since $K_{\R^{2}}$ is a compact operator form $L^{p}(\R^{2})$ to $L^{\infty}_{loc}(\R^{2})$ for $p>2$ we deduce that
\begin{equation} \nonumber
K_{\R^{2}} [{w}^{\varepsilon_{n}}] \longrightarrow K_{\R^{2}} [{w}] \ \text{ in } C^{0}([0,\underline{T}]; L^{\infty}_{loc}(\R^{2}))  \ \text{ as } n \rightarrow +\infty.
\end{equation}
{\bf 3.} {\it Convergence of $u^{\varepsilon} $.} We use Lemma~\ref{LemEstCheckvTildev2}, $\| \nabla \Phi_{i}^\varepsilon \|_{L^2(\mathcal{F}_{0}^\varepsilon)}=\mathcal{O}(\varepsilon^{1+\delta_{i\geq 3}})$ (see \eqref{phi-scaling}--\eqref{phi-scaling2}) and the uniform estimates on $\ell^{\varepsilon}$, $\varepsilon r^{\varepsilon}$, $K_{\R^{2}} [\omega^{\varepsilon}](t,0)$ and $D K_{\R^{2}} [\omega^{\varepsilon}](t,0)$ (see Proposition~\ref{Pro:ModulatedEnergy} and Lemma~\ref{LemModulesBornes}) to deduce that
\begin{equation*} 
 K_{\R^{2}}[\omega^{\varepsilon}]  - \tilde{v}^{\eps} \longrightarrow 0 \ \text{ in } L^{\infty}(0,\underline{T};L^{2}(\R^{2})) \text{ as } \varepsilon \rightarrow 0^{+}.
\end{equation*}
Here we have extended (for instance) $\tilde{v}^{\varepsilon}$ inside ${\mathcal S}_{0}^{\varepsilon}$ by $\ell^{\varepsilon} + r^\varepsilon x^\perp$.
Coming back to original variables, and recalling the notation \eqref{NotationRot2D}, we define $\tilde{u}^{\eps}$ by the relation
\begin{equation*} 
\tilde{u}^{\eps} (t,x) := R_{\theta^{\varepsilon}} \, \tilde{v}^\eps(t,R_{\theta^{\varepsilon}}^{T}(x-h^{\eps}(t))) ,
\end{equation*}
that is, $\tilde{u}^{\varepsilon}$ is the ``circulation-free'' part of $u^{\varepsilon}$:
\begin{equation} \nonumber
\tilde{u}^{\eps} = u^{\eps}- \gamma R_{\theta^{\varepsilon}} \, H^{\varepsilon}(R_{\theta^{\varepsilon}}^{T} (x-h^{\varepsilon}(t))) .
\end{equation}
We infer that
\begin{equation} \nonumber
K_{\R^{2}}[w^{\eps}]  - \tilde{u}^{\varepsilon} \longrightarrow 0 \ \text{ in } L^{\infty}(0,\underline{T};L^{2}(\R^{2})) \text{ as } \varepsilon \rightarrow 0^{+} ,
\end{equation}
where we have extended $\tilde{u}^{\varepsilon}$ in ${\mathcal S}^{\varepsilon}(t)$ by $(h^{\varepsilon})' + r^\varepsilon (x -h^{\varepsilon}(t))^{\perp}$. \par
Now we use the fact that for $p<2$, one has
\begin{equation} \label{CVHeps}
H^{\varepsilon_{n}}(\cdot) \longrightarrow H(\cdot) \ \text{ in } L^{p}_{loc}(\R^{2}) \ \text{ as } n \rightarrow +\infty,
\end{equation}
(see \cite[Lemma 3.11]{Cricri} and \cite[Lemma 4.2]{ift_lop_euler}), where $H$ is defined as
\begin{equation*}
H(x):= \frac{x^{\perp}}{2 \pi |x|^{2}}.
\end{equation*}
This involves that, as seen in \cite{GLS}, one has due to \eqref{ConvergenceH}:
\begin{equation} \label{CVWS1}
R_{\theta^{\varepsilon_{n}}} H^{\varepsilon_{n}}((R_{\theta^{\varepsilon_{n}}})^{T} (\cdot-h^{\varepsilon_{n}}(t))) \longrightarrow H(\cdot - h(t)) \ \text{ in } L^{\infty}(0,\underline{T};L^{p}_{loc}(\R^{2})) \ \text{ as } n \rightarrow +\infty,
\end{equation}
where we extend $H^{\varepsilon}$ by $0$ inside ${\mathcal F}^{\varepsilon}_{0}$. This convergence follows from \eqref{CVHeps} and a change of variable, noting that $H$ is invariant by rotation. \par
Gathering the above convergences, we obtain for $p<2$
\begin{equation*} 
{u}^{\varepsilon_{n}}(x) \longrightarrow K_{\R^{2}} [{w}] +\gamma H(\cdot - h(t)) 
\ \text{ in } L^{\infty}(0,\underline{T};L^{p}_{loc}(\R^{2})) \ \text{ as } n \rightarrow +\infty.
\end{equation*}
%
\noindent
{\bf 4.} {\it Fluid equation in the limit.}
Let us show that  $u$ and $w$  satisfy  \  \eqref{EqSolFaibleIntro}.
Since ${u}^{\varepsilon}$ and ${w}^{\varepsilon}$ satisfy \eqref{vorty1Dprime}, it can be easily seen that for any test function $\psi\in C^\infty_c([0,\underline{T})\times\R^2)$,
\begin{equation} \label{SolFEps}
\int_0^\infty\int_{\R^2} \psi_t  {w}^{\varepsilon_{n}} \, dx\, dt
+ \int_0^\infty \int_{\R^2} \nabla \psi \cdot {u}^{\varepsilon_{n}}  {w}^{\varepsilon_{n}} \, dx\, dt 
+ \int_{\R^2} \psi(0,x)  w_0(x) \, dx =0.
\end{equation}
Here we extend $u^{\varepsilon_{n}}$ and $w^{\varepsilon_{n}}$ inside ${\mathcal S}_{0}^{\varepsilon}$ as before; in particular $w^{\varepsilon_{n}}$ is extended by $0$. Then one can check the convergence as $n \rightarrow +\infty$ of each term in \eqref{SolFEps}: the convergence of the first one is a direct consequence of \eqref{CVTildeW} and the second one is a matter of weak/strong convergence. This allows to get \eqref{EulerPoint} on $(0,\underline{T})$. \par
\ \par
\noindent
{\bf 5.} {\it Solid equation in the limit.}
We will rely on the normal form  \eqref{Eq:NormalForm}. More precisely we first infer from \eqref{Eq:NormalForm}, using the form \eqref{DefG} of $G(\varepsilon,t)$, that 
\begin{equation} \label{AlmostThere}
\gamma
P_\flat Q_{\theta^\eps} (   \tilde{p}^{\varepsilon} \times {\bf B} )
= P_\flat Q_{\theta^\eps} \Big\{ \big[ \varepsilon^\alpha {\mathcal M}_{g} + \varepsilon^2 {\mathcal M}_{a} \big] (\tilde{p}^{\varepsilon})' 
+ \langle \varepsilon^{\alpha-1} \Lambda_{g} + \varepsilon \Lambda_{a} , \tilde{p}^{\varepsilon}, \tilde{p}^{\varepsilon} \rangle \Big\}
- \varepsilon^{\min(\alpha,2)} P_\flat Q_{\theta^\eps}  F(\varepsilon,t) ,
\end{equation}
where we define $P_{\flat}$ as the projection from $\R^{3}$ to $\R^{2}$ on the two first coordinates. \par
Now we claim to have the following identity
\begin{equation} \label{reviens}
 P_{\flat} Q_{\theta^\eps} \Big\{
\left( \eps^\alpha \mathcal{M}_{g}  + \eps^2  \mathcal{M}_{a}\right) (\tilde p^\eps)'
+\eps^{\alpha-1}\langle \Lambda_{g},\tilde p^\eps,\tilde p^\eps\rangle 
+ \eps \langle \Lambda_{a},\tilde p^\eps,\tilde p^\eps\rangle 
\Big\}
= P_{\flat} \left[ \left( \eps^\alpha \mathcal{M}_{g} Q_{\theta^\eps} 
+ \eps^2 Q_{\theta^\eps} \mathcal{M}_{a} \right) \tilde{p}^{\varepsilon} \right]',
\end{equation}
whose proof we temporarily postpone. This allows to simplify \eqref{AlmostThere} in the form
\begin{equation*}
\gamma P_\flat Q_{\theta^\eps} (   \tilde{p}^{\varepsilon} \times {\bf B} )
=  P_{\flat} \left[ \left( \eps^\alpha \mathcal{M}_{g} Q_{\theta^\eps} + \eps^2 Q_{\theta^\eps} \mathcal{M}_{a} \right) \tilde{p}^{\varepsilon} \right]'
- \varepsilon^{\min(\alpha,2)} P_\flat Q_{\theta^\eps}  F(\varepsilon,t) .
\end{equation*}
Now we see that terms in the right hand side converge to $0$ in $W^{-1,\infty}([0,\underline{T}])$. For the last term, the convergence is clear (and even strong in $L^{\infty}$), due to \eqref{IneqWNL}. The first term in the right hand side is composed of time-derivatives of bounded terms in $L^{\infty}$ (according to Proposition~\ref{Pro:ModulatedEnergy} and Lemma~\ref{LemModulesBornes}) multiplied by positive powers of $\varepsilon$. Hence this term converges also to $0$ in $W^{-1,\infty}$. \par
Hence at this stage we know that $P_\flat Q_{\theta^\eps_{n}} ( \tilde{p}^{\varepsilon_{n}} \times {\bf B} )$ converges to $0$ in $W^{-1,\infty}$. Now we have from the definition of ${\bf B}$ \eqref{DefB} and relations \eqref{KomKw}--\eqref{nablaK} and \eqref{vprod} that
\begin{equation*}
P_\flat  Q_{\theta^\eps} (   \tilde{p}^{\varepsilon} \times {\bf B} ) = 
- \varepsilon r^{\varepsilon} R_{\theta^\eps} \xi +
\Big((h^{\varepsilon} )' - K_{\R^{2}} [ w^{\varepsilon}] ( t,h^{\varepsilon} ) - \eps D K_{\R^{2}}[w^{\varepsilon}](t,h^{\varepsilon}) \cdot R_{\theta^\varepsilon(t)}\xi \Big)^{\perp}
 .
\end{equation*}
Using \eqref{DeriveeRotation} and Lemma~\ref{LemModulesBornes}, we observe that as $\varepsilon \rightarrow 0^{+}$,
\begin{equation*}
\varepsilon r^{\varepsilon} R_{\theta^{\eps}} \xi = -\varepsilon (R_{\theta^{\eps}} \xi^{\perp})' \longrightarrow 0 \text{ in } W^{-1,\infty}([0,\underline{T}]) 
\ \text{ and } \ 
\eps D K_{\R^{2}}[w^{\varepsilon}](\cdot,h^{\varepsilon}) \cdot R_{\theta^{\varepsilon}}\xi \longrightarrow 0 \text{ in } L^{\infty}(0,\underline{T}). 
\end{equation*}
Hence we reach the convergence
\begin{equation*}
(h^{\varepsilon_{n}} )' - K_{\R^{2}} [ w^{\varepsilon_{n}}] (\cdot, h^{\varepsilon_{n}} ) \longrightarrow 0 \ \text{ in } \ W^{-1,\infty}([0,\underline{T}]). 
\end{equation*}
Now we use the uniform estimates on $K_{\R^2}[w^\eps]$ in $L^{\infty}(0,\underline{T};{\mathcal{LL}}(\R^{2}))$ and \eqref{ConvergenceH} to get that
\begin{equation*}
K_{\R^2}[w^\eps_{n}](\cdot,h^{\varepsilon_{n}}) \longrightarrow
K_{\R^2}[{w}](\cdot,h) \ \text{ in } \ L^{\infty}(0,\underline{T}) \ \text{ as } \ n \rightarrow +\infty.
\end{equation*}
So we finally obtain
\begin{equation*} 
(h^{\varepsilon_{n}} )' \longrightarrow  K_{\R^{2}} [ w] (\cdot, h ),
\end{equation*}
in $W^{-1,\infty}([0,\underline{T}])$ and then in $L^{\infty}([0,\underline{T}])$ weak-$\star$ due to the boundedness of $(h^{\varepsilon})'$ in $L^{\infty}(0,\underline{T})$. \par
It follows that the solutions of our system converge, up to a subsequence, to those of the wave/vortex system. Due to the uniqueness of the latter, the convergence takes place along the whole family $(\varepsilon\in (0,\varepsilon_{0}])$. Now it remains only to prove \eqref{reviens}. \par
\ \par
\noindent
{\it 6. Proof of the claim \eqref{reviens}.} We introduce
\begin{equation*}
J_{3}:=\begin{pmatrix}
 0 & -1 & 0\\
 1 & 0&0\\
 0 & 0 &0
\end{pmatrix},
\end{equation*}
so that \eqref{DeriveeRotation} is recast as
\begin{equation} \label{DeriveeRotation2}
Q_{\theta^{\varepsilon}} ' = r^{\varepsilon} Q_{\theta^{\varepsilon}} J_{3}.
\end{equation}
According to \eqref{Mtheta} and \eqref{DeriveeRotation2}, we have
\begin{align*}
\left[ \mathcal{M}^\eps_{\theta^\eps} Q_{\theta^\eps}  \begin{pmatrix} \tilde \ell^\eps \\  r^\eps \end{pmatrix} \right]' = &  \left[ Q_{\theta^\eps}\mathcal{M}^\eps\begin{pmatrix} \tilde \ell^\eps \\  r^\eps \end{pmatrix}\right]'
=   r^\eps Q_{\theta^\eps} J_{3} \mathcal{M}^\eps  \begin{pmatrix} \tilde \ell^\eps \\  r^\eps \end{pmatrix} + Q_{\theta^\eps}\mathcal{M}^\eps \begin{pmatrix} \tilde \ell^\eps \\  r^\eps \end{pmatrix}'    \nonumber
\\
=&r^\eps Q_{\theta^\eps}
\begin{pmatrix}
\left((m^\eps {\rm Id}+\mathcal{M}_{\flat}^\eps) \tilde\ell^\eps\right)^\perp + r^\eps \begin{pmatrix}-m_{2,3}^\eps\\ m_{1,3}^\eps\end{pmatrix} \\
0
\end{pmatrix}+ Q_{\theta^\eps}\mathcal{M}^\eps \begin{pmatrix} \tilde \ell^\eps \\  r^\eps \end{pmatrix}',
\end{align*}
where we recall that $\mathcal{M}_{\flat}^\eps$ is the  $2\times2$ restriction of $\mathcal{M}_{a}^\eps$.
Now we extract the powers of $\varepsilon$ of the mass coefficients. Thanks to \eqref{AddedMass} and \eqref{raoul} we have that
\begin{gather*} 
\mathcal{M}^\eps_{\theta^\eps} = I_\eps \left( \eps^\alpha \mathcal{M}_{g}  + \eps^2  \mathcal{M}_{a,\theta^\eps}\right)    I_{\eps} , \ 
m^{\varepsilon} = \varepsilon^{\alpha} m^{1}, \ 
\mathcal{M}_{\flat}^\varepsilon = \varepsilon^{2} \mathcal{M}_{\flat}^1, \\
\begin{pmatrix}-m_{2,3}^\eps\\ m_{1,3}^\eps\end{pmatrix} = \varepsilon^{3} \begin{pmatrix}-m_{2,3}^1\\ m_{1,3}^1 \end{pmatrix}, \ 
\mathcal{M}^\eps= I_\eps( \eps^\alpha \mathcal{M}_{g}  + \eps^2  \mathcal{M}_{a}) I_{\varepsilon},
\end{gather*}
where we recall that $I_{\varepsilon}$ is defined in \eqref{defIeps}. It follows that
\begin{align*}
\left[ I_\eps \left( \eps^\alpha \mathcal{M}_{g}  + \eps^2  \mathcal{M}_{a,\theta^\eps}\right) I_{\eps} Q_{\theta^\eps}  \begin{pmatrix} \tilde \ell^\eps \\  r^\eps \end{pmatrix} \right]' 
=& r^\eps Q_{\theta^\eps}
\begin{pmatrix}
\left((\varepsilon^{\alpha} m^1 {\rm Id}+ \varepsilon^{2} \mathcal{M}_{\flat}^1) \tilde\ell^\eps\right)^\perp + r^\eps \varepsilon^{3} \begin{pmatrix}-m_{2,3}^1\\ m_{1,3}^1 \end{pmatrix} \\
0 \end{pmatrix} \\
&+ Q_{\theta^\eps}I_\eps \left( \eps^\alpha \mathcal{M}_{g}  + \eps^2  \mathcal{M}_{a}\right) I_{\eps} \begin{pmatrix} \tilde \ell^\eps \\  r^\eps \end{pmatrix}'.
\end{align*}
We recognize that
\begin{equation*}
\begin{pmatrix}-m_{2,3}^1\\ m_{1,3}^1 \end{pmatrix} = (P_{\flat} \mu^{1})^{\perp}
\ \text{ and } \ 
I_{\varepsilon} \begin{pmatrix} \tilde \ell^\eps \\  r^\eps \end{pmatrix} = \tilde{p}^{\varepsilon}.
\end{equation*}
Projecting on the first two coordinates (recall that we denote $P_{\flat}$ this projection), and using $I_{\varepsilon} Q_{\theta^\eps}= Q_{\theta^\eps} I_{\varepsilon}$, and $P_{\flat} I_{\varepsilon} = P_{\flat}$ we deduce
\begin{equation*} 
P_{\flat} \left[ \left( \eps^\alpha \mathcal{M}_{g} + \eps^2 \mathcal{M}_{a,\theta^\eps} \right)  Q_{\theta^\eps} \tilde{p}^{\varepsilon} \right]'  \\
=   P_{\flat} Q_{\theta^\eps} \left[ 
\eps^{\alpha-1}\langle \Lambda_{g},\tilde p^\eps,\tilde p^\eps\rangle 
+ \eps \langle \Lambda_{a},\tilde p^\eps,\tilde p^\eps\rangle 
+ \left( \eps^\alpha \mathcal{M}_{g}  + \eps^2  \mathcal{M}_{a}\right) (\tilde p^\eps)'
\right],
\end{equation*}
with $\Lambda_{g}$ and $\Lambda_{a}$ defined in \eqref{DefGammag} and \eqref{DefGammaa}.
This ends the proof of \eqref{reviens} and hence of Proposition~\ref{Pro:CVLocal}.
\end{proof}
%
%
%
%
%
%
%
%
\subsection{Conclusion}
\label{Subsec:GlobalPTL}
It remains to prove that the convergence obtained above does not take place only during the interval $[0,\underline{T}]$, but during any interval $[0,T]$, $T>0$.
We know that the solution $(h,w)$ of the vortex/wave system is global in time, and that $h$ remains for all $t$ at positive distance from $\supp w$ and that this support remains bounded.
Let us be given $T>0$. We let $\overline{T}:=T+1$ and we know from \eqref{supportcontrol}  that there exists $\rho_{\overline{T}}>1$ such that 
\begin{equation} \label{Eq:DBarre}
\forall t \in [0,T+1], \ \ 
\supp w(t,\cdot) \subset B(h(t),\rho_{\overline{T}}) \setminus B(h(t),1/\rho_{\overline{T}}).
\end{equation}
Let us recall here the definition \eqref{T-eps} of $T_{\varepsilon}$:
\begin{equation*} 
T_{\varepsilon}:= \sup\Big\{ \tau \in [0,\overline{T}],\ \forall t\in [0,\tau], \ 
\supp w^\varepsilon(t) \subset B(h^\varepsilon(t),2\rho_{\overline{T}}) \setminus B(h^\varepsilon(t),1/(2\rho_{\overline{T}}))
 \Big\}.
\end{equation*}
Using Proposition~\ref{Pro:TempsMinimal}, it is clear that for suitably small $\varepsilon_{0}$, one has $\inf_{ \varepsilon \in (0,\varepsilon_{0}]} T_{\varepsilon} >0$.
Therefore, we denote
$$\tilde{T}:=\liminf_{\varepsilon \rightarrow 0^+} T_{\varepsilon} >0 .$$
Due to Proposition~\ref{Pro:ModulatedEnergy} (with $\rho=2\rho_{\overline{T}}$), there exists $C>0$ such that for all $\varepsilon \in (0,\varepsilon_{0})$, one has the following estimate
\begin{equation} \label{estunif2}
|\ell^{\varepsilon}| + | \varepsilon r^{\varepsilon} | \leq C \ \text{ for } t \in [0,T_{\varepsilon}].
\end{equation}
Now we claim that
\begin{equation} \label{TTilde}
\tilde{T}  \geq T + \frac{1}{2},
\end{equation}
and reason by contradiction. If it were not the case and $\tilde{T} < T + \frac{1}{2}$, then the following would occur. First we extract a subsequence $(\varepsilon_{n})$ such that $T_{\varepsilon_{n}} \rightarrow \tilde{T}$.
On any compact interval $[0,\tilde{T}-\eta]$ of $[0,\tilde{T})$ with $\eta>0$, one finds $N$ such that $\inf_{n \geq N} T_{\varepsilon_{n}} \geq \tilde{T}-\eta$. Therefore we can apply Proposition~\ref{Pro:CVLocal} with $\underline{T}= \tilde{T}-\eta$ and the convergences of the previous subsection hold true. \par
Now we prove that
\begin{equation} \label{CVSupport}
\| d_{H}(\supp w^{\varepsilon_{n}}(t,\cdot), \supp w(t,\cdot)) \|_{L^{\infty}(0,\tilde{T}-\eta)} \longrightarrow 0 \text{ as } n \rightarrow +\infty,
\end{equation}
where $d_{H}$ is the Hausdorff distance. To see that, one modifies \eqref{CVWS1} into 
\begin{equation*} 
R_{\theta^{\varepsilon_{n}}} H^{\varepsilon_{n}}R_{-\theta^{\varepsilon_{n}}} (\cdot-h^{\varepsilon_{n}})) \longrightarrow H(\cdot - h)
 \ \text{ in } L^{\infty} (B(0,2 \rho_{\overline{T}}) \setminus B(0,1/(2 \rho_{\overline{T}}))) \ \text{ as } n \rightarrow +\infty .
\end{equation*}
This follows from the invariance of $H$ with respect to rotations, and the fact that we have uniform estimates on $\nabla H^{\varepsilon_{n}}$ outside $B(0,1/2\rho_{\overline{T}})$ for $n$ large enough (see Lemma~\ref{LemHomega}). \par
Thanks to the uniform $\log$-Lipschitz estimates and Gronwall's lemma, we see that the flows $\Phi^{\varepsilon}_{n}=\Phi^{\varepsilon}_{n}(t,x)$ and $\Phi=\Phi(t,x)$ associated respectively to  ${u}^{\varepsilon_{n}}$  and $K_{\R^{2}} [{w}](t,\cdot) +\gamma H(\cdot - h(t))$ satisfy
\begin{equation*}
\Phi^{\varepsilon_{n}} \longrightarrow \Phi \text{ uniformly in } \ [0,\tilde{T} - \eta] \times \supp w_{0} \text{ as } n \rightarrow +\infty.
\end{equation*}
This involves \eqref{CVSupport}. Now due to the definition of $T_{\varepsilon}$ and the fact that $T_{\varepsilon_{n}}\leq \tilde T + 1/2 < \overline{T}$ for $n$ large enough,  we have that
\begin{equation*}
\supp w^{\varepsilon_{n}}(T_{\varepsilon_{n}},\cdot) \cap \Big[  B(h^{\varepsilon_{n}}(T_{\varepsilon_{n}}),2\rho_{\overline{T}}) \setminus \overline{B}(h^{\varepsilon_{n}}(T_{\varepsilon_{n}}),1/(2\rho_{\overline{T}})) \Big]^c \neq \emptyset.
\end{equation*}
Hence using that $(h^{\varepsilon_{n}})'$ is uniformly bounded \eqref{estunif2} and that $w^{\varepsilon_{n}}$ is transported by a velocity uniformly bounded as well (see Lemma~\ref{LeminfVomega}), we can find $\eta>0$ small enough so that, for $n$ large enough
\begin{equation*}
\supp w^{\varepsilon_{n}}(\tilde{T}-\eta,\cdot) \cap \Big[  B(h^{\varepsilon_{n}}(\tilde{T}-\eta),3\rho_{\overline{T}}/2) \setminus B(h^{\varepsilon_{n}}(\tilde{T}-\eta),2/(3\rho_{\overline{T}})) \Big]^c \neq \emptyset.
\end{equation*}
Now we pass to the limit as $n \rightarrow +\infty$ with \eqref{CVSupport}, and get a contraction with the definition \eqref{Eq:DBarre} of $\rho_{\overline{T}}$. This proves \eqref{TTilde}. \par
Finally, one can apply the convergence result of Proposition~\ref{Pro:CVLocal} on $[0,T]$. This ends the proof of Theorem~\ref{MR}. \qed
%
%
%
%
%
%
%
%
%
%
\appendix
\section{Lemmas from complex analysis} \label{complex ana}
In this section, we gather several computations relying on complex analysis, which were used earlier in the proof. We recall the notation $\widehat f = f_{1}-if_{2}$ for any $f=(f_{1},f_{2})$, whereas $z=x_{1}+ix_{2}$. 
We will in particular compute the first coefficients in the Laurent series of $\widehat{\nabla \Phi_{i}^\varepsilon}$ (defined in Paragraph~\ref{ParKirchoff}).
\subsection{Basic computations}
We first recall the Blasius lemma about tangent vector fields (for a proof, see for instance \cite{MP} or \cite{GLS}):
\begin{Lemma} \label{blasius}
Let  $\mathcal{C}$ be a smooth Jordan curve in the plane, $f:=(f_1 , f_2)$ and $g:=(g_1 ,g_2 )$ two smooth tangent vector fields on $\mathcal{C}$. 
Then 
\begin{eqnarray*} 
\int_{ \mathcal{C}} (f  \cdot g) n \, ds =  i \left( \int_{ \mathcal{C}} \widehat f \widehat g \, dz \right)^* , \\
\int_{ \mathcal{C}} (f  \cdot g) (x^{\perp} \cdot n)  \,ds =  \Re \left( \int_{ \mathcal{C}} z  \widehat f \widehat g \, dz \right).
\end{eqnarray*}
\end{Lemma}
Now we relate (classically) $\int_{ \mathcal{C}} \widehat f$ with the flux and the circulation of $f$:
\begin{Lemma} \label{lem-flux-circu}
Let  $\mathcal{C}$ be a smooth Jordan curve in the plane and $f:=(f_1 , f_2)$ a smooth vector field on $\mathcal{C}$, then:
\begin{equation*}
\int_{ \mathcal{C}} \widehat f  \, dz =  \int_{ \mathcal{C}} f\cdot \tau \, ds  -i \int_{ \mathcal{C}} f\cdot n \, ds.
\end{equation*}
\end{Lemma}
\begin{proof}
Denoting by $\gamma=(\gamma_{1},\gamma_{2})$ a parametrization of $\mathcal{C}$ then $\tau=(\gamma_{1}',\gamma_{2}')/|\gamma'|$, then we can write $n=(-\gamma_{2}',\gamma_{1}')/|\gamma'|$, $ds = |\gamma'(t)|dt$ and $dz = (\gamma_{1}'(t)+i\gamma_{2}'(t)) dt$. Now the conclusion simply follows from
\begin{equation*}
\int_{ \mathcal{C}} (f_1 - if_2)  \, dz = \int (f_{1} \gamma_{1}' + f_{2} \gamma_{2}')\, dt -i  \int (-f_{1} \gamma_{2}' + f_{2} \gamma_{1}')\, dt .
\end{equation*}
\end{proof}
\begin{Corollary} \label{cor}
Let  $\mathcal{C}$ be a smooth Jordan curve in the plane and $f:=(f_1 , f_2)$ a smooth vector field on $\mathcal{C}$, then:
\begin{eqnarray*}
\int_{ \mathcal{C}} z\widehat f  \, dz &=&  \int_{ \mathcal{C}} \begin{pmatrix} x\cdot f \\ x^\perp \cdot f \end{pmatrix} \cdot \tau \, ds  -i \int_{ \mathcal{C}} \begin{pmatrix} x\cdot f \\ x^\perp \cdot f \end{pmatrix} \cdot n \, ds\\
&=& \int_{ \mathcal{C}} (x_{1} +i x_{2}) (f\cdot \tau)  \, ds -i \int_{ \mathcal{C}} (x_{1} +i x_{2}) (f\cdot n)  \, ds.
\end{eqnarray*}
\end{Corollary}
\begin{proof}
We apply Lemma~\ref{lem-flux-circu} to $g$ given by $g_{1}-ig_{2}=z(f_1 - if_2)$ and one checks that
\begin{equation*}
(x_{1}+ix_{2})(f_{1}-if_{2})= (x_{1} f_{1} + x_{2} f_{2})-i(-x_{2}f_{1}+x_{1}f_{2})= (x\cdot f) -i(x^\perp \cdot f).
\end{equation*}
To obtain the second equality, we simply use $(n_{1},n_{2})=(-\tau_{2},\tau_{1})$:
\begin{eqnarray*}
\begin{pmatrix} x\cdot f \\ x^\perp \cdot f \end{pmatrix} \cdot \tau  -i \begin{pmatrix} x\cdot f \\ x^\perp \cdot f \end{pmatrix} \cdot n &=&  (x_{1}f_{1}+x_{2}f_{2})\tau_{1} +  (-x_{2}f_{1}+x_{1}f_{2})\tau_{2} -i(x_{1}f_{1}+x_{2}f_{2})n_{1} -i  (-x_{2}f_{1}+x_{1}f_{2})n_{2}\\
&=& (x_{1}+ix_{2})  (f_{1}\tau_{1}+f_{2}\tau_{2}) -i(x_{1}+ix_{2})(f_{1}n_{1}+f_{2}n_{2}),
\end{eqnarray*}
which ends the proof.
\end{proof}
\begin{Remark}\label{Rembarz}
Replacing $x_{2}$ by $-x_{2}$ in the previous computation, we obtain
\begin{equation*}
\int_{ \mathcal{C}} \bar{z} \widehat f  \, dz
=\int_{ \mathcal{C}} (x_{1} -i x_{2}) (f\cdot \tau)  \, ds -i \int_{ \mathcal{C}} (x_{1} -i x_{2}) (f\cdot n)  \, ds.
\end{equation*}
\end{Remark}
A straightforward application of Lemma~\ref{lem-flux-circu}, Corollary~\ref{cor} and Remark~\ref{Rembarz} is the following (using $H^1 \cdot n = 0$ on $\partial {\mathcal S}_{0}$).
\begin{Lemma} \label{LemzH}
Let $H^1$ be defined in \eqref{DefHeps}, then
\begin{equation*}
\left(  \int_{ \partial  \mathcal{S}_{0}} \bar{z}\widehat{H^1}  \, dz \right)^* = \int_{ \partial  \mathcal{S}_{0}} z \widehat{H^1}  \, dz
\qquad
\text{and}
\qquad
\Re \left( i \int_{ \partial  \mathcal{S}^1_{0}} |z|^2 \widehat{H^1}  \, dz \right)=0.
\end{equation*}
\end{Lemma}
\ \par
Next in the following lemma, we regroup elementary computations which will intervene in the next subsections. These are of the form:
\begin{equation} \label{IntegralesStandard}
\int_{ \partial  \mathcal{S}^{\varepsilon}_{0} } P(x_{1},x_{2}) K_{j} \, ds = \varepsilon^{d+\delta_{j \geq 3} +1}\int_{ \partial  \mathcal{S}_{0} } P(x_{1},x_{2}) K_{j} \, ds,
\end{equation}
where $P$ is a homogeneous polynomial of degree $d$ in variables $x_{1}$ and $x_{2}$ and where $j \in \{1, \dots, 5 \}$. We recall that $|\mathcal{S}^\eps_{0} |$ is the Lebesgue measure of $\mathcal{S}^\eps_{0}$ and that $x_{G}^\eps$ is the position of the geometrical center of $\mathcal{S}^\eps_{0}$ (see \eqref{DefXg}). Additionally, we introduce the following coefficients
\begin{equation*}
m_{6}^{\eps}:= \int_{\mathcal{S}^\eps_{0}} (x_1^2-x_{2}^2 ) \, dx, \quad
m_{7}^{\eps}:= 2 \int_{\mathcal{S}^\eps_{0}} x_{1}x_2  \, dx, \quad
m_{8}^{\eps} := \int_{\mathcal{S}^\eps_{0}}  |x|^2  \, dx,
\end{equation*}
and notice that $m_{i}^{\eps} = \varepsilon^{4} m_{i}^{1}$ for $i=6,7,8$.
\begin{Lemma} \label{CalculIntegralesStandard}
One has the following relations for \eqref{IntegralesStandard}: \par
\ \par \noindent
$\bullet \ d(P)=0, \ j \in \{1,\dots,5\}:$
\begin{equation*} 
\int_{ \partial  \mathcal{S}^{\varepsilon}_{0} } n_{1}\,ds
= \int_{ \partial  \mathcal{S}^{\varepsilon}_{0} } n_{2}\,ds
= \int_{ \partial  \mathcal{S}^{\varepsilon}_{0} } x^\perp \cdot n\,ds 
= \int_{ \partial  \mathcal{S}^{\varepsilon}_{0} } \begin{pmatrix} x_{2}\\x_{1} \end{pmatrix}  \cdot n\,ds
= \int_{ \partial  \mathcal{S}^{\varepsilon}_{0} } \begin{pmatrix} -x_{1}\\x_{2} \end{pmatrix}  \cdot n\,ds =0.
\end{equation*}
$\bullet \ d(P)=1, \ j \in \{1,2,3\}:$
\begin{equation*}
\int_{ \partial  \mathcal{S}^\eps_{0}} x_{j} n_{i}  \, ds
=-\delta_{i,j} |\mathcal{S}^\eps_{0} | \text{ for } i,j=1,2, \ 
\int_{ \partial  \mathcal{S}^\eps_{0}} x_{1} (x^\perp \cdot n)  \, ds 
=  |\mathcal{S}^\eps_{0} |x_{G,2}^\eps, \ \
\int_{ \partial  \mathcal{S}^\eps_{0}} x_{2} (x^\perp \cdot n)  \, ds
= - |\mathcal{S}^\eps_{0} |x_{G,1}^\eps .
\end{equation*}
$\bullet \ d(P)=1, \ j \in \{4,5\}:$
\begin{eqnarray} \nonumber
\int_{ \partial  \mathcal{S}^\eps_{0}} x_{1} \, \begin{pmatrix}-x_{1}\\x_{2}\end{pmatrix} \cdot n  \, ds
=  |\mathcal{S}^\eps_{0} |x_{G,1}^\eps, \quad
\int_{ \partial  \mathcal{S}^\eps_{0}} x_{2} \, \begin{pmatrix}-x_{1}\\x_{2}\end{pmatrix} \cdot n \, ds 
= - |\mathcal{S}^\eps_{0} |x_{G,2}^\eps, \\
\nonumber
\int_{ \partial  \mathcal{S}^\eps_{0}} x_{1} \, \begin{pmatrix}x_{2}\\x_{1}\end{pmatrix} \cdot n  \, ds
= - |\mathcal{S}^\eps_{0} |x_{G,2}^\eps, \quad
\int_{ \partial  \mathcal{S}^\eps_{0}} x_{2} \, \begin{pmatrix}x_{2}\\x_{1}\end{pmatrix} \cdot n  \, ds
= - |\mathcal{S}^\eps_{0} |x_{G,1}^\eps.
\end{eqnarray}
$\bullet \ d(P)=2, \ j \in \{1,2,3\}:$
\begin{gather} \nonumber
\int_{\partial  \mathcal{S}^\eps_{0}  }|x|^2 n_{i}  \, ds  =-2x_{G,i}^\eps |\mathcal{S}^\eps_{0} | \text{ for }i=1,2, \quad
\int_{\partial  \mathcal{S}^\eps_{0}  }|x|^2x^\perp\cdot n  \, ds =0, \\
\nonumber
\int_{ \partial  \mathcal{S}^\eps_{0}} x_{1}x_{2}n_{1}  \, ds  =- |\mathcal{S}^\eps_{0} |x_{G,2}^\eps, \quad
\int_{ \partial  \mathcal{S}^\eps_{0}} x_{1}x_{2}n_{2}  \, ds = - |\mathcal{S}^\eps_{0} |x_{G,1}^\eps, \quad
\int_{ \partial  \mathcal{S}^\eps_{0}} x_{1}x_{2}x^\perp\cdot n \, ds = - m_{6}^{\eps} , \\
\nonumber
\int_{ \partial  \mathcal{S}^\eps_{0}} (x_{1}^2-x_{2}^2)n_{1}  \, ds = -2 |\mathcal{S}^\eps_{0}| x_{G,1}^\eps, \quad
\int_{ \partial  \mathcal{S}^\eps_{0}} (x_{1}^2-x_{2}^2)n_{2}  \, ds = 2 |\mathcal{S}^\eps_{0}| x_{G,2}^\eps, \quad
\int_{ \partial  \mathcal{S}^\eps_{0}} (x_{1}^2-x_{2}^2)x^\perp\cdot n  \, ds =2m_{7}^{\eps} .
\end{gather}%
$\bullet \ d(P)=2, \ j \in \{4,5\}:$
\begin{gather}
\nonumber
\int_{ \partial  \mathcal{S}^\eps_{0}} |x|^2 \, \begin{pmatrix}-x_{1}\\x_{2}\end{pmatrix} \cdot n  \, ds = 2m_6^{\eps}, \quad
\int_{ \partial  \mathcal{S}^\eps_{0}} |x|^2 \, \begin{pmatrix}x_{2}\\x_{1}\end{pmatrix} \cdot n \, ds =-2m_{7}^{\eps}, \\
\nonumber
\int_{ \partial  \mathcal{S}^\eps_{0}} x_{1}x_{2} \begin{pmatrix} -x_{1}\\ x_{2}\end{pmatrix}\cdot n \, ds  =0 , \quad
\int_{ \partial  \mathcal{S}^\eps_{0}} x_{1}x_{2} \begin{pmatrix} x_{2}\\ x_{1}\end{pmatrix}\cdot n \, ds =-m_{8}^{\eps}, \\
\nonumber
\int_{ \partial  \mathcal{S}^\eps_{0}} (x_{1}^2-x_{2}^2) \begin{pmatrix} -x_{1}\\ x_{2}\end{pmatrix} \cdot n  \, ds=2m_{8}^{\eps} , \quad
\int_{ \partial  \mathcal{S}^\eps_{0}} (x_{1}^2-x_{2}^2) \begin{pmatrix} x_{2}\\ x_{1}\end{pmatrix} \cdot n  \, ds=0 .
\end{gather}%
\end{Lemma}
\begin{proof}
These relations are straightforward consequences of the divergence theorem inside $\partial {\mathcal S}_{0}^{\varepsilon}$ (recalling that $n$ points outside the fluid), for instance
\begin{equation*}
\int_{ \partial  \mathcal{S}^\eps_{0}} x_{1} \, \begin{pmatrix}-x_{1}\\x_{2}\end{pmatrix} \cdot n  \, ds
=  -\int_{  \mathcal{S}_{0}^\eps } \div  \begin{pmatrix}-x_{1}^2\\x_{1}x_{2}\end{pmatrix}\,dx
= \int_{  \mathcal{S}_{0}^\eps } x_{1}\, dx =
 |\mathcal{S}^\eps_{0} |x_{G,1}^\eps.
\end{equation*}
\end{proof}
We finish this section with some other basic computations in the complex variable.
\begin{Lemma} \label{compuC}
We have the following relations:
\begin{gather}
\label{compu1}
\int_{ \partial  \mathcal{S}^\eps_{0}} \bar{z}^2  \, dz = 4 |\mathcal{S}^\eps_{0} |x_{G,2}^\eps+4i|\mathcal{S}^\eps_{0} |x_{G,1}^\eps, \\
\label{compu2}
\int_{ \partial  \mathcal{S}^\eps_{0}} \bar{z}  \, dz= 2 i |\mathcal{S}^\eps_{0} |, \\
\label{compu3}
\int_{ \partial  \mathcal{S}^\eps_{0}} |z|^2  \, dz = -2 |\mathcal{S}^\eps_{0} |x_{G,2}^\eps+2 i|\mathcal{S}^\eps_{0} |x_{G,1}^\eps , \\
\label{compu4}
\int_{ \partial  \mathcal{S}^\eps_{0}} z|z|^2  \, dz =-2m_{7}^{\eps}+2im_{6}^{\eps}.
\end{gather}
\end{Lemma}
\begin{proof}
These as elementary consequences of Stokes' formula for the operator $\frac{\partial}{\partial \overline{z}} = \frac{1}{2}(\frac{\partial}{\partial x} - \frac{1}{i} \frac{\partial}{\partial y})$:
\begin{equation*}
\int_{\partial {\mathcal S}_{0}^{\varepsilon}} u \, dz = 2 i \int_{{\mathcal S}_{0}^{\varepsilon}} \frac{\partial u}{\partial \overline{z}} \, dx \, dy.
\end{equation*}
\end{proof}
\subsection{First coefficients of the Laurent series}
\label{SubsecFirstCoefs}
In this subsection, we use these lemmas to compute the first coefficients of the Laurent series associated to $\widehat{\nabla \Phi^\eps_{i}}$ for $i=1,2,3,4,5$. \par
We begin with an elementary lemma.
\begin{Lemma} \label{LemmeCircu}
The functions $\widehat{\nabla \Phi^1_{i}}$ admit Laurent series at infinity and moreover satisfy for $i=1,2,3,4,5$:
\begin{equation*}
\widehat{\nabla \Phi^1_{i}} = {\mathcal O}(1/z^{2}) \ \text{ as } \ |z| \rightarrow +\infty.
\end{equation*}
\end{Lemma}
\begin{proof}
That $\widehat{\nabla \Phi^1_{i}}$ admits Laurent series at infinity is elementary. First, it is divergence free and curl free (see \eqref{def Phi}). Moreover, the regularity inequality for harmonic functions $h$:
\begin{equation*}
\| h \|_{C^{1}(B(x,r))} \leq C \| h \|_{C^{0}(B(x,2r))},
\end{equation*}
proves that $\widehat{\nabla \Phi^1_{i}}(z) \rightarrow 0$ as $|z| \rightarrow +\infty$. 
The first term of this Laurent series, that is, the term of order $1/z$, is $(\int_{ \partial  \mathcal{S}_{0} } \widehat{\nabla \Phi^1_{i}} \, dz)/(2i\pi z)$, so we can compute it with Lemma~\ref{lem-flux-circu}. 
As $\nabla \Phi^1_{i}$ is a gradient, the circulation is zero, and Lemma~\ref{CalculIntegralesStandard} (case $d(P)=0$) proves that the flux is zero as well.
\end{proof}
\begin{Remark} \label{RemCircu}
One can prove in the same way, relying on \eqref{DefHeps}, that the function $\widehat{H^1}$ admits Laurent series at infinity and moreover satisfies:
\begin{equation*} 
\widehat{H}^{1}(z) = \frac{1}{2i\pi z} + {\mathcal O}(1/z^{2}) \ \text{ as } z \rightarrow \infty.
\end{equation*}
\end{Remark} 
The next coefficients in the Laurent series of $\widehat{\nabla \Phi^1_{i}}$  are given in the following lemma.
\begin{Lemma} \label{lemzphi}
One has:
\begin{eqnarray*}
\int_{ \partial  \mathcal{S}^\eps_{0}} z\widehat{\nabla \Phi_{i}^\eps}  \, dz &=&-(m_{i,2}^{\eps}+|\mathcal{S}^\eps_{0} |\delta_{i,2})+i(m_{i,1}^{\eps}+|\mathcal{S}^\eps_{0} |\delta_{i,1}),\quad \text{for }i=1,2, \\
\int_{ \partial  \mathcal{S}^\eps_{0}} z\widehat{ \nabla\Phi_{3}^\eps}   \, dz &=&-(m_{3,2}^{\eps}+|\mathcal{S}^\eps_{0} |x_{G,1}^\eps)+i(m_{3,1}^{\eps}-|\mathcal{S}^\eps_{0} |x_{G,2}^\eps), \\
\int_{ \partial  \mathcal{S}^\eps_{0}} z\widehat{ \nabla\Phi_{4}^\eps}   \, dz &=&-(m_{4,2}^{\eps}+|\mathcal{S}^\eps_{0} |x_{G,2}^\eps)+i(m_{4,1}^{\eps}-|\mathcal{S}^\eps_{0} |x_{G,1}^\eps), \\
\int_{ \partial  \mathcal{S}^\eps_{0}} z\widehat{ \nabla\Phi_{5}^\eps}   \, dz &=&-(m_{5,2}^{\eps}+|\mathcal{S}^\eps_{0} |x_{G,1}^\eps)+i(m_{5,1}^{\eps}+|\mathcal{S}^\eps_{0} |x_{G,2}^\eps),
\end{eqnarray*}
where $m_{i,j}^{\eps}=\int_{\mathcal{F}^{\eps}_0} \nabla \Phi^{\eps}_i \cdot \nabla \Phi^{\eps}_j \ dx$ are the coefficients defined in \eqref{DefMasse}.
\end{Lemma}
\begin{proof}
We use Corollary~\ref{cor} with $f=\nabla \Phi_{i}^\eps$:
\begin{eqnarray*}
\int_{ \partial  \mathcal{S}^\eps_{0}} z\widehat{ \nabla\Phi_{i}^\eps}   \, dz
&=& \int_{ \partial  \mathcal{S}^\eps_{0}} (x_{1} +i x_{2}) \partial_{\tau}\Phi_{i}^\eps  \, ds 
-i \int_{ \partial  \mathcal{S}^\eps_{0}} (x_{1} +i x_{2}) \partial_{n}\Phi_{i}^\eps  \, ds.
\end{eqnarray*}
We integrate the first integral by parts:
\begin{eqnarray*}
\int_{ \partial  \mathcal{S}^\eps_{0}} (x_{1} +i x_{2}) \partial_{\tau}\Phi_{i}^\eps  \, ds 
&=&  -\int_{ \partial  \mathcal{S}^\eps_{0}} \partial_{\tau}(x_{1} +i x_{2}) \Phi_{i}^\eps  \, ds
=-\int_{ \partial  \mathcal{S}^\eps_{0}}(\tau_{1} +i \tau_{2}) \Phi_{i}^\eps  \, ds
=-\int_{ \partial  \mathcal{S}^\eps_{0}}(n_{2} -i n_{1}) \Phi_{i}^\eps  \, ds\\
&=& -\int_{ \mathcal{F}_{0}^\eps}\nabla\Phi_{2}^\eps \cdot  \nabla  \Phi_{i}^\eps 
+i \int_{ \mathcal{F}_{0}^\eps}\nabla\Phi_{1}^\eps \cdot  \nabla  \Phi_{i}^\eps  = -m_{i,2}^{\eps} +i m_{i,1}^{\eps}
\end{eqnarray*}
The second integral can be computed thanks to the boundary condition in \eqref{def Phi} (with \eqref{DefK1-3} and \eqref{DefK4-5}) and to Lemma~\ref{CalculIntegralesStandard}. Gathering the expressions, we reach the result. 

\end{proof}
\begin{Remark}\label{remzphi}
Using Remark~\ref{Rembarz}, we can reproduce the same proof as previously to establish that:
\begin{eqnarray*}
\int_{ \partial  \mathcal{S}^\eps_{0}} \bar{z}\widehat{\nabla \Phi_{i}^\eps}  \, dz &=&(-m_{i,2}^{\eps}+|\mathcal{S}^\eps_{0} |\delta_{i,2})+i(-m_{i,1}^{\eps}+|\mathcal{S}^\eps_{0} |\delta_{i,1}),\quad \text{for }i=1,2, \\
\int_{ \partial  \mathcal{S}^\eps_{0}} \bar{z}\widehat{\nabla \Phi_{3}^\eps}   \, dz &=&(-m_{3,2}^{\eps}+|\mathcal{S}^\eps_{0} |x_{G,1}^\eps)-i(m_{3,1}^{\eps}+|\mathcal{S}^\eps_{0} |x_{G,2}^\eps), \\
\int_{ \partial  \mathcal{S}^\eps_{0}} \bar{z}\widehat{ \nabla\Phi_{4}^\eps}   \, dz &=&(-m_{4,2}^{\eps}+|\mathcal{S}^\eps_{0} |x_{G,2}^\eps)-i(m_{4,1}^{\eps}+|\mathcal{S}^\eps_{0} |x_{G,1}^\eps), \\
\int_{ \partial  \mathcal{S}^\eps_{0}} \bar{z}\widehat{ \nabla\Phi_{5}^\eps}   \, dz &=&(-m_{5,2}^{\eps}+|\mathcal{S}^\eps_{0} |x_{G,1}^\eps)+i(-m_{5,1}^{\eps}+|\mathcal{S}^\eps_{0} |x_{G,2}^\eps).
\end{eqnarray*}
\end{Remark}
\subsection{Second order moments}
In this subsection, we compute second order complex moments of $\widehat{\nabla \Phi^\eps_{i}}$. In particular we obtain the next terms in their Laurent series. \par
\begin{Lemma}\label{lem|z|2phi}
One has:
\begin{gather*}
\int_{\partial \mathcal{S}^\eps_{0}} |z|^2 \widehat{\nabla \Phi_{i}^\eps} \, dz = -2m_{i,3}^{\eps}
+ 2i |\mathcal{S}^\eps_{0}| x_{G,i}^\eps \quad \text{for }i=1,2, \quad
\int_{ \partial  \mathcal{S}^\eps_{0}} |z|^2\widehat{\nabla \Phi_{3}^\eps}   \, dz =-2m_{3,3}^{\eps}, \\
\int_{ \partial  \mathcal{S}^\eps_{0}} |z|^2\widehat{\nabla \Phi_{4}^\eps}   \, dz =-2m_{4,3}^{\eps}-2m_6^{\eps}i,  \quad
\int_{ \partial  \mathcal{S}^\eps_{0}} |z|^2\widehat{\nabla \Phi_{5}^\eps}   \, dz =-2m_{5,3}^{\eps}+2m_{7}^{\eps}i.
\end{gather*}
\end{Lemma}
\begin{proof}
Applying Lemma~\ref{lem-flux-circu} to $(|z|^2\partial_1 \Phi_{i}^\eps,|z|^2\partial_2 \Phi_{i}^\eps)$ we have
\[
\int_{ \partial  \mathcal{S}^\eps_{0}} |z|^2(\partial_1 \Phi_{i}^\eps - i \partial_2 \Phi_{i}^\eps)  \, dz = 
\int_{ \partial  \mathcal{S}^\eps_{0}} |x|^2\partial_{\tau} \Phi_{i}^\eps  \, ds
-i\int_{ \partial  \mathcal{S}^\eps_{0}} |x|^2\partial_{n} \Phi_{i}^\eps  \, ds.
\]
We easily verify that
\begin{eqnarray*}
\int_{ \partial  \mathcal{S}^\eps_{0}} |x|^2\partial_{\tau} \Phi_{i}^\eps  \, ds 
= -\int_{ \partial  \mathcal{S}^\eps_{0}}  \Phi_{i}^\eps 2x\cdot\tau  \, ds
= -\int_{ \partial  \mathcal{S}^\eps_{0}}  \Phi_{i}^\eps 2(x^\perp \cdot n)  \, ds = -2m_{i,3}^{\eps}.
\end{eqnarray*}
The value of $\int_{ \partial  \mathcal{S}^\eps_{0}} |x|^2\partial_{n} \Phi_{i}^\eps  \, ds$ is computed in Lemma~\ref{CalculIntegralesStandard}.
\end{proof}
As for Corollary~\ref{cor}, we deduce from Lemma~\ref{lem-flux-circu} the following result.
\begin{Lemma} \label{lemz2}
Let $\mathcal{C}$ be a smooth Jordan curve, $f:=(f_1 , f_2)$ a smooth vector fields on $\mathcal{C}$:
\begin{equation*} 
\int_{ \mathcal{C}} z^2\widehat f  \, dz 
= \int_{ \mathcal{C}} (x_{1} +i x_{2})^2 (f\cdot \tau)  \, ds -i \int_{ \mathcal{C}} (x_{1} -i x_{2})^2 (f\cdot n)  \, ds.
\end{equation*}
\end{Lemma}
\begin{proof}
We apply Lemma~\ref{lem-flux-circu} to the function 
\begin{equation*}
g=g_{1}-ig_{2} = z^2(f_1 -if_2) =[(x_1^2-x_2^2)f_1+2x_1x_2f_2] - i[(x_1^2-x_2^2)f_2-2x_1x_2 f_1],
\end{equation*}
hence $g=(x_1^2-x_2^2)f- 2x_1x_2 f^\perp$. We get
\begin{equation*}
\begin{split}
\int_{ \mathcal{C}} z^2\widehat f  \, dz &= \int_{ \mathcal{C}} (x_{1}^2 - x_{2}^2) (f\cdot \tau)  \, ds -\int_{ \mathcal{C}} 2 x_{1}x_{2} (f^\perp\cdot \tau)  \, ds -i\int_{ \mathcal{C}} (x_{1}^2 - x_{2}^2) (f\cdot n)  \, ds +i \int_{ \mathcal{C}} 2 x_{1}x_{2} (f^\perp\cdot n)  \, ds\\
&= \int_{ \mathcal{C}} (x_{1}^2 - x_{2}^2) (f\cdot \tau)  \, ds + \int_{ \mathcal{C}} 2 x_{1}x_{2} (f\cdot n)  \, ds -i\int_{ \mathcal{C}} (x_{1}^2 - x_{2}^2) (f\cdot n)  \, ds +i \int_{ \mathcal{C}} 2 x_{1}x_{2} (f\cdot \tau)  \, ds.
\end{split}
\end{equation*}
\end{proof}
Finally, we apply this lemma to compute $\int_{ \partial  \mathcal{S}^\eps_{0}} z^2\widehat{\nabla \Phi_i^\eps}  \, dz$:
\begin{Lemma} \label{lemz2phi}
One has:
\begin{eqnarray*}
\int_{ \partial  \mathcal{S}^\eps_{0}} z^2\widehat{\nabla \Phi_{1}^\eps}   \, dz 
=-2(m_{1,5}^{\eps} +  |\mathcal{S}^\eps_{0} |x_{G,2}^\eps )+2i(-m_{1,4}^{\eps} +  |\mathcal{S}^\eps_{0} |x_{G,1}^\eps ), &&
\int_{ \partial  \mathcal{S}^\eps_{0}} z^2\widehat{\nabla \Phi_{4}^\eps}   \, dz 
=-2m_{4,5}^{\eps}-2i( m_{4,4}^{\eps}+m_{8}^{\eps}), \\
\int_{ \partial  \mathcal{S}^\eps_{0}} z^2\widehat{\nabla \Phi_{2}^\eps}   \, dz
=-2(m_{2,5}^{\eps} +  |\mathcal{S}^\eps_{0} |x_{G,1}^\eps )-2i(m_{2,4}^{\eps} +  |\mathcal{S}^\eps_{0} |x_{G,2}^\eps ), &&
\int_{ \partial  \mathcal{S}^\eps_{0}} z^2\widehat{\nabla \Phi_{5}^\eps}   \, dz 
=-2( m_{5,5}^{\eps}+ m_{8}^{\eps})-2i m_{5,4}^{\eps}, \\
\int_{ \partial  \mathcal{S}^\eps_{0}} z^2\widehat{\nabla \Phi_{3}^\eps}   \, dz
=-2(m_{3,5}^{\eps} +m_{6}^{\eps}  )-2i (m_{3,4}^{\eps} +m_{7}^{\eps}  ).
\end{eqnarray*}
\end{Lemma}
\begin{proof}
We use Lemma~\ref{lemz2}. For the tangential part, we compute:
\begin{equation*}\begin{split}
\int_{ \partial  \mathcal{S}^\eps_{0}} (x_{1}^2-x_{2}^2)\partial_{\tau} \Phi_{i}^\eps  \, ds 
&= -2\int_{ \partial  \mathcal{S}^\eps_{0}}  \Phi_{i}^\eps \begin{pmatrix} x_{1}\\ -x_{2}\end{pmatrix}\cdot\tau  \, ds
= -2\int_{ \partial  \mathcal{S}^\eps_{0}}  \Phi_{i}^\eps \begin{pmatrix} x_{2}\\ x_{1}\end{pmatrix} \cdot n  \, ds=-2m_{i,5}^{\eps},\\
2\int_{ \partial  \mathcal{S}^\eps_{0}} x_{1}x_{2}\partial_{\tau} \Phi_{i}^\eps  \, ds 
&= -2\int_{ \partial  \mathcal{S}^\eps_{0}}  \Phi_{i}^\eps \begin{pmatrix} x_{2}\\ x_{1}\end{pmatrix}\cdot\tau  \, ds
= -2\int_{ \partial  \mathcal{S}^\eps_{0}}  \Phi_{i}^\eps \begin{pmatrix} -x_{1}\\ x_{2}\end{pmatrix} \cdot n  \, ds=-2m_{i,4}^{\eps}.
\end{split}\end{equation*}
For the normal part, we use the boundary condition on $\partial_{n} \Phi_{i}^\eps$ and Lemma~\ref{CalculIntegralesStandard}.
Putting these computations in the relation of Lemma~\ref{lemz2} gives the result.
\end{proof}
\ \par
\noindent
{\bf Acknowledgements.}  The authors are partially supported by the Agence Nationale de la Recherche, Project DYFICOLTI, grant ANR-13-BS01-0003-01.  C.L. and F.S. are partially supported by the project \emph{Instabilities in Hydrodynamics} funded by the Paris city hall (program \emph{Emergences}) and the \emph{Fondation Sciences Math\'ematiques de Paris}.
%
%
%
%
%
%
%
%
%

\end{document}